\documentclass[english, 3p]{elsarticle}

\pdfoutput=1 

\usepackage[latin1]{inputenc}
\usepackage[T1]{fontenc}
\usepackage{babel,geometry,lmodern}
\usepackage{amsmath}
\usepackage{amsfonts}
\usepackage{amssymb}
\usepackage{graphicx}
\usepackage[caption=false]{subfig}
\usepackage{color}
\usepackage[hang,flushmargin]{footmisc} 
\usepackage{tikz}
\usepackage{multicol}
\usepackage{multirow}
\usepackage{framed}
\usepackage{stmaryrd} 
\usepackage{enumitem}

\usepackage[refpage]{nomencl}

\numberwithin{equation}{section}

\usepackage{amsthm}
\newtheorem{theorem}{Theorem}
\newtheorem{lemma}{Lemma}

\newcommand{\mcFh}{\mathcal{F}_h}
\newcommand{\mcEh}{\mathcal{E}_h}
\newcommand{\mcF}[1]{\mathcal{F}_{#1}}
\newcommand{\Dh}{D_h}
\newcommand{\Dhinv}{D_h^{\dagger}}
\newcommand{\bfn}{\mathbf{n}}
\newcommand{\bfu}{\mathbf{u}}
\newcommand{\bfnFE}{\mathbf{n}_F \cdot \mathbf{n}_E}
\newcommand{\pdiv}{P^0_{0,\text{div}}(\mcFh)}
\newcommand{\pdivperp}{P_{0,\text{div}}^{0,\perp}(\mcFh)}
\newcommand{\jump}[1]{\llbracket #1 \rrbracket}
\newcommand{\average}[1]{\{\hspace{-3pt}\{ #1 \}\hspace{-3pt}\}}

\makeatletter
\def\ps@pprintTitle{%
 \let\@oddhead\@empty
 \let\@evenhead\@empty
 \def\@oddfoot{}%
 \let\@evenfoot\@oddfoot}
\makeatother

\begin{document}

\title{Postprocessing of Non-Conservative Flux for Compatibility with Transport in Heterogeneous Media}

\author[ntnu]{Lars H.~Ods\ae ter\corref{cor}}
\ead{lars.odsater@math.ntnu.no}

\author[mfw]{Mary F.~Wheeler}
\author[ntnu]{Trond Kvamsdal}
\author[mgl]{Mats G.~Larson}

\address[ntnu]{Department of Mathematical Sciences, NTNU Norwegian University of Science and Technology,\\ Alfred Getz' vei 1, 7491 Trondheim, Norway}
\address[mfw]{The Institute for Computational Engineering and Sciences, The University of Texas at Austin, Austin, TX 78712, USA}
\address[mgl]{Department of Mathematics and Mathematical Statistics, Umeå University, SE-901 87 Umeå, Sweden}

\cortext[cor]{Corresponding author}

\begin{abstract}
A conservative flux postprocessing algorithm is presented for both steady-state and dynamic flow models. The postprocessed flux is shown to have the same convergence order as the original flux. An arbitrary flux approximation is projected into a conservative subspace by adding a piecewise constant correction that is  minimized in a weighted $L^2$ norm. The application of a weighted norm  appears to yield better results for heterogeneous media than the standard $L^2$ norm which has been considered in earlier works. We also study the effect of different flux calculations on the domain boundary. In particular  we consider the continuous Galerkin finite element method for solving Darcy flow  and couple it with a discontinuous Galerkin finite element method for an advective transport problem.
\end{abstract}


\maketitle

\section{Introduction}
In this paper we consider the following coupled flow and transport problem that arise in porous media:
\begin{align}
\partial_t (\beta p) - \nabla\cdot(\mathbf{K} \nabla p) &= q, \label{eq:flow_intro} \\
\partial_t(\phi c) + \nabla\cdot(c\mathbf{u} - \mathbf{D}\nabla c) &= f. \label{eq:transport_intro}
\end{align}
Equation \eqref{eq:flow_intro}, often referred to as the Darcy flow equation, governs conservation of mass for a slightly compressible single-phase fluid in a porous media. Here $p$ represents pressure and $\mathbf{u}=-\mathbf{K}\nabla p$ the Darcy velocity. The second equation \eqref{eq:transport_intro} is known as the transport equation, and describes advective and diffusive transport of a concentration $c$. Such transport models are employed in modeling tracers in a porous media \cite{mishra1991tracer}. 
Choosing compatible numerical solvers for the flow and transport equations may be of importance for accuracy, stability and conservation properties \cite{dawson2004compatible}.
Here we discuss using a continuous Galerkin (CG) finite element method for the flow equation and apply a postprocessing method to compute fluxes on element boundaries to obtain local conservation. A discontinuous Galerkin (DG) finite element method with upwinding is employed for the transport equation \cite{wheeler1978elliptic,sun2005symmetric}. DG allows for discontinuities in the solution and has the advantages of local mass conservation, less numerical diffusion, favorable h- and p-refinement, handling of discontinuous coefficients, and efficient implementation. 

CG is a well-developed numerical discretization for partial differential equations. It is numerically efficient for problems requiring dynamic grid adaptivity. It is known that CG requires postprocessing to obtain locally conservative fluxes on element boundaries
\cite{ainsworth2000posteriori,kvamsdal1998variationally,hughes2000continuous,melbo2003goal,larson2004conservative,sun2006projections,cockburn2007locally,bush2013application,odsater2015apt,deng2016construction,becker2016local}.
This has been the topic also for studies of environmental modeling in bays and estuaries where CG has been employed for shallow water equations \cite{chippada1998projection}. 
Applying non-conservative flux to the transport equation may result in non-physical concentration solutions \cite{sun2009locally,lee2016locally,odsater2015apt}. 

Computing fluxes for CG models has been considered in many technical papers; we briefly describe some well known results and note that the list is incomplete. Optimal postprocessing of fluxes on element boundaries for one-dimensional problems was studied by Wheeler \cite{wheeler1974galerkin} and generalized by Dupont \cite{dupont1976unified}. Douglas et al.\ \cite{douglas1974galerkin} analyzed methods for approximating flux on the domain boundary for multi-dimensional problems based on the approach of J.\ Wheeler \cite{wheeler1973soh}. Postprocessing of locally conservative (or self-equilibrated) fluxes on element boundaries for multi-dimensional problems was studied by Ladeveze and Leguillon \cite{ladeveze1983error} for error estimation purposes. Ainsworth and Oden \cite{ainsworth2000posteriori} proved the existence of such self-equilbrated fluxes for general CG methods including 1-irregular meshes with hanging nodes. Superconvergence of recovered gradients of linear CG approximations for elliptic and parabolic problems was treated by Wheeler and Whiteman \cite{wheeler1987superconvergent,wheeler1994superconvergence}. 

For completeness we mention alternative schemes to CG for the pressure equation; mixed finite element methods \cite{wheeler2006multipoint}, dual-grid and control volume methods \cite{aavatsmark2002introduction}, finite volume methods \cite{chatzipantelidis2005finite}, mimetic finite difference methods \cite{brezzi2006convergence}, and DG \cite{riviere1999improved}. 
All of these are conservative in the sense that they either are formulated in a mixed form so that locally conservative fluxes are obtained directly without the need for \emph{any} postprocessing, or have an embedded local conservation statement in their derivation so that locally conservative fluxes can be calculated in a straightforward manner from the pressure solution.
Recent papers \cite{bush2013application,deng2016construction} have observed that  CG with postprocessing on the dual grid is more robust than standard control volume approaches. Here the postprocessing involves only local calculations. It is well known that for Laplace's equation, control volume and CG on the dual grid are equivalent. Lack or complexity of dynamic grid adaptivity is a disadvantage for many of the methods mentioned above. DG is promising both with respect to local conservation and dynamic grid adaptivity, but is computationally costly due to a high number of degrees of freedom. A conservative scheme based on enrichment of CG was proposed by \cite{sun2009locally} for elliptic problems and later extended to parabolic equations in \cite{lee2016locally}. 

The postprocessing method we propose in this paper is built upon the work of Sun and Wheeler \cite{sun2006projections} and Larson and Niklasson \cite{larson2004conservative} for the steady-state flow model (Eq.~\eqref{eq:flow_intro} with $\beta=0$). Both of these papers present an algorithm for computing conservative fluxes on element boundaries. Here a given general non-conservative flux approximation is modified by adding piecewise constant corrections which are minimized in a given norm. The minimization requirement ensures that the postprocessed flux has the same order of convergence as the original flux. The works by \cite{sun2006projections} and \cite{larson2004conservative} have strong similarities and are in fact identical under some specific choice of parameters, but have been formulated differently. While a variational formulation is used in \cite{larson2004conservative}, the method is presented elementwise in \cite{sun2006projections}. In this paper we present both and demonstrate the relationship between the two results.
We mention that these postprocessing methods have been applied in a series of recent works \cite{kees2008lcs,povich2013fem,beirao2015ppo,scudeler2016mass}.

The main novelties of our work compared to \cite{sun2006projections} and \cite{larson2004conservative} are summarized below.
\begin{itemize}[noitemsep]
\item The correction term is minimized in a weighted $L^2$ norm instead of the standard $L^2$ norm. This gives control of which faces should be weighted most. Our choice of weights corresponds to the inverse of the effective face permeability. This is shown to better preserve low permeable interfaces.
\item Our method applies to a wide range of grids, including non-conforming and unstructured grids, in contrast to \cite{sun2006projections}.
\item The method is applied to the time dependent flow model (Eq.~\ref{eq:flow_intro} with $\beta\neq 0$).
\item We solve the coupled problem \eqref{eq:flow_intro}-\eqref{eq:transport_intro} to demonstrate the importance of locally conservative flux and to illustrate the effect of some parameters of our postprocessing method.
\end{itemize}

The presented method is general in the sense that it takes as input any flux approximation, not restricted to non-conservative flux from classical CG, but may also originate from other numerical schemes, e.g.\ isogeometric finite elements \cite{bekele2015adaptive}, or even measurements. We remark that our method only produce conservative fluxes on element boundaries. To extend the flux to a velocity field on the element interiors one can set up a localized mixed finite element problem on each element, see \cite{sun2006projections}. 
We also point out that minimizing in a weighted norm was considered in \cite{wu2008global} in an upscaling framework. However, our presentation includes error analysis, and we also study the impact of weighting on the transport problem. An alternative approach to preserve low permeable interfaces is to add a penalization step to correct the postprocessed flux \cite{schiavazzi2013redundant}. 

This paper is outlined as follows. Section~\ref{sec:preliminaries} provides some preliminaries, including the model equations, notational comments, conservation conditions, and discretization schemes for CG and DG. Next, in Section~\ref{sec:postprocessing}, we go into details of the postprocessing method, first for the time independent case and later extended to the general case. We formulate our approach based on a discrete divergence operator and its left inverse. Furthermore, we prove an error estimate and discuss some parameters of our method.
In Section~\ref{sec:examples} we demonstrate our method with some numerical examples. Finally, in Section~\ref{sec:conclusion}, we conclude this work.

\section{Preliminaries}
\label{sec:preliminaries}

\subsection{Model Equations}

We consider a coupled flow and transport problem in a bounded domain $\Omega\subset \mathbb{R}^d$ ($d=2,3$) and in the time interval $[0,T],\,T>0$. 
\nomenclature{$\Omega$}{Model domain}%
\nomenclature{$T$}{End time}%

\paragraph{Flow Equation}
For flow, we consider the linear parabolic problem
\begin{align}
\partial_t (\beta p) - \nabla\cdot \left(\mathbf{K}\nabla p\right) = q, \quad (\mathbf{x},t)\in \Omega \times (0,T].
\label{eq:flow_eq}
\end{align}
\nomenclature{$\beta$}{Time coefficient}%
\nomenclature{$p$}{Fluid pressure}%
\nomenclature{$\mathbf{K}$}{Conductivity tensor}%
\nomenclature{$q$}{Source term}%
The unknown variable is the pressure $p$, from which the velocity $\bfu$ is defined by $\mathbf{u} = -\mathbf{K}\nabla p$. The conductivity $\mathbf{K}=\mathbf{K}(\mathbf{x})$ is the ratio between permeability and viscosity, and $\mathbf{K}$ is assumed to be symmetric positive definite and bounded from below and above. Furthermore, $\beta=\beta(\mathbf{x},t)$ is a positive coefficient and $q=q(\mathbf{x},t)$ is a source term. In the case $\beta = 0$, the flow equation is elliptic and stationary.  Throughout this paper we let $\mu=1$ for simplicity and will use the terms conductivity and permeability interchangeably.

The domain boundary $\partial \Omega$ is divided into a Dirichlet part, $\Gamma_D$, and a Neumann part, $\Gamma_N$, such that $\overline{\Gamma}_D \cup \overline{\Gamma}_N = \overline{\partial \Omega}$ and $\Gamma_D \cap \Gamma_N = \emptyset$. The boundary and initial  conditions are
\begin{subequations}
\begin{align}
p &= p_B, \quad (\mathbf{x},t)\in\Gamma_D \times (0,T], \\
\mathbf{u}\cdot\mathbf{n} = -\mathbf{K}\nabla p\cdot \mathbf{n} &= u_B, \quad (\mathbf{x},t)\in\Gamma_N \times (0,T], \\
p &= p_0, \quad\ (\mathbf{x},t)\in\Omega\times \{0\},
\end{align}
\end{subequations}
where $\mathbf{n}$ is the outward unit normal vector on $\partial \Omega$ and $p_B=p_B(\mathbf{x},t)$, $u_B=u_B(\mathbf{x},t)$ and $p_0=p_0(\mathbf{x})$ are known functions. 
\nomenclature{$\partial \Omega$}{Domain boundary}%
\nomenclature{$\Gamma_D$}{Dirichlet boundary}%
\nomenclature{$\Gamma_N$}{Neumann boundary}%
\nomenclature{$\mathbf{n}$}{Unit normal vector}%

\paragraph{Transport Equation}
The model equation for transport is the time dependent advection-diffusion equation,
\begin{align}
\partial_t(\phi c) + \nabla\cdot(\mathbf{u}c - \mathbf{D}\nabla c) = f, \quad (\mathbf{x},t)\in \Omega \times (0,T].
\label{eq:transport_eq}
\end{align}
\nomenclature{$\phi$}{Porosity}%
\nomenclature{$c$}{Concentration}%
\nomenclature{$\mathbf{D}$}{Diffusion tensor}%
\nomenclature{$f$}{Source term for transport equation}%
The unknown variable is the concentration $c$. Furthermore, $\phi=\phi(\mathbf{x})$ is the porosity (fraction of void volume) and $\mathbf{D}=\mathbf{D}(\mathbf{x},c)$ is the diffusion/dispersion tensor. The right hand side $f=f(\mathbf{x},t)$ is a source term, and when coupled with the flow equation \eqref{eq:flow_eq}, it is usually interpreted as $f=qc^*$, where $c^*$ denotes the upstream concentration, so that
\begin{align}
qc^* =
\begin{cases}
qc, & \text{if } q \le 0, \\
qc_w, & \text{if } q > 0,
\end{cases}
\end{align}
\nomenclature{$c^*$}{Upwind concentration}%
where $c_w=c_w(\mathbf{x},t)$ denotes the source (well) concentration. 

The boundary is divided into a inflow boundary,
$\Gamma_{\text{in}}=\{\mathbf{x}\in\partial\Omega : \mathbf{u}\cdot\mathbf{n}<0\}$,
and a outflow/no-flow boundary,
$\Gamma_{\text{out}}=\{\mathbf{x}\in\partial\Omega : \mathbf{u}\cdot\mathbf{n}\ge0\}$. Let $c_B=c_B(\mathbf{x},t)$ denote the inflow concentration on $\Gamma_{\text{in}}$ and $c_0=c_0(\mathbf{x})$ the initial concentration. The boundary and initial conditions are given as
\begin{subequations}
\begin{align}
(\mathbf{u}c - \mathbf{D}\nabla c)\cdot\mathbf{n} &= c_B \mathbf{u}\cdot\mathbf{n}, \ \quad (\mathbf{x},t)\in \Gamma_{\text{in}} \times (0,T], \\
(- \mathbf{D}\nabla c)\cdot\mathbf{n} &= 0, \qquad\qquad (\mathbf{x},t)\in \Gamma_{\text{out}} \times (0,T], \\
c &= c_0, \qquad\quad\ \ (\mathbf{x},t)\in\Omega\times\{0\}.
\end{align}
\end{subequations}
\nomenclature{$\Gamma_{\text{in}}$}{Inflow boundary}%
\nomenclature{$\Gamma_{\text{out}}$}{Outflow boundary}%
In this work, we will focus on advection dominated flow and disregard diffusion by setting $\mathbf{D}= 0$.

\subsection{Notation}

\paragraph{Discretization of the Domain}
Let $\mathcal{E}_h$ be a partition of $\Omega$ into triangles or quadrilaterals ($d=2$), or tetrahedra, prisms or hexahedra ($d=3$). We denote by $E_i\in\mathcal{E}_h$, for $i=1,2,\ldots, N$, the $N$ elements of the partition, and let $h_i$ be the diameter of $E_i$. We assume $\mathcal{E}_h$ to be regular in the sense that all elements are convex and that there exists $\rho>0$ such that each element $E_i$ contains a ball of radius $\rho h_i$ in its interior. Furthermore, $\mathcal{E}_h$ should be quasi-uniform, i.e., there is a $\tau>0$ such that $\frac{h}{h_i}\le\tau$ for all $E_i\in\mathcal{E}_h$, where $h$ is the maximum diameter of all elements. Notice that we allow for elements of mixed type and non-matching grids (hanging nodes).
\nomenclature{$\mathcal{E}_h$}{Partition of $\Omega$ into elements}%
\nomenclature{$E$}{An element (or grid cell) in $\mathcal{E}_h$}%
\nomenclature{$h$}{Maximum element diameter}%

We denote by $\mcF{h,I}$ the set of all interior edges ($d=2$) or faces ($d=3$), i.e.,
\begin{align}
\mcF{h,I} = \{F\in\mathbb{R}^{d-1} : F = E_i\cap E_j, E_i \in \mathcal{E}_h,E_j \in \mathcal{E}_h, E_i\neq E_j \}.
\end{align}
For simplicity we only use the term face in the following. Furthermore, we define $\mcF{h,B}$ as the set of all element faces that intersect with $\partial \Omega$. We assume that each face in $\mcF{h,B}$ is either completely on the Dirichlet or Neumann part of the boundary, such that $\mcF{h,B}$ can be decomposed into $\mcF{h,D}$ and $\mcF{h,N}$, i.e., the sets of faces on the Dirichlet and Neumann boundary, respectively. Analogously, let $\mcF{h,\text{out}}$ and $\mcF{h,\text{in}}$ be the sets of faces on $\mcF{\text{out}}$ and $\mcF{\text{in}}$, respectively.
Next, let $\mcFh = \mcF{h,I} \cup \mcF{h,B}$. For each face $F\in\mcF{h}$ we choose a unit normal vector $\mathbf{n}_{F}$\footnote{This can be done by choosing $\mathbf{n}_{F}$ to coincide with the outward unit normal of the element with lowest element number.}. The unit normal vector on $F\in\mcF{h,B}$ is chosen to coincide with the outward unit normal vector. Furthermore, $\bfn_E$ denotes the unit normal vector pointing out of $E$, such that $\bfn_E\vert_F=\pm \bfn_F$.
\nomenclature{$\mcFh$}{Set of all element faces}%
\nomenclature{$\mcF{h,I}$}{Set of all interior faces}%
\nomenclature{$\mcF{h,B}$}{Set of all boundary faces}%
\nomenclature{$\mcF{h,D}$}{Set of all Dirichlet faces}%
\nomenclature{$\mcF{h,N}$}{Set of all Neumann faces}%
\nomenclature{$\mcF{h,\text{in}}$}{Set of all faces on inflow boundary}%
\nomenclature{$\mcF{h,\text{out}}$}{Set of all faces on outflow (out) boundary}%
\nomenclature{$F$}{Edge in $\mcFh$}%

\paragraph{Piecewise Polynomial Spaces}
Let $P_r(\mathcal{E}_h)$ be the space of piecewise polynomial functions of degree $r$,
\begin{align}
P_r(\mathcal{E}_h) = \{ \varphi\in L^2(\Omega) : \varphi\vert_E \in Q_r(E), E\in \mathcal{E}_h \},
\end{align}
where $Q_r$ denotes the tensor product of polynomial spaces of degree less than or equal to $r$ in each spatial direction\footnote{To be rigorous, $Q_r$ is the space of functions such that when mapped to the reference element are polynomials of degree $r$.}.
We also need the continuous subspace of $P_r(\mathcal{E}_h)$,
\begin{align}
P^C_r(\mathcal{E}_h) = P_r(\mathcal{E}_h) \cap C(\Omega).
\end{align}
Furthermore, we define the space of piecewise polynomial functions on element faces as
\begin{align}
P_r(\mcFh) = \{ \varphi\in L^2(\mcFh) : \varphi\vert_{F} \in \mathcal{Q}_r(F), F\in \mcFh \}.
\end{align}
Moreover, let $P^0_r(\mcFh)$ denote the subspace of $P_r(\mcFh)$ whose functions are zero on the Neumann boundary,
\begin{align}
P^0_r(\mcFh) = \{ \varphi\in P_r(\mcFh) : \phi\vert_{F} = 0, F\in\mcF{h,N} \}.
\end{align}
\nomenclature{$P_r(\mathcal{E}_h)$}{Space of piecewise polynomial functions of degree $r$}%
\nomenclature{$P_r^C(\mathcal{E}_h)$}{Space of continuous and piecewise polynomial functions of degree $r$}%
\nomenclature{$P_r(\mcFh)$}{Space of piecewise polynomial functions of degree $r$ over faces}%
\nomenclature{$P_r^0(\mcFh)$}{Space of piecewise polynomial functions of degree $r$ that vanish on the Neumann boundary}%

\paragraph{Inner Products and Norms}
We denote by $(\cdot,\cdot)_S$ the standard $L^2$ inner product over a domain $S\in\mathbb{R}^d$, or $\langle \cdot,\cdot\rangle_S$ if $S\in\mathbb{R}^{d-1}$. The standard $L^2$ norm over $S$ is denoted $\Vert\cdot\Vert_S$. If $S=\Omega$, we write $(\cdot,\cdot)$ or $\Vert\cdot\Vert$ for simplicity. Furthermore, define the broken inner products and norms
\begin{align}
(v,w)_{\mcEh} = \sum_{E\in\mcEh}(v,w)_E, \qquad &\Vert v\Vert_{\mcEh}^2 
= (v,v)_{\mcEh}
= \sum_{E\in\mcEh} \Vert v\Vert_E^2, \\
\langle v,w \rangle_{\mcFh} 
= \sum_{F\in\mcFh}\langle v,w\rangle_F, \qquad
&\Vert v\Vert_{\mcFh}^2 
= \langle v,v \rangle_{\mcFh}
= \sum_{F\in\mcFh} \Vert v\Vert_F^2.
\end{align}
\nomenclature{$(\cdot,\cdot)_S$}{Standard $L^2$ inner product over domain $S\in\mathbb{R}^d$}%
\nomenclature{$\langle\cdot,\cdot\rangle_S$}{Standard $L^2$ inner product over domain $S\in\mathbb{R}^{d-1}$}%

The measure of a domain $S$ is denoted $\vert S\vert$. In particular this means that $\vert F\vert$ is the length ($d=2$) or area ($d=3$) of a face $F\in\mcFh$, while $\vert E\vert$ is the area ($d=2$) or volume ($d=3$) of an element $E\in\mathcal{E}_h$.
\nomenclature{$\vert \cdot\vert$}{Measure (length, area or volume) of domain}

\paragraph{Average and Jump Operators}
Next, for $s>0$, define
\begin{align}
H^s(\mathcal{E}_h) = \left\{ \varphi \in L^2(\Omega) : \varphi\vert_E\in H^s(E), E\in\mathcal{E}_h\right\}.
\end{align}
Now, let $E_i,E_j\in\mathcal{E}_h$ and $F = \partial E_i\cap\partial E_j\in\mcF{h,I}$ with $\mathbf{n}_{F}$ exterior to $E_i$. Then, for $\mathbf{v}\in \left(H^s(\mathcal{E}_h)\right)^d,\ s>\frac{1}{2}$, we define the average over $F$ as
\begin{align}
\average{\mathbf{v}}_{\theta} &= \theta_F (\mathbf{v}\vert_{E_i})\big\vert_F + (1-\theta_F) (\mathbf{v}\vert_{E_j})\big\vert_F,
\label{eq:average}
\end{align}
where $\theta$ is a given weight with $\theta_F=\theta\vert_F$ and $0<\theta_F<1$. For the standard average $\theta=\frac{1}{2}$, we simply write $\average{\mathbf{v}}$. 
In this work we consider weights $\vartheta$ that depend on $\mathbf{K}$,
\begin{align}
\vartheta_F = \frac{\delta_{Kn}^j}{\delta_{Kn}^i + \delta_{Kn}^j}, \qquad \delta_{Kn}^i = \bfn_F^{\top}\mathbf{K}_i\bfn_F,
\label{eq:harmonic_weights}
\end{align}
where $\delta_{Kn}^i$ is the normal component of $\mathbf{K}$ across $F$ and $\mathbf{K}_i$ is the permeability in $E_i$. This choice of weights was considered by \cite{burman2006domain} for the isotropic case, and later extended to the anisotropic case in \cite{ern2008discontinuous}. Now
\begin{align}
k_e = 2\vartheta_F \delta_{Kn}^i = 2 (1-\vartheta_F) \delta_{Kn}^j = 
2 \frac{\delta_{Kn}^i\delta_{Kn}^j}{\delta_{Kn}^i+\delta_{Kn}^j}
\end{align}
is the harmonic average of the normal component of $\mathbf{K}$ along $F$. Observe that for isotropic permeability, $\mathbf{K}=k\mathbb{I}$, where $\mathbb{I}$ is the identity matrix and $k$ is the directional independent permeability, we have that
\begin{align}
\vartheta_F = \frac{k_j}{k_i+k_j}, \qquad k_e = \frac{2k_ik_j}{k_i+k_j},
\end{align}
and it follows that
\begin{align}
\average{\mathbf{K}\nabla p}_{\vartheta} 
= \frac{k_j}{k_i + k_j}k_i \left((\nabla p)\vert_{E_i}\right)\big\vert_F 
+ \frac{k_i}{k_i + k_j}k_j \left((\nabla p)\vert_{E_j}\right)\big\vert_F 
= k_e\average{\nabla p}.
\end{align}

Next, for $v\in H^s(\mathcal{E}_h),\ s>\frac{1}{2}$, we define the jump over $F$ as
\begin{align}
\jump{v} &= (v\vert_{E_i})\big\vert_F - (v\vert_{E_j})\big\vert_F 
= (v\vert_{E_i})\big\vert_F\bfn_{E_i}\cdot\bfn_F + (v\vert_{E_j})\big\vert_F \bfn_{E_j}\cdot\bfn_F.
\end{align}
\nomenclature{$\{\cdot\}$}{Average operator}%
\nomenclature{$[\cdot]$}{Jump operator}%
For completeness, we extend the average and jump to $F\in\mcF{h,B}$, $F\subset\partial E_i$, by
\begin{align}
\average{\mathbf{v}}_{\theta} &= (\mathbf{v}\vert_{E_i})\big\vert_F, \\
\jump{v} &= (v\vert_{E_i})\big\vert_F.
\end{align}

\subsection{Conservation Properties}

\paragraph{Compatibility Condition}
Consider first the case $\beta=0$. If we multiply Eq.\ \eqref{eq:flow_eq} by a test function $\varphi$, and then integrate and sum the result over each element $E\in\mathcal{E}_h$, we get that
\begin{align}
(\mathbf{u}, \nabla \varphi)_{\mcEh} + 
\langle \mathbf{u}\cdot\mathbf{n}, \jump{\varphi}\rangle_{\mcF{h}} = (q,\varphi)_{\mcEh}.
\end{align}

Let $\mathbf{u}_h$ and $U_h$ be approximations to $\mathbf{u}$ in $\mcEh$ and $\mathbf{u}\cdot\mathbf{n}$ on $\mcFh$, respectively. Furthermore, let the space of test functions be $P_r(\mathcal{E}_h)$. The $r$th order compatibility condition for the velocity approximation reads
\begin{align}
(\mathbf{u}_h, \nabla \varphi)_{\mcEh} + 
\langle U_h, \jump{\varphi}\rangle_{\mcFh} =
(q,\varphi)_{\mcEh}, \quad \forall \varphi\in P_r(\mathcal{E}_h).
\label{eq:compatibility}
\end{align}
\nomenclature{$\mathbf{u}_h$}{Approximation to velocity $\mathbf{u}$}%
\nomenclature{$U_h$}{Approximation to flux $\mathbf{u}\cdot\mathbf{n}$}%

\paragraph{Local Conservation}
$U_h\in L^1(\mcFh)$ is locally conservative if it is $0$th order compatible, i.e., 
\begin{align}
\langle U_h, \jump{\varphi}\rangle_{\mcFh} =
(q,\varphi)_{\mcEh}, \quad \forall \varphi\in P_0(\mathcal{E}_h),
\label{eq:localcons-var}
\end{align}
or, equivalently, on element form,
\begin{align}
\int_{\partial E} U_h \mathbf{n}_{F}\cdot\mathbf{n}_{ E} = \int_E q, \quad \forall E\in\mathcal{E}_h.
\label{eq:local_mc}
\end{align}

\paragraph{Global Conservation}
$U_h\in L^1(\mcFh)$ is globally conservative if it satisfies \eqref{eq:localcons-var} with $\varphi= 1$,
\begin{align}
\langle U_h,1 \rangle_{\mcF{h,B}} = (q,1)_{\mcEh}, \qquad \text{or} \qquad \int_{\partial \Omega} U_h = \int_{\Omega} q.
\label{eq:global_mc}
\end{align}
Global conservation follows from local conservation and flux continuity.

\paragraph{Time Dependent Flow}
For $\beta\neq 0$, denote by $U_h^{n}$ and $p_h^{n}$ the flux and pressure approximation at time $t_n$, respectively, and let $q^n=q(\cdot,t_n)$ and $\beta^n=\beta(\cdot,t_n)$. Now, local conservation is defined as
\begin{align}
\langle U_h^n, \jump{\varphi}\rangle_{\mcFh} =
(q^n - \bar{\partial}_t (\beta^n p_h^n),\varphi)_{\mcEh}, \quad \forall \varphi\in P_0(\mathcal{E}_h),
\label{eq:localcons-var-time}
\end{align}
or, equivalently, on element form,
\begin{align}
\int_{\partial E} U_h^n \mathbf{n}_{F}\cdot\mathbf{n}_{E} = \int_E \left(q^n - \bar{\partial}_t (\beta^n p_h^n) \right), \quad \forall E\in\mathcal{E}_h,
\label{eq:local_mc_time}
\end{align}
where $\bar{\partial}_t$ is the discrete approximation to $\partial_t$ used to solve the flow equation \eqref{eq:flow_eq}, e.g., for backward Euler with step size $\Delta t$, $\bar{\partial}_t p_h^n = \frac{1}{\Delta t}(p_h^n-p_h^{n-1})$.
\nomenclature{$\bar{\partial}_t$}{Discrete approximation to $\partial_t$}%
\nomenclature{$\Delta t$}{Time step size}%
\nomenclature{$p_h$}{Approximation to $p$}

Global conservation is in a similar manner defined as
\begin{align}
\langle U_h^n,1 \rangle_{\mcF{h,B}} = (q^n - \bar{\partial}_t (\beta^n p_h^n),1)_{\mcEh}, \qquad \text{or} \qquad \int_{\partial \Omega} U_h^n = \int_{\Omega} \left(q^n - \bar{\partial}_t (\beta^n p_h^n) \right).
\label{eq:global_mc_td}
\end{align}

\subsection{Numerical Schemes}

We will briefly write down the numerical schemes used to solve the flow and transport problem. 
The flow equation \eqref{eq:flow_eq} is solved with the continuous Galerkin (CG) finite element method, with either strong or weak enforcement of the Dirichlet conditions, while the transport equation \eqref{eq:transport_eq} is solved with a discontinuous Galerkin (DG) finite element method. For time integration we use backward Euler. 

\paragraph{CG Scheme for the Flow Equation}
Let $P^C_r(\mathcal{E}_h;\upsilon)$ denote the subspace of $P^C_r(\mathcal{E}_h)$ such that the trace on $\Gamma_D$ is equal to $\upsilon$,
\begin{align}
P^C_r(\mathcal{E}_h;\upsilon) = \{ \varphi\in P^C_r(\mathcal{E}_h) : \varphi\vert_{\Gamma_D} = \upsilon\}.
\end{align}
\nomenclature{$P^C_r(\mathcal{E}_h;\omega)$}{Space of piecewise polynomial functions of degree $r$ whose trace operator is equal to $\omega$ on the Dirichlet boundary}%
Denote by $\tilde{p}_B$ the projection of $p_B$ into the polynomial space.
Given $p_h^{n-1}$ with $p_h^0=p_0$, the standard CG scheme for Eq.\ \eqref{eq:flow_eq} is to seek $p_h^n\in P^C_r(\mathcal{E}_h;\tilde{p}_B)$ such that
\begin{align}
\left(\beta \bar{\partial}_t p_h^n, \varphi\right)_{\mcEh} + a(p_h^n,\varphi) = l(\varphi), \quad \forall \varphi\in P^C_r(\mathcal{E}_h;0),
\label{eq:cg_scheme_strong}
\end{align}
where the bilinear form $a(p,\psi)$ and the linear functional $l(\psi)$ are defined as follows:
\begin{align}
a(p,\psi) &= (\mathbf{K}\nabla p, \nabla \psi)_{\mcEh},  \\
l(\psi) &= (q,\psi)_{\mcEh} - \langle u_B,\psi \rangle_{\mcF{h,N}}.
\end{align}
\nomenclature{$a(\cdot,\cdot)$}{Bilinear form of the Laplace operator}%
\nomenclature{$l(\cdot$)}{Linear functional for Poisson equation}%
The energy norm associated with the discrete form \eqref{eq:cg_scheme_strong} is given as
\begin{align}
\Vert p \Vert_a^2 = a(p,p) = (\mathbf{K}\nabla p, \nabla p)_{\mcEh}.
\end{align}
\nomenclature{$\Vert \cdot \Vert_a$}{Energy norm, induced from bilinear form $a(\cdot,\cdot)$}%
In the case where $\mathbf{K}$ is the identity matrix and $p$ is sufficiently smooth, the following error estimates hold \cite{wheeler1973priori}, 
\begin{align}
\Vert p_h^n - p(t_n)\Vert \le C(h^{r+1} + \Delta t), \qquad \Vert p_h^n - p(t_n)\Vert_a \le C(h^r + \Delta t),
\label{eq:conv_estimate_pressure}
\end{align}
where $C$ is a constant independent on $h$ and $\Delta t$.

Alternatively, one may impose the Dirichlet conditions weakly by adding a penalty term. Instead of \eqref{eq:cg_scheme_strong} we seek $p_h^n\in P^C_r(\mathcal{E}_h)$ such that
\begin{align}
\left(\bar{\partial}_t (\beta^n p_h^n), \varphi\right)_{\mcEh} + \tilde{a}(p_h^n,\varphi) = \tilde{l}(\varphi), \quad \forall \varphi\in P^C_r(\mathcal{E}_h),
\label{eq:cg_scheme}
\end{align}
where the bilinear form $\tilde{a}(p,\psi)$ and the linear functional $\tilde{l}(\psi)$ are defined as follows:
\begin{align}
\tilde{a}(p,\psi) &= (\mathbf{K}\nabla p, \nabla \psi)_{\mcEh} + J_{D,\sigma}(p,\psi) 
- \langle \mathbf{K}\nabla p\cdot \bfn_F, \psi \rangle_{\mcF{h,D}} - s_{\text{form}} \langle \mathbf{K}\nabla \psi\cdot\bfn_F, p \rangle_{\mcF{h,D}}, \\
\tilde{l}(\psi) &= (q,\psi)_{\mcEh} + J_{D,\sigma}(p_B,\psi) - s_{\text{form}} \langle \mathbf{K}\nabla\psi\cdot\bfn_F, p_B \rangle_{\mcF{h,D}} - \langle u_B,\psi \rangle_{\mcF{h,N}}.
\end{align}
The Dirichlet penalty term $J_{D,\sigma}(p,\psi)$ is defined as
\begin{align}
 J_{D,\sigma}(p,\psi) = \left\langle \frac{r^2\sigma_{F}}{\vert F\vert}p,\psi \right\rangle_{\mcF{h,D}},
\end{align}
where the penalty parameter $\sigma_{F}$ is constant on each face. In our work, we set $s_{\text{form}}=1$, resulting in a symmetric formulation.
\nomenclature{$\sigma_{F}$}{Dirichlet penalty parameter}

\paragraph{Velocity Calculations from CG Solution}

Since $p_h$ is only $C^0$ continuous across element faces, the approximate velocity $\mathbf{u}_h=-\mathbf{K}\nabla p_h$ is undefined on the faces. For this reason, we take the average value and define the velocity approximation from CG as
\begin{align}
	\mathbf{u}_h &= -\mathbf{K}\nabla p_h, \quad \text{on } E\in\mathcal{E}_h, \\
	U_h & =
	\begin{cases}
		-\average{\mathbf{K} \nabla p_h \cdot \mathbf{n}}_{\theta}, & \text{on } F \in \mcF{h,I}, \\
		-\mathbf{K} \nabla p_h \cdot \mathbf{n} + \frac{r^2\sigma_{F}}{\vert F\vert} (p_h-p_B), & \text{on } F\in \mcF{h,D}, \\
		u_B, & \text{on } F\in \mcF{h,N}.
	\end{cases}
\label{eq:cg_flux}
\end{align}
The extra penalty term on the Dirichlet boundary is added to give a globally conservative approximation when boundary conditions are imposed weakly. Notice that this term vanishes for strong boundary conditions as $p_h=p_B$ on $\Gamma_D$. Global conservation for weak boundary conditions follows from \eqref{eq:cg_scheme} with $\varphi= 1$.

\paragraph{Flux Recovery on Dirichlet Boundary}

The flux approximation \eqref{eq:cg_flux} is not globally conservative when the boundary conditions are imposed strongly. However, there is a technique to recover globally conservative fluxes on the Dirichlet boundary  \cite{wheeler1973soh,carey1982derivative,carey1985abf,pehlivanov1992superconvergence,kvamsdal1998variationally,hughes2000continuous,melbo2003goal}. This method is briefly recaptured here. 

Let $P_r^C(\mcF{h,D}) = P^C_r(\mathcal{E}_h) \setminus P^C_r(\mathcal{E}_h;0)$, i.e., the space of continuous functions that are piecewise polynomials of order $r$ with support only on elements with at least one of its faces in $\mcF{h,D}$. The modified continuous Galerkin method now reads: Find $p_h^n\in P^C_r(\mathcal{E}_h;p_B)$ and $\mathcal{U}_h^n\in P_r^C(\mcF{h,D})$ such that
\begin{align}
- \langle\mathcal{U}_h^n, \varphi\rangle_{\mcF{h,D}} = a(p_h^n, \varphi) - l(\varphi) + \left(\bar{\partial}_t (\beta^n p_h^n), \varphi\right), \quad \forall \varphi\in P^C_r(\mathcal{E}_h).
\label{eq:flux_recovery}
\end{align}
\nomenclature{$\mathcal{U}_h$}{Dirichlet flux from recovery operation}%
We can now split this equation into two parts:
\begin{align}
0 &= a(p_h^n, \psi) - l(\psi) + \left(\bar{\partial}_t (\beta^n p_h^n), \psi\right), \quad \forall \psi\in P^C_r(\mathcal{E}_h;0), \label{eq:flux_recovery_1} \\
- \langle\mathcal{U}_h^n, \varphi\rangle_{\mcF{h,D}} &= a(p_h^n, \varphi) - l(\varphi) + \left(\bar{\partial}_t (\beta^n p_h^n), \varphi\right), \quad \forall \varphi\in P_r^C(\mcF{h,D}). \label{eq:flux_recovery_2}
\end{align}
The first equation is the original problem \eqref{eq:cg_scheme_strong}, while the second determines $\mathcal{U}_h^n$, which we can use as an approximation to the flux on the Dirichlet boundary. If we assume that $p_h^n$ is determined from \eqref{eq:cg_scheme_strong} (or equivalently \eqref{eq:flux_recovery_1}), the right hand side of \eqref{eq:flux_recovery_2} is given. Global conservation of the flux $\mathcal{U}_h^n$ follows from \eqref{eq:flux_recovery} with $\varphi= 1$.

\paragraph{DG Scheme for the Transport Equation}

Given $c_h^{n-1}$ with $c_h^0=c_0$, a DG scheme with upwinding \cite{sun2005symmetric} for Eq.\ \eqref{eq:transport_eq} with $\mathbf{D}= 0$ is to seek $c_h^n\in P_r(\mathcal{E}_h)$ satisfying
\begin{align}
\left(\bar{\partial}_t(\phi c_h^n), \varphi\right)_{\mcEh} + b(c_h^n, \varphi) = k(\varphi), \quad \forall \varphi\in P_r(\mathcal{E}_h),
\end{align} 
where the bilinear form $b(c,\psi)$ and the linear functional $k(\psi)$ are defined as follows:
\begin{align}
b(c,\psi) &=  - (c\bfu,\nabla\psi)_{\mcEh} - (q^- c, \psi)_{\mcEh} 
+ \langle c^*\mathbf{u}\cdot\bfn_F, \jump{\psi} \rangle_{\mcF{h,I}} + 
\langle c\bfu\cdot\bfn_F,\psi\rangle_{\mcF{h,\text{out}}} + J_{\sigma}(c,\psi), \\
k(\psi) &= (c_wq^+,\psi)_{\mcEh} - 
\langle c_B\bfu\cdot\bfn_F,\psi\rangle_{\mcF{h,\text{in}}}.
\end{align}
The interior penalty term is defined as
\begin{align}
 J_{\sigma}(c,\psi) = \left\langle \frac{r^2\sigma_{F}}{\vert F\vert}\jump{c},\jump{\psi} \right\rangle_{\mcF{h,I}},
\end{align}
while $c^*$ denotes the upwind concentration, defined as
\begin{align}
c^*\vert_{F} = 
\begin{cases}
(c\vert_{E_i})\vert_{F}, &\text{if } \mathbf{u}\cdot\mathbf{n}_{F} \ge 0, \\
(c\vert_{E_j})\vert_{F}, &\text{if } \mathbf{u}\cdot\mathbf{n}_{F} < 0,
\end{cases}
\end{align}
where $\bfn_F$ is exterior to $E_i$.
Furthermore, $q^-$ and $q^+$ are the negative and positive parts of the source term, respectively, i.e.\ 
\begin{align}
q^- = \min(q,0), \quad q^+ = \max(q,0).
\end{align}

The above scheme assumes that $\mathbf{u}$ is known. Whenever we only have an approximation, e.g.\ from \eqref{eq:cg_flux}, we substitute $\mathbf{u}$ by $\mathbf{u}_h$ and $\mathbf{u}\cdot\mathbf{n}_{F}$ by $U_h$.
In this work, we only consider the lowest order method ($r=0$), for which $\mathbf{u}$ (or an approximation to it) is not needed in the DG scheme since the first term in $b(c,\psi)$ vanishes.

\section{Postprocessing}
\label{sec:postprocessing}

In this section we will define an algorithm to postprocess a given flux approximation to obtain a locally conservative flux. In the derivation, we will assume a time independent problem ($\beta=0$), and then finally, in Section \ref{sec:timedep_pp}, we will show how this approach can be extended to the general case. We will start by defining a discrete divergence operator and its left inverse, and then later show how to use these to construct a locally conservative flux.

\subsection{A Discrete Divergence Operator and its Left Inverse}

\paragraph{Elementwise Definitions}
Let $\Dh:L^1(\mcFh)\rightarrow P_0(\mcEh)$ denote the discrete divergence operator defined by
\begin{align}\label{eq:Dhelem}
\int_E \Dh v = \int_{\partial E} v \bfnFE,
\qquad \forall v \in L^1(\mcFh), \qquad \forall E \in \mcEh.
\end{align}
Next, let $\Dhinv: P_0(\mcEh) \rightarrow P_0^0(\mcFh) $ be a left 
inverse of $\Dh$, i.e., 
\begin{align}\label{eq:Dhinvelem}
\int_E v = \int_{\partial E} (\Dhinv v) \bfnFE,
\qquad \forall v \in {P}_0(\mcEh),
\qquad \forall E\in\mathcal{E}_h.
\end{align}
Both $\Dh$ and $\Dhinv$ are linear, and by definition,
\begin{align}
\Dh \Dhinv v = v, \qquad \forall v \in {P}_0(\mcEh).
\label{eq:leftinv}
\end{align}
Observe that $\Dhinv$ takes functions into $P_0^0(\mcFh)$, so that $\Dhinv v = 0$ on $\Gamma_N$ by definition.

\paragraph{Variational Definitions} 
We note that we have the following equivalent forms of \eqref{eq:Dhelem} and \eqref{eq:Dhinvelem},
\begin{align}\label{eq:Dhvar}
(\Dh v, w)_{\mcEh} &= \langle v , \jump{w}\rangle_{\mcFh}, \qquad \forall w \in P_0(\mcEh),
\\ \label{eq:Dhinvvar}
(v,w)_{\mcEh} &= \langle \Dhinv v, \jump{w}\rangle_{\mcFh}, \qquad \forall w \in P_0(\mcEh).
\end{align}
To see that our definitions are equivalent, we may first test with the 
characteristic function of element $E$ to retrieve the elementwise definition 
from the variational formulations. Conversely we may multiply each elementwise 
equation with a constant and sum all the equations, and use the definition of 
the jump operator to conclude that the variational equations hold.

The left inverse $\Dhinv$ is not uniquely defined since the dimension of 
$P_0^0(\mcFh)$ is larger than the dimension of $P_0(\mcEh)$\footnote{This is true for most grids, and if not, then \eqref{eq:Dhinvvar} is sufficient.}. We may 
determine $\Dhinv v$ uniquely for each $v \in P_0(\mcEh)$ by 
minimizing a given norm of $\Dhinv v$. We next consider minimization with 
respect to a weighted $L^2$ norm.

\paragraph{Minimization} 
We define the weighted $L^2$ inner product and norm as
\begin{equation}
\langle v,w\rangle_{\omega,\mcFh} = \langle \omega v, w\rangle_{\mcFh} 
= \sum_{F \in \mcFh} \langle\omega v,w\rangle_F, 
\qquad \Vert v \Vert^2_{\omega,\mcFh} = \langle v,v\rangle_{\omega,\mcFh},
\label{eq:innerp-w}
\end{equation}
where $\omega |_F = \omega_F >0 $ for each $F \in \mcFh$ is a given bounded weight.
For $\omega= 1$, we have the standard $L^2$ norm.

Introducing the divergence-free subspace, $\pdiv$, defined by
\begin{align}
\pdiv = \{ v\in P_0^0(\mcFh) : D_hv = 0 \},
\end{align}
we have the orthogonal decomposition  
\begin{equation}
P_0^0(\mcFh) = \pdiv \oplus \pdivperp,
\label{eq:decomposition}
\end{equation}
with respect to the weighted inner product \eqref{eq:innerp-w}. 
For $v_0\in\pdiv$ we get from \eqref{eq:Dhvar} that
\begin{align}
0 = (\Dh v_0, w)_{\mcEh} = 
\langle v_0 , \jump{w}\rangle_{\mcFh} = 
\langle v_0, \omega^{-1}\jump{w}]\rangle_{\omega,\mcFh}, 
\qquad \forall w\in P_0(\mcEh).
\end{align}
Observe that the sum over $\mcF{h,N}$ vanishes as $v_0=0$ on $\mcF{h,N}$ by definition. Hence, alternatively, we may define \eqref{eq:decomposition} by
\begin{align}
\pdivperp = \{v \in P_0^0(\mcFh) : v = \omega^{-1}\jump{w} \text{ on } F\in\mcFh\setminus\mcF{h,N}, w\in P_0(\mcEh) \}.
\end{align}
It follows that for $v\in P_0(\mcEh)$, 
\begin{align}
\Dhinv v = z + \omega^{-1}\jump{y} \in \pdiv \oplus \pdivperp, \qquad \text{on } \mcFh\setminus\mcF{h,N},
\end{align}
for some $z\in\pdiv$ and $y\in P_0(\mcEh)$. 
Recall that $\Dhinv v=0$ on $\mcF{h,N}$. 
Using orthogonality and \eqref{eq:Dhinvvar} we obtain
\begin{align}
\begin{split}
(v,w)_{\mcEh} = \langle\Dhinv v, \jump{w}\rangle_{\mcFh} 
&= \langle z + \omega^{-1}\jump{y},\jump{w}\rangle_{\mcFh\setminus\mcF{h,N}}  \\
&= \langle z + \omega^{-1}\jump{y},\omega^{-1}\jump{w}\rangle_{\omega,\mcFh\setminus\mcF{h,N}} \\
&= \langle \omega^{-1}\jump{y},\omega^{-1}\jump{w}\rangle_{\omega,\mcFh\setminus\mcF{h,N}}, 
\qquad \forall w \in P_0(\mcEh).
\end{split}
\end{align}
Furthermore, since  
\begin{equation}
\Vert\Dhinv v \Vert^2_{\omega,\mcFh} =  \Vert z + \omega^{-1}\jump{y}\Vert^2_{\omega,\mcFh\setminus\mcF{h,N}} = \Vert z \Vert^2_{\omega,\mcFh\setminus\mcF{h,N}} + \Vert\omega^{-1}\jump{y} \Vert^2_{\omega,\mcFh\setminus\mcF{h,N}} 
\end{equation}
we see that minimizing the norm $\Vert\Dhinv v \Vert^2_{\omega,\mcFh}$ enforces 
$z = 0$.

We conclude that, subject to minimization,
\begin{align}
\Dhinv v = 
\begin{cases}
0, & \text{on } \mcF{h,N}, \\
\omega^{-1}\jump{y}, & \text{otherwise}.
\end{cases}
\label{eq:correction}
\end{align}
where $y\in P_0(\mcEh)$ is the solution to the variational problem
\begin{align}
d(y,w) = (v,w)_{\mcEh}, \qquad \forall w\in P_0(\mcEh).
\label{eq:Dhinv_variational}
\end{align}
The bilinear form $d(y,w) : P_0(\mcEh) \times P_0(\mcEh) \rightarrow \mathbb{R}$ is defined as
\begin{align}
d(y,w) = \langle\omega^{-1}\jump{y},\omega^{-1}\jump{w}\rangle_{\omega,\mcFh\setminus\mcF{h,N}} = 
\langle\omega^{-1}\jump{y},\jump{w}\rangle_{\mcFh\setminus\mcF{h,N}}.
\end{align}
We prove later, in Lemma \ref{lemma:unique}, that \eqref{eq:Dhinv_variational} admits a unique solution. The choice of weights is discussed in Section \ref{sec:pp_weights}.

\paragraph{The Operator ${\Dhinv \Dh}$} 

Let $v\in L^1(\mcFh)$. From the definitions \eqref{eq:Dhvar} and \eqref{eq:Dhinvvar}  
we have the following identity 
\begin{align}\label{eq:proj-a}
\langle \Dhinv \Dh v - v, \jump{w}\rangle_{\mcFh} &= 0, \qquad \forall w\in P_0(\mcEh),
\end{align}
since
\begin{align}
\langle\Dhinv \Dh v, \jump{w}\rangle_{\mcFh} &= (\Dh v, w)_{\mcEh} = \langle v,\jump{w}\rangle_{\mcFh}.
\end{align}
Now using \eqref{eq:correction} we know that there is an
$y\in P_0(\mcEh)$ such that $\Dhinv \Dh v = \omega^{-1} \jump{y}$ on $\mcFh\setminus\mcF{h,N}$ (and $\Dhinv \Dh v = 0$ on $\mcF{h,N}$). From \eqref{eq:proj-a} it follows that
\begin{align}
\begin{split}
0 &= \langle\omega^{-1} \jump{y} - v, \jump{w}\rangle_{\mcFh\setminus\mcF{h,N}} + \langle - v, \jump{w}\rangle_{\mcF{h,N}} \\
&= \langle\omega^{-1} \jump{y} - v, \omega^{-1}\jump{w}\rangle_{\omega,\mcFh\setminus\mcF{h,N}} + \langle - v, \omega^{-1}\jump{w}\rangle_{\omega,\mcF{h,N}}
, \qquad \forall w\in P_0(\mcEh).
\label{eq:proj-ip}
\end{split}
\end{align}
Now, if $v=0$ on $\mcF{h,N}$, the second term vanish. If we denote by $L^{1}_0(\mcFh)$ the subspace of $L^{1}(\mcFh)$ with functions that are zero on $\mcF{h,N}$, i.e.,
\begin{align}
L^{1}_0(\mcFh) = \left\{ v\in L^1(\mcFh) : v\vert_F=0, F\in \mcF{h,N} \right\},
\end{align}
we conclude from \eqref{eq:proj-ip} that the operator $\Dhinv \Dh : L^{1}(\mcFh) \rightarrow P_0(\mcFh)$ is the orthogonal projection of 
$L^{1}_0(\mcFh)$ onto the subspace $\pdivperp$
with respect to the weighted inner product $\langle \cdot, \cdot\rangle_{\omega,\mcFh}$. In particular, 
it follows that 
\begin{equation}
\Vert\Dhinv \Dh \Vert = 1.
\end{equation}

\paragraph{Remark} 
An alternative approach to obtain \eqref{eq:correction} and \eqref{eq:Dhinv_variational} is to use Lagrangian multipliers for minimizing $\Vert \Dhinv v \Vert^2_{\omega,\mcFh}$ subject to the constraints \eqref{eq:Dhinvvar}. If we let $x=\Dhinv v$, the Lagrangian reads
\begin{equation}
L(x,\lambda) = \frac{1}{2}\Vert x \Vert^2_{\omega,\mcFh} 
- \langle x, \jump{\lambda}\rangle_{\mcFh}
+ (v,\lambda)_{\mcEh},
\end{equation}
with corresponding derivative $DL: P_0^0(\mcFh) \times P_0(\mcEh)
\rightarrow \mathbb{R}$ given by
\begin{equation}
DL(x,\lambda)(\delta x, \delta\lambda) 
= \langle \omega x,\delta x\rangle_{\mcFh} - \langle\delta x,\jump{\lambda}\rangle_{\mcFh}
+ \langle x, \jump{\delta \lambda}\rangle_{\mcFh}
+ (v,\delta \lambda)_{\mcEh}.
\end{equation}
By requiring $DL(x,\lambda)(\delta x, \delta\lambda) = 0, \forall \delta x\in P_0(\mcFh), \forall \delta \lambda\in P_0(\mcEh)$, we end up with the same result as \eqref{eq:correction} and \eqref{eq:Dhinv_variational}.

\subsection{Postprocessing Algorithm} 
In the following, let $U_h\in L^1(\mcFh)$ be some approximation to the flux $\bfu\cdot\bfn$ on $\mcFh$. We define a residual operator, $\mathcal{R}: L^1(\mcFh) \rightarrow P_0(\mcEh)$, to measure to discrepancy from local conservation,
\begin{align}
\mathcal{R}(U_h) = P_0 q - \Dh U_h,
\label{eq:residual}
\end{align}
where $P_0$ is the $L^2$ projection onto $P_0(\mcEh)$, i.e., $(P_0q)\vert_E = \vert E \vert^{-1} \int_E q$. Clearly, $U_h$ is locally conservative if and only if $\mathcal{R}(U_h) = 0$, and $U_h$ is globally conservative if and only if $\int_{\Omega} \mathcal{R}(U_h) = 0$.

The next lemma shows how the left inverse $\Dhinv$ can 
be used to project an arbitrary flux approximation to a locally conservative flux. 

\begin{lemma}
\label{lemma:cons}
Given $U_h\in L^1(\mcFh)$, the modified flux
\begin{align}\label{eq:def-postprocessing}
V_h = U_h +\Dhinv(\mathcal{R}(U_h)) = U_h + \Dhinv (P_0q - \Dh U_h)
\end{align}
is locally conservative.
\end{lemma}

\begin{proof} Using the fact that $\Dhinv$ is a left inverse of $\Dh$ 
we obtain
\begin{align}
\begin{split}
\mathcal{R}(V_h) = P_0q - \Dh V_h 
&= P_0q - \Dh\left( U_h + \Dhinv(P_0q - D_h U_h) \right) \\
&= P_0q - \Dh U_h - \Dh\Dhinv P_0q + \Dh\Dhinv\Dh U_h \\
&= P_0q - \Dh U_h - P_0q + \Dh U_h \\
&= 0.
\end{split}
\end{align}
\end{proof}

Applying \eqref{eq:correction} and \eqref{eq:Dhinv_variational}, we may summarize the postprocessing algorithm as in the box below. The postprocessing steps and the different operators are illustrated in Fig.\ \ref{fig:pp_operators}.

\begin{framed}
\noindent\textbf{Postprocessing algorithm}\\[2mm]
Given $U_h\in L^1(\mcFh)$, the postprocessed flux is defined as
\begin{align}
V_h = U_h + \Dhinv(\mathcal{R}(U_h)) = 
\begin{cases}
U_h , & \text{on } \mcF{h,N}, \\
U_h + \omega^{-1} \jump{y}, & \text{on } \mcFh\setminus\mcF{h,N},
\end{cases}
\label{eq:pp_def}
\end{align}
where $y\in P_0(\mcEh)$ is the unique solution to
\begin{align}
d(y,w) = (\mathcal{R}(U_h), w)_{\mcEh}, \qquad \forall w\in P_0(\mcEh),
\label{eq:pp_variational}
\end{align}
with
\begin{align}
d(y,w) = \langle\omega^{-1}\jump{y},\jump{w}\rangle_{\mcFh\setminus\mcF{h,N}}.
\end{align}
\end{framed}

\begin{figure}[bpt]
\centering
\includegraphics[width=0.7\textwidth]{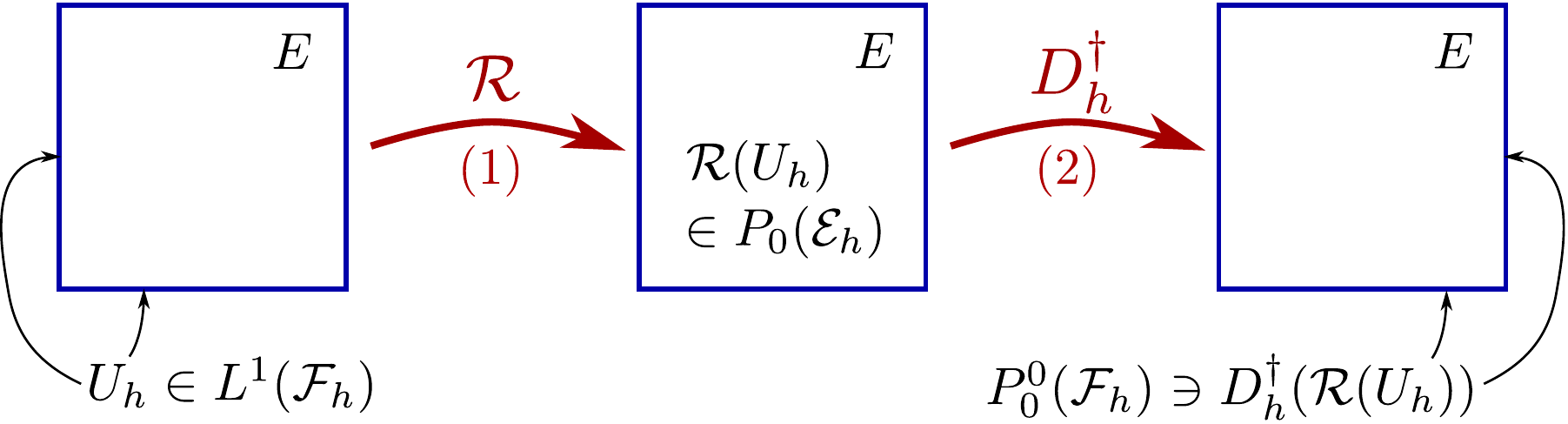}
\caption{Illustration of the postprocessing process. A non-conservative flux $U_h$ is taken as input. First the operator $\mathcal{R}$ calculates the element residuals (1). Then the operator $\Dhinv$ projects the residuals onto the element faces such that the updated flux $V_h = U_h - \Dhinv(\mathcal{R}(U_h))$ is locally conservative (2). This is a global process, although illustrated on a single element $E$ here for simplicity.}
\label{fig:pp_operators}
\end{figure}

\begin{lemma}
\label{lemma:unique}
The variational problem \eqref{eq:pp_variational} has a unique solution.
\end{lemma}

\begin{proof}
We need to prove coersivity of the bilinear form $d(\cdot,\cdot)$. If $w\in P_0(\mcEh)$ and $d(w,w) = \Vert \omega^{-1}\jump{w} \Vert_{\omega,\mcFh} = 0$ then $w$ is a constant function. If $\Gamma_D$ is nonempty then $\jump{w}\vert_F = w_F$ for $F\subset \Gamma_D$, so that $w=0$. Otherwise, if $\Gamma_D$ is empty, then $w$ may be a nonzero constant $C$, but then the right hand side
\begin{align}
(\mathcal{R}(U_h), C)_{\mcEh}
= C \int_{\Omega} \mathcal{R}(U_h)
= C \left( \int_{\Omega} q - \int_{\Gamma_N} u_B \right)
= 0,
\end{align} 
since we require $\int_{\Gamma_N} u_B = \int_{\Omega} q$ for the pure Neumann problem to be well posed. This shows uniqueness up to a constant. Since we only need the jump in $y$, our algorithm is well defined.
\end{proof}

\paragraph{Matrix Formulation}

Let $\chi_i$, for $i=1,2,\ldots,N$, denote the characteristic functions, i.e., $\chi_i=1$ for $x\in E_i$ and $0$ otherwise. This is a basis for $P_0(\mcEh)$,  so we can write $y=\sum_{i=1}^Ny_i\chi_i$ and express the variational formulation \eqref{eq:pp_variational} in matrix form
\begin{align}
\mathbf{A}\mathbf{y} = \mathbf{r},
\label{eq:pp_matrix_system}
\end{align}
where $\mathbf{A}\in\mathbb{R}^{N\times N}$ is the matrix with entries
\begin{align}
A_{ij} &= d(\chi_j,\chi_i) = 
\langle \omega^{-1}\jump{\chi_j},\jump{\chi_i} \rangle_{\mcFh} =
\begin{cases}
-\omega_F^{-1} \vert F\vert, &  i\neq j, F=\partial E_i\cap\partial E_j, \\[3mm]
\displaystyle
\sum_{F\in\partial E_i\setminus \Gamma_N} \omega_F^{-1} \vert F\vert, & i=j.
\end{cases}
\label{eq:pp_matrix_entries}
\end{align}
Furthermore, $\mathbf{y}\in \mathbb{R}^N$ is the vector with entries $y_i$, and $\mathbf{r}\in \mathbb{R}^N$ is the vector of residuals, i.e., with entries
\begin{align}
r_i = \left(\mathcal{R}(U_h), \chi_i\right) = \int_{E_i}q - \int_{\partial E_i} U_h\bfn_F\cdot\bfn_{E_i}.
\end{align}
Observe that $\mathbf{A}$ is symmetric with non-zero pattern equal to the grid connectivity.

\subsection{Error Estimate}

To measure the error on $\mcFh$ we introduce the face norm
\begin{align}
\Vert v \Vert_{h,\mcFh}^2 = \sum_{F\in\mcFh} h \Vert v \Vert_{F}^2.
\end{align}
This norm has the advantage that $\Vert 1 \Vert_{\mcFh,h}$ is bounded as $h\rightarrow 0$.
Furthermore, we use the notation $x \lesssim y$ whenever there exists a positive constant $C$ independent on $h$ such that $x\le Cy$.

\begin{lemma} 
\label{lemma:ee}
If $U_h$ is an approximation to the exact flux $U = \bfu \cdot \bfn$ such that 
\begin{equation}
\Vert U - U_h \Vert_{h,\mcFh} \lesssim  h^s,
\end{equation}
then the local conservation residual satisfies the estimate
\begin{equation}\label{eq:residual-est}
\Vert \mathcal{R}(U_h) \Vert_{\mcEh} \lesssim h^{s-1},
\end{equation}
and the postprocessed locally conservative flux
$V_h$, defined by (\ref{eq:pp_def}), satisfies 
\begin{equation}\label{eq:normalvel-estimate}
\Vert U - V_h \Vert_{h,\mcFh} \lesssim h^{s}.
\end{equation}
\end{lemma}
\begin{proof} We have 
\begin{equation}
\Vert \mathcal{R}(U_h) \Vert_{\mcEh}
=
\Vert P_0 q - \Dh U_h \Vert_{\mcEh}  
= 
\Vert \Dh U - \Dh U_h \Vert_{\mcEh}
= 
\Vert \Dh (U - U_h) \Vert_ {\mcEh} 
\lesssim 
h^{-1/2} \Vert U - U_h \Vert_ {\mcFh}.
\end{equation}
Here we have used that $P_0q=\Dh U$ and the bound $\Vert \Dh v \Vert_{\mcEh} \lesssim h^{-1/2} \Vert v \Vert_{\mcFh}$ 
which follows by setting $w= \Dh v$ in (\ref{eq:Dhvar}), 
\begin{equation}
\Vert \Dh v \Vert^2_{\mcEh} 
= (v, [\Dh v])_{\mcFh} 
\leq
\Vert v \Vert_{\mcFh} \Vert [\Dh v ] \Vert_{\mcFh} 
\lesssim 
\Vert v \Vert_{\mcFh} h^{-1/2} \Vert \Dh v  \Vert_{\mcEh}.
\end{equation} 
In the last step we used the triangle inequality and the fact that 
$\Vert [w] \Vert_{\partial E} \lesssim h^{-1/2} \Vert w \Vert_E$ for $ w \in P_0(E)$.
The bound on $\Vert \mathcal{R}(U_h) \Vert_{\mcEh}$ \eqref{eq:residual-est} follows since
\begin{align}
\Vert v \Vert_{\mcFh} \lesssim h^{-1/2} \Vert v \Vert_{h,\mcFh}.
\end{align}

Furthermore, we have
\begin{align}
\begin{split}
\Vert U - V_h \Vert_{h,\mcFh} 
&= 
\Vert U - (U_h + \Dhinv (P_0q - \Dh U_h )\Vert_{h,\mcFh} \\
&= 
\Vert (U - U_h) - \Dhinv \Dh U + \Dhinv \Dh U_h )\Vert_{h,\mcFh} \\
&\leq \Vert U - U_h \Vert_{h,\mcFh} + \Vert \Dhinv \Dh (U - U_h )\Vert_{h,\mcFh} 
\\
&\lesssim   
\Vert U - U_h \Vert_{h,\mcFh} 
\end{split}
\end{align}
where we used that $U-U_h$ is zero on the Neumann boundary so that $\Dhinv \Dh$ is a projection.
\end{proof}

The following main result follows directly from Lemma \ref{lemma:cons}, \ref{lemma:unique} and \ref{lemma:ee}.

\begin{theorem}
\label{theorem}
The postprocessed flux as defined by Eq.\ \ref{eq:pp_def} is (i) locally conservative; (ii) uniquely defined; and (iii) has the same convergence order as the original flux.
\end{theorem}

\subsection{Alternative Approach}

An alternative approach to the one depicted above is to work on the element level. After realizing that $\Dhinv v\in \pdivperp$, one may construct a basis for $\pdivperp$. The set $\{\varphi_i\}_{i=1}^N$, with
\begin{align}
\varphi_i = 
\begin{cases}
- \omega_F^{-1} \bfn_F \cdot \bfn_{E_i}, & x\in F \subset \partial E_i \setminus \Gamma_{N}, \\
0, & \text{otherwise},
\end{cases}
\label{eq:basis}
\end{align}
is a basis for $\pdivperp$. We can then write
\begin{align}
\Dhinv v = \sum_{i=1}^N \alpha_i \varphi_i.
\end{align}
From the requirement of $\Dhinv$ given by \eqref{eq:leftinv}, we get that
\begin{align}
\Dh\Dhinv v = D_h\left( \sum_{i=1}^N \alpha_i \varphi_i \right) = \sum_{i=1}^N \alpha_i \Dh \varphi_i = v, 
\qquad \forall v\in P_0(\mcEh).
\end{align}
This is a linear system of $N$ equations that uniquely determines the coefficients $\alpha_i$ for a given $v$.

We remark that this is the approach presented in \cite{sun2006projections}, but for the pure Dirichlet problem and only for the case where $\Dhinv v$ is minimized in the standard $L^2$ norm. The basis used in \cite{sun2006projections} is
\begin{align}
\tilde{\varphi}_i = 
\begin{cases}
-\frac{\vert E_i\vert}{\vert F\vert} \bfn_F \cdot \bfn_{E_i}, & x\in F \subset \partial E_i \\
0, & \text{otherwise}.
\end{cases}
\label{eq:basis_sun}
\end{align}
One can show that this is a basis only when $\vert F\vert = C$ for all $F\in\mcFh$, i.e., when all faces are equally large.

\subsection{Choice of Weights}
\label{sec:pp_weights}

An important parameter in our postprocessing method is the choice of weights. Using $\omega= 1$ will result in minimization in the standard $L^2$ norm. This means that the correction $\Dhinv(\mathcal{R}(U_h))$ will be minimized, but such that all faces are given the same weight. By choosing $\omega\neq 1$, we can control which faces should be weighted most in the minimization process. Our choice of weights is the inverse of the effective normal component of the permeability, i.e.,
\begin{align}
\omega_F = k_e^{-1} = \frac{\delta_{Kn}^i + \delta_{Kn}^j}{2\delta_{Kn}^i\delta_{Kn}^j},
\label{eq:weights}
\end{align}
where $\delta_{Kn}^i$ was defined in Eq.\ \eqref{eq:harmonic_weights}.

With this choice, $\Dhinv v = k_e\jump{y}$, so that faces with low effective permeability will have a relatively small correction. We will reason this choice by an example. Consider two neighboring elements sharing the face $F$ and with isotropic permeability $k_1$ and $k_2$. If we fix $k_1=1$, the effective permeability will be $k_e = 2k_2/(1+k_2)$. In the limit $k_2\rightarrow 0$, this face should approach a no-flow interface (a Neumann type of boundary with $u_B=\bfu\cdot\bfn =0$). With the harmonic average $\average{\cdot}_{\vartheta}$, $U_h$ as defined from the CG solution, Eq.~\eqref{eq:cg_flux}, would approach zero as desired. However, in the postprocessing step, the correction on $F$ can be made relatively large (compared to $U_h$) if $\omega=1$, and thus the effect of harmonic averaging might be reduced after postprocessing. Using \eqref{eq:weights}, we are able to preserve $V_h\sim 0$. The drawback is that the correction we are doing to the original flux will be larger measured in the standard $L^2$ norm. In Section \ref{sec:examples}, we will demonstrate the effect of weighting with some numerical examples.

\subsection{Time Dependent Flow}
\label{sec:timedep_pp}

Let us now look at the case with time dependent pressure and flux, i.e., $\beta\neq 0$. We need to take the compressibility (or time dependency of the pressure) into account when calculating the residual. If we discretize the flow equation \eqref{eq:flow_eq} in time, we get
\begin{align}
\bar{\partial}_t(\beta^n p^n) - \nabla\cdot(\mathbf{K}\nabla p^n) = q^n,
\end{align}
where $p^n$ and $q^n$ are the pressure and source, respectively, at time $t=t_n$. Now, treating $\bar{\partial}_t(\beta^n p^n)$ as a source term, we can extend the postprocessing method by replacing $q$ by $\tilde{q}=q^n-\beta \bar{\partial}_t(\beta^n p^n)$ in the above formulation. The residual operator now reads
\begin{align}
\mathcal{R}(U_h^n) = P_0(q^n-\bar{\partial}_t(\beta^n p^n) - \Dh U_h^n.
\end{align}
We may now use the algorithm given by Eq.~\eqref{eq:pp_def} with this extended residual operator.

For a time dependent problem, we need to perform postprocessing after each time step. However, we observe that the matrix $\mathbf{A}$ in Eq.~\eqref{eq:pp_matrix_system} is only dependent on the weights $\omega$ and the grid. 
Thus, we only need to assemble $\mathbf{A}$ whenever we alter the grid. 

\subsection{Postprocessing Parameters}
\label{sec:pp_parameters}

Given a CG pressure solution $p_h$, we have introduced different ways to calculate the CG flux approximation $U_h$. The first parameter is how we calculate the flux along the Dirichlet boundary, and the second parameter is the choice of weights $\theta$ in the average operator. To clearly express which method we are using, we introduce the following notation:
\begin{align}
\text{CG}(\alpha, \theta), \qquad \alpha = \{\text{SD,WD,RD}\}, \quad \theta=\{1/2,\vartheta\}.
\end{align}
The CG flux $U_h$ is then calculated as follows. On the internal and Neumann faces we have
\begin{align}
	U_h & =
	\begin{cases}
		-\average{\mathbf{K} \nabla p_h \cdot \mathbf{n}_{F}}_{\theta}, & \text{on } F \in \Gamma_{h,I}, \\
		u_B, & \text{on } F\in \Gamma_{h,N}.
	\end{cases}
\end{align}
The flux calculation on the Dirichlet boundary is given by $\alpha$ in the following way:
\begin{itemize}
\item $\alpha=\text{SD}$: CG with strong Dirichlet boundary conditions (Eq.~\eqref{eq:cg_scheme_strong}),
\begin{align}
U_h = -\mathbf{K} \nabla p_h \cdot \mathbf{n}_{F}, \quad \text{on } F \in \Gamma_{h,D}.
\end{align}
\item $\alpha=\text{WD}$: CG with weak Dirichlet boundary conditions (Eq.~\eqref{eq:cg_scheme}),
\begin{align}
U_h = -\mathbf{K} \nabla p_h \cdot \mathbf{n}_{F} + \frac{r^2\sigma_{F}}{\vert F\vert} (p_h-p_B), \quad \text{on } F \in \Gamma_{h,D}.
\end{align}
\item $\alpha=\text{RD}$: CG with strong Dirichlet boundary conditions and with recovered flux along the Dirichlet boundary (Eq.~\eqref{eq:flux_recovery_2}),
\begin{align}
U_h = \mathcal{U}_h, \quad \text{on } F \in \Gamma_{h,D}.
\end{align}
\end{itemize}

Furthermore, for the postprocessed flux, we have one more parameter describing which norm we are using for minimization. We use the following notation,
\begin{align}
\text{PP}(\alpha, \theta, \lambda), \qquad \alpha = \{\text{SD,WD,RD}\}, \quad \theta=\{1/2,\vartheta\}, 
\quad \lambda = \{\text{L2},\text{wL2}\},
\end{align}
where $\lambda=\text{L2}$ and $\lambda=\text{wL2}$ denotes minimization in the standard $L^2$ norm and the weighted $L^2$ norm, respectively. In the weighted $L^2$ norm we use weight $\omega = k_e^{-1}$ as described in Section \ref{sec:pp_weights}. We note that the methods considered in \cite{sun2006projections} and \cite{larson2004conservative} corresponds to CG(SD,1/2,L2).

In the case of homogeneous permeability, the parameters $\theta$ and $\lambda$ are obsolete, and we simply write CG($\alpha$) and PP($\alpha$). In the case $\text{PP}(\text{RD},\cdot,\cdot)$, we consider the flux on the Dirichlet boundary as fixed and thus consider the postprocessing step as a pure Neumann problem.

\section{Numerical Examples}
\label{sec:examples}

The postprocessing algorithm, along with solvers for the flow and transport equations, have been implemented. All implementations are based on the open source finite element library deal.II \cite{bangerth2007dealii}.
The numerical examples and timings were performed on a single core of an Intel Xeon X7542 (2.67\,GHz, 18\,MB cache) with 64-bit Ubuntu 14.04 and 256\,GB memory.
For the flow equation we use CG with bilinear elements ($r=1$), while for the transport equation we use DG with piecewise constants ($r=0$). 
In this section we run a series of test cases to verify our implementations and evaluate the postprocessing algorithm. Our main objectives are to
\begin{enumerate}[itemsep=0pt,label*={(\roman*)}]
\item Verify that the postprocessed flux is locally conservative on a range of grid types;
\item Test if we are able to recover exact flux for a problem with analytic solution of one polynomial degree higher than the test space (expressed as an amenable \emph{consistency condition} in \cite[Section 4.1]{Ainsworth1991ape});
\item Study the effect of how flux on the Dirichlet boundary is calculated, as discussed in Section \ref{sec:pp_parameters};
\item Verify the error estimates given by Lemma \ref{lemma:ee};
\item Study the choice of weights in the average operator and the choice of norm used for minimization in the postprocessing method;
\item Measure the computational complexity of the postprocessing problem compared to the flow problem.
\item Demonstrate the importance of locally conservative flux when solving the transport equation.
\end{enumerate}

For the latter objective, we introduce an overshoot quantity for the concentration solution $c_h$,
\begin{align}
\mathcal{O}(c_h) = \Vert \text{max}(c_h-\bar{c},0) + \text{max}(-c_h,0)\Vert_{\mcEh},
\label{eq:overshoot}
\end{align}
where $\bar{c}$ is the upper bound on the concentration, $\bar{c} = \text{max}(c_B, c_w, c_0)$. For the incompressible flow problem ($\beta=0$), the concentration is expected to obey the maximum principle $c\le \bar{c}$ and be positive. Hence, $\mathcal{O}(c_h)$ is used as a measure of the violation of these principles.

To solve the coupled flow and transport problem, Eq.\ \eqref{eq:flow_eq}--\eqref{eq:transport_eq}, we use an iterative solution technique. In each time step we first solve for pressure, then postprocess the flux if necessary, and at last solve the transport problem with the obtained flux approximation. This coupled process is illustrated by the flow chart in Fig.\ \ref{fig:flowchart}. If $\beta=0$, we only need to solve for pressure and postprocess the flux once, and then do time iterations on the transport solver only. We also run cases without the postprocessing step, i.e., use $U_h$ directly in the transport solver.

\begin{figure}[bpt]
\centering
\subfloat[$\beta = 0$.]{
\includegraphics[height=0.4\textwidth]{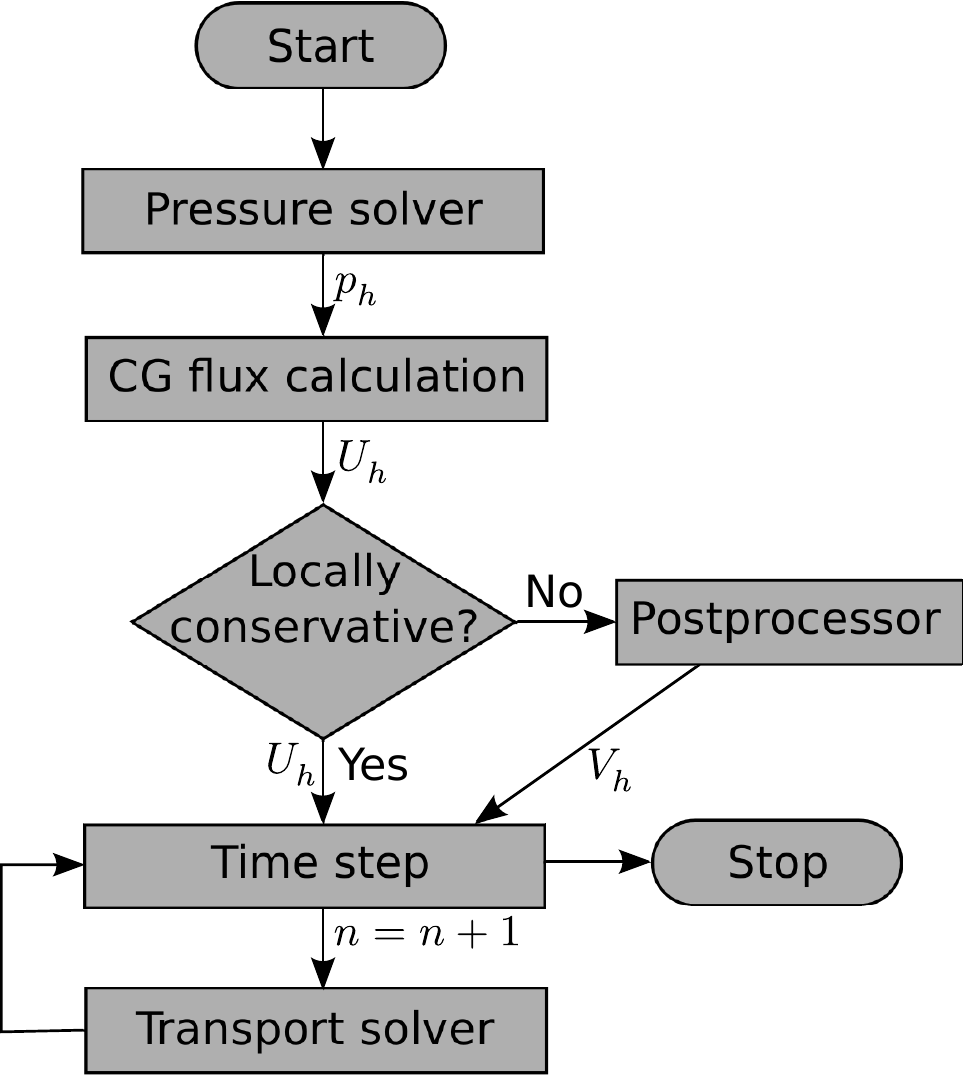}
}
\hspace*{10mm}
\subfloat[$\beta\neq 0$.] {
\includegraphics[height=0.4\textwidth]{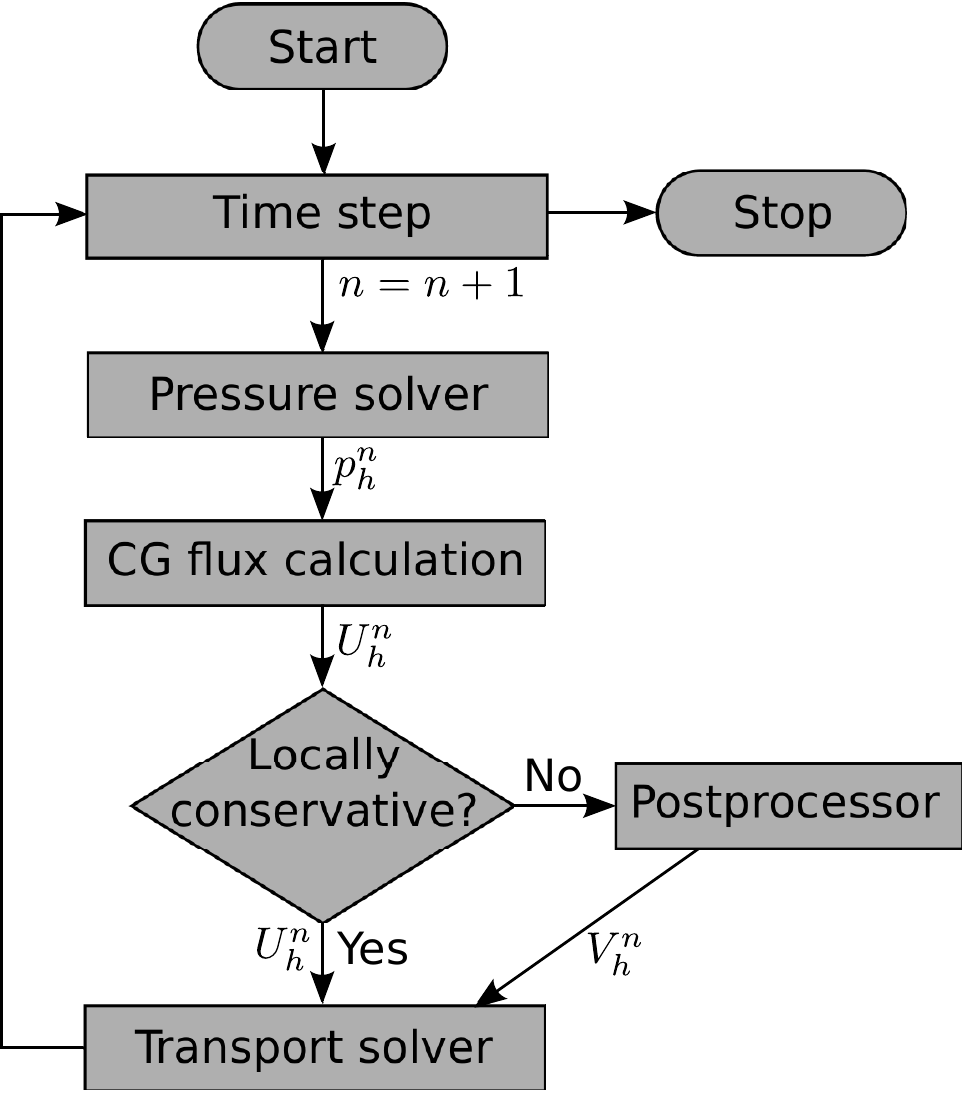}
}
\caption{Flowcharts describing the solution strategy for the elliptic case (a) and the parabolic case (b).}
\label{fig:flowchart}
\end{figure}

\subsection{Consistency Tests}

Our first example is a pure flow problem to examine the objectives (i)-(iii).
Consider the problem 
\begin{subequations}
\begin{align}
-\nabla\cdot(\nabla p) &= 2, \quad\text{on } \Omega=(0,1)^2, \\
p &= 1, \quad \text{for } x=0, \\
p &= 0, \quad \text{for } x=1, \\
\bfu\cdot\bfn &= 0, \quad \text{for } y=\{0,1\}.
\end{align}
\end{subequations}
This problem has the analytical solution $p(x,y) = 1-x^2$, and is essentially a one-dimensional problem. Since the permeability tensor is constant ($\mathbf{K}=\mathbb{I}$), there is no effect of harmonic averaging or weighting of the $L^2$ norm.

Results for different grids and calculations of fluxes along the Dirichlet boundary are presented in Table \ref{tab:comp_tests}. First observe that the residual for the postprocessed flux, $\mathcal{R}(V_h)$, is zero in all cases. This demonstrates that $V_h$ is locally conservative and that our postprocessing method works. For the uniform 1D grid all methods give exact solution for $V_h$. The flux error, $\Vert \mathbf{u}\cdot\mathbf{n} - U_h\Vert_{\mcFh}$, for CG(WD) can be made arbitrarily small by increasing the penalty term $\sigma_F$. This illustrate some of the ambiguity with weak boundary conditions. In the limit $\sigma_F\rightarrow\infty$, CG(WD) and CG(RD) are equivalent.
The postprocessed flux error, $\Vert \mathbf{u}\cdot\mathbf{n} - V_h\Vert_{\mcFh}$, for CG(SD) and CG(WD) is non-zero for the nonuniform 1D grid because the flux $U_h$ on the Dirichlet boundary is wrong. On this grid, CG(RD) reproduce the exact flux. For the two latter grids, the distorted and matching 2D grids, CG(SD) seems to give the best result. 

We observe that for the distorted and non-matching 2D grids, we do not obtain exact fluxes for CG(RD). In Table \ref{tab:gamma_flux} we report on the integrated flux $\int_{\gamma}U_h$ along vertical mesh lines $\gamma$, which divides the domain $\Omega$ in two. For the nonmatching 2D grid (Table~\ref{tab:gamma_flux_nm}), we see that we recover the exact value with all methods. For the distorted 2D grid (Table~\ref{tab:gamma_flux_d}), this is only the case for CG(RD). This follows from the fact that the fluxes are globally conservative and that the integrated flux is exactly recovered along the Dirichlet boundaries \cite{kvamsdal1998variationally}. Notice that for CG(SD) and CG(WD), the value of the integrated flux $V_h$ is shifted by the same value for all $\gamma$ ($0.0033$ for CG(SD) and $0.0020$ for CG(WD)).

\begin{table}[bpt]
\caption{\textbf{Consistency tests}. Norm of residual and flux error before ($U_h$) and after ($V_h$) postprocessing for different grids and flux calculations along the Dirichlet boundary. The penalty term for CG(WD) is $\sigma_{\gamma} = 10$.}
\label{tab:comp_tests}
\vspace*{2mm}
\centering
\footnotesize

\subfloat[Uniform 1D grid.]{
\begin{minipage}{0.09\textwidth}
\includegraphics[width=\linewidth]{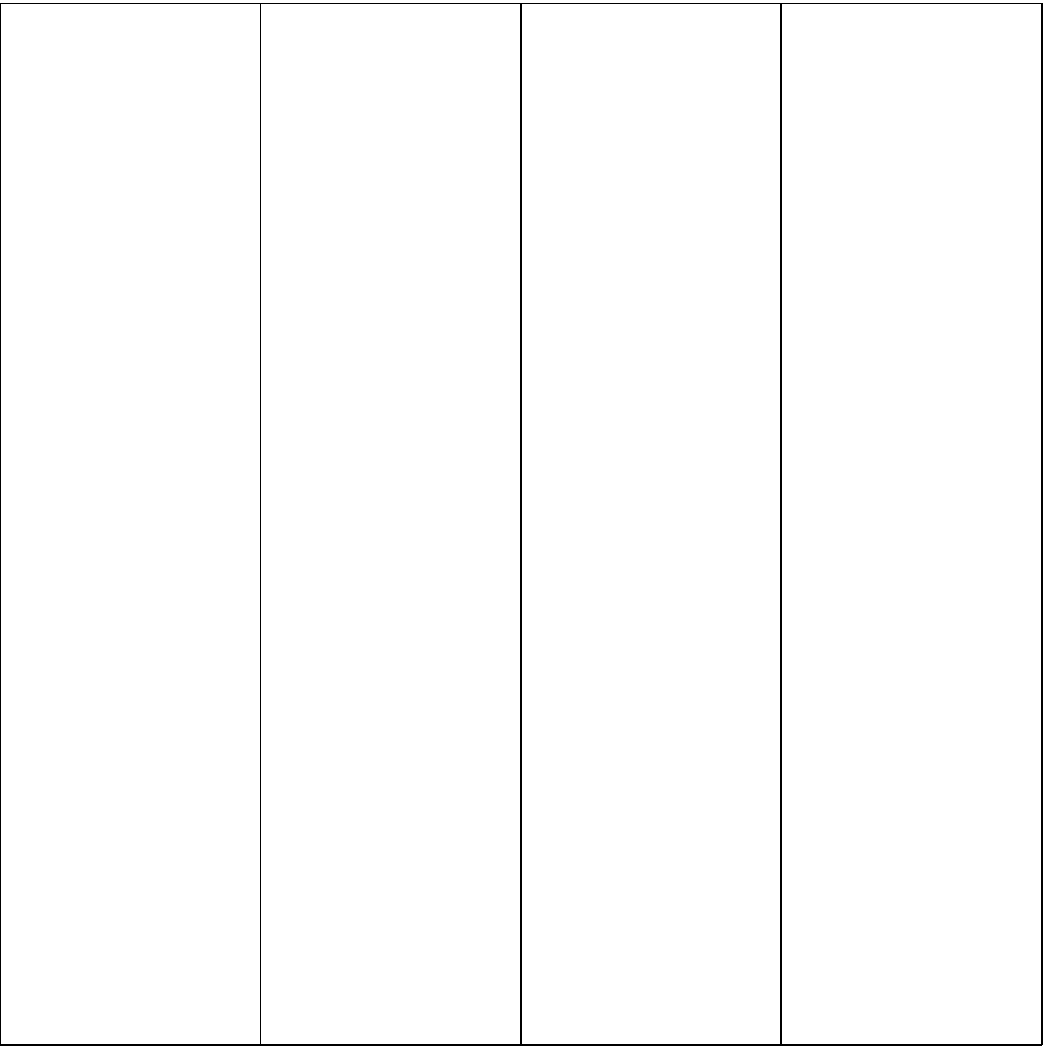}
\end{minipage}
\hspace*{2mm}
\begin{tabular}{lcccc}
\hline
Method & 
$\Vert\mathcal{R}(U_h)\Vert_{\mcEh}$ & 
$\Vert\mathcal{R}(V_h)\Vert_{\mcEh}$ & 
$\Vert \mathbf{u}\cdot\mathbf{n} - U_h\Vert_{\mcFh}$ & 
$\Vert \mathbf{u}\cdot\mathbf{n} - V_h\Vert_{\mcFh}$ \\
\hline
CG(SD)  & 0.707 & 2.4e-16 & 0.354 & 9.7e-16 \\
CG(WD)  & 0.333 & 3.9e-17 & 0.118 & 1.2e-15 \\
CG(RD)  & 1.2e-15 & 1.2e-15 & 1.7e-15 & 1.7e-15 \\
\hline
\end{tabular} 
}

\subfloat[Nonuniform 1D grid.]{
\begin{minipage}{0.09\textwidth}
\includegraphics[width=\linewidth]{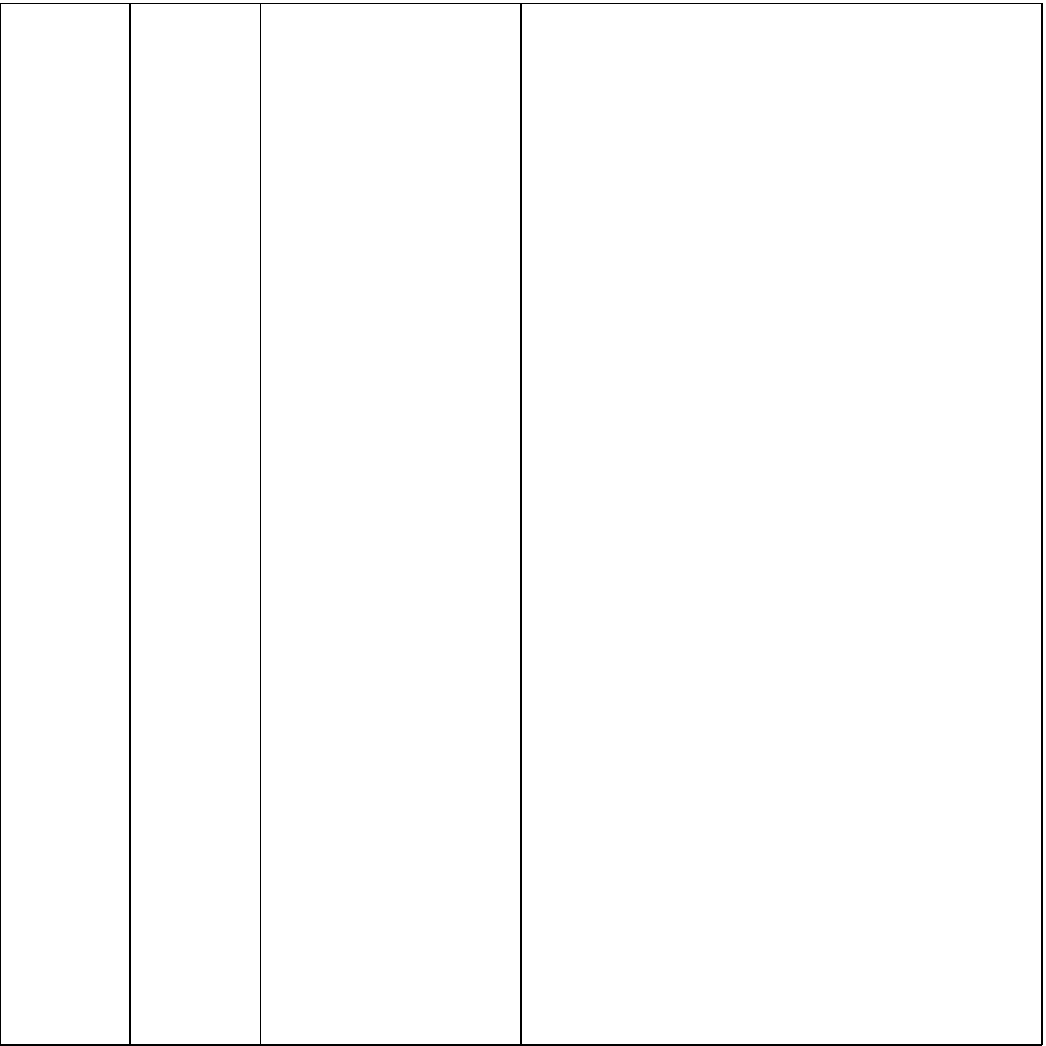}
\end{minipage}
\hspace*{2mm}
\begin{tabular}{lcccc}
\hline
Method & 
$\Vert\mathcal{R}(U_h)\Vert_{\mcEh}$ & 
$\Vert\mathcal{R}(V_h)\Vert_{\mcEh}$ & 
$\Vert \mathbf{u}\cdot\mathbf{n} - U_h\Vert_{\mcFh}$ & 
$\Vert \mathbf{u}\cdot\mathbf{n} - V_h\Vert_{\mcFh}$ \\
\hline
CG(SD)   & 0.976 & 2.9e-16 & 0.534 & 0.084 \\
CG(WD)   & 0.888 & 9.6e-17 & 0.265 & 0.168 \\
CG(RD)   & 0.280 & 6.3e-16 & 0.140 & 1.7e-15 \\
\hline
\end{tabular} 
}

\subfloat[Uniform 2D grid.]{
\begin{minipage}{0.09\textwidth}
\includegraphics[width=\linewidth]{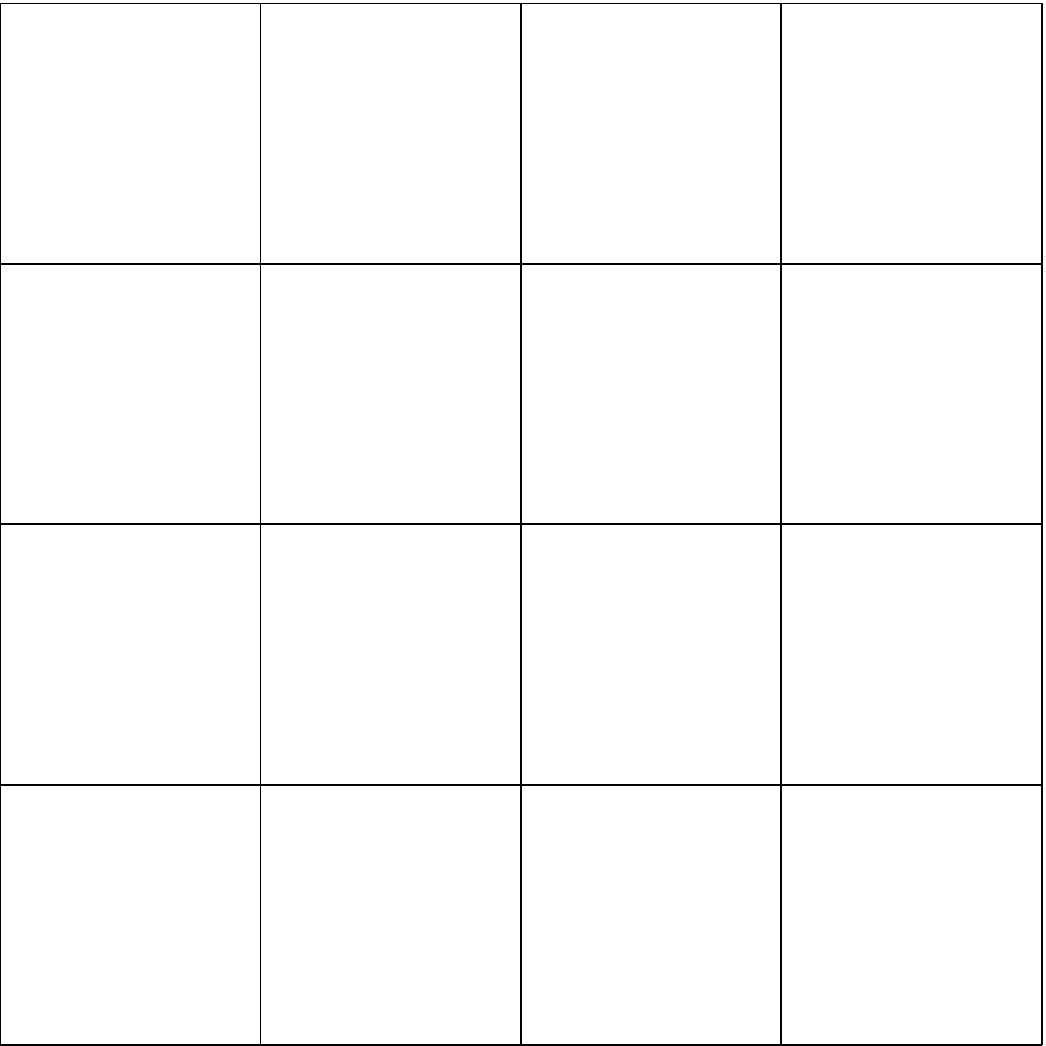}
\end{minipage}
\hspace*{2mm}
\begin{tabular}{lcccc}
\hline
Method & 
$\Vert\mathcal{R}(U_h)\Vert_{\mcEh}$ & 
$\Vert\mathcal{R}(V_h)\Vert_{\mcEh}$ & 
$\Vert \mathbf{u}\cdot\mathbf{n} - U_h\Vert_{\mcFh}$ & 
$\Vert \mathbf{u}\cdot\mathbf{n} - V_h\Vert_{\mcFh}$ \\
\hline
CG(SD)   & 0.707   & 4.6e-16 & 0.354 & 7.0e-16 \\
CG(WD)   & 0.056   & 3.0e-17 & 0.020 & 1.1e-14 \\
CG(RD)   & 2.8e-15 & 2.8e-15 & 1.4e-15 & 1.4e-15 \\
\hline
\end{tabular} 
}

\subfloat[Distorted 2D grid.]{
\begin{minipage}{0.09\textwidth}
\includegraphics[width=\linewidth]{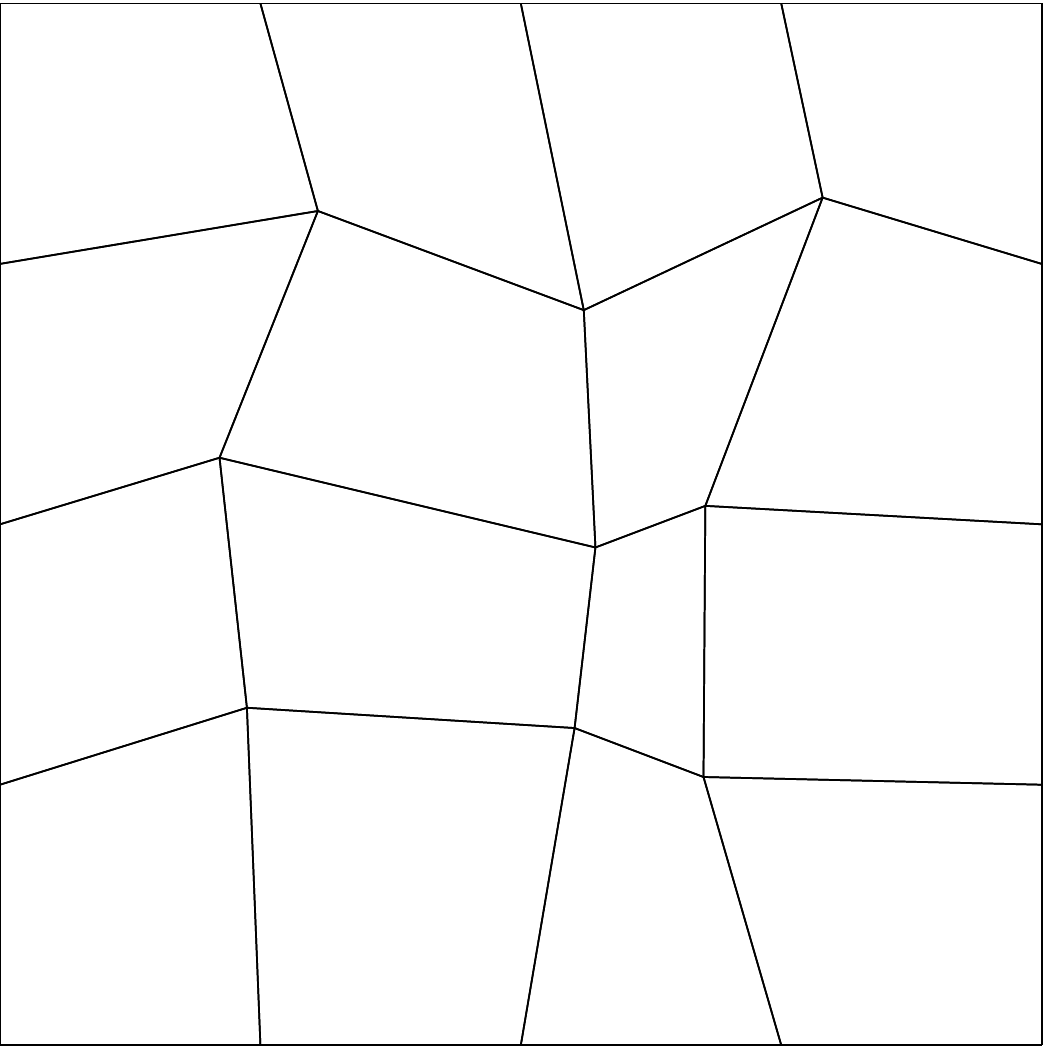}
\end{minipage}
\hspace*{2mm}
\begin{tabular}{lcccc}
\hline
Method & 
$\Vert\mathcal{R}(U_h)\Vert_{\mcEh}$ & 
$\Vert\mathcal{R}(V_h)\Vert_{\mcEh}$ & 
$\Vert \mathbf{u}\cdot\mathbf{n} - U_h\Vert_{\mcFh}$ & 
$\Vert \mathbf{u}\cdot\mathbf{n} - V_h\Vert_{\mcFh}$ \\
\hline
CG(SD)    & 0.908 & 1.4e-15 & 0.401 & 0.073 \\
CG(WD)    & 0.443 & 7.2e-17 & 0.122 & 0.078 \\
CG(RD)    & 0.462 & 6.0e-15 & 0.131 & 0.085 \\
\hline
\end{tabular} 
}

\subfloat[Nonmatching 2D grid.]{
\begin{minipage}{0.09\textwidth}
\includegraphics[width=\linewidth]{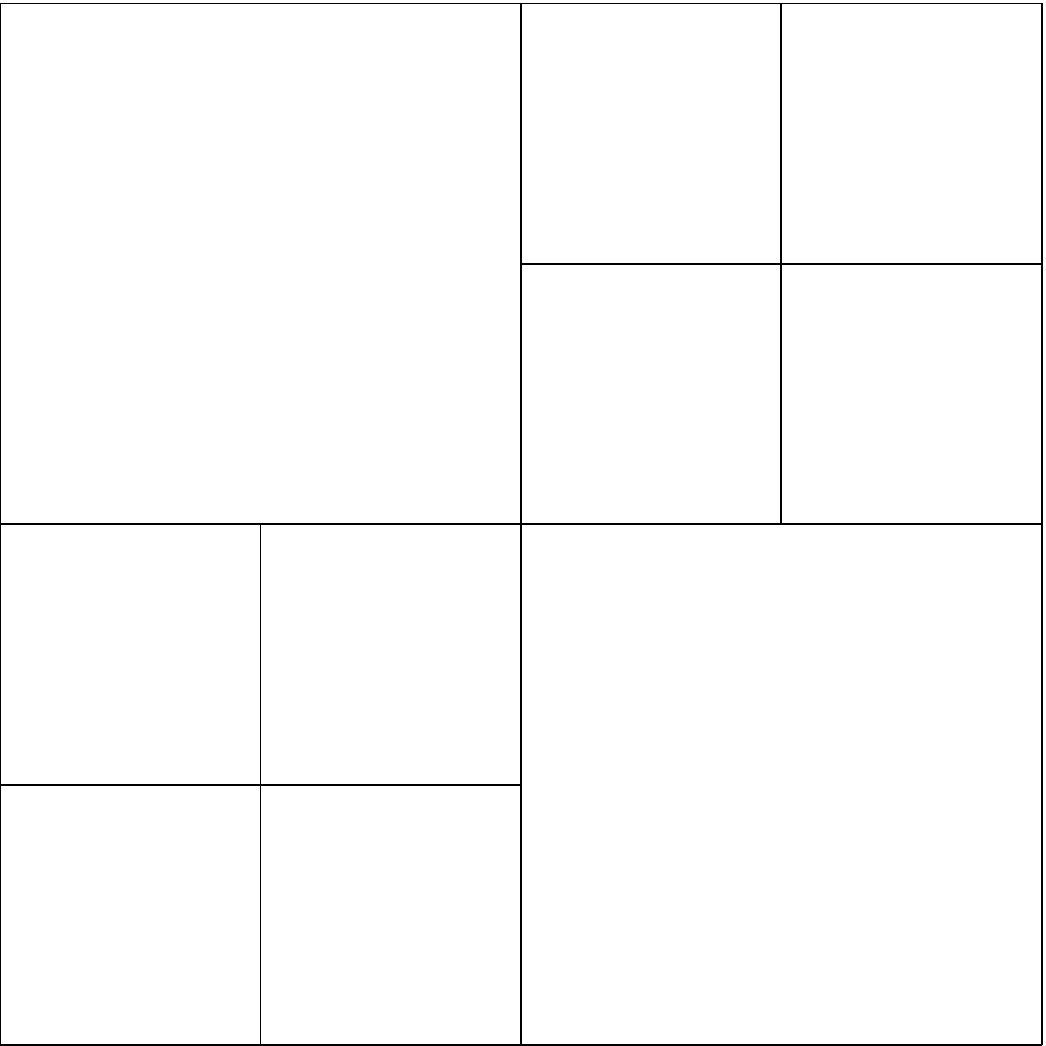}
\end{minipage}
\hspace*{2mm}
\begin{tabular}{lcccc}
\hline
Method & 
$\Vert\mathcal{R}(U_h)\Vert_{\mcEh}$ & 
$\Vert\mathcal{R}(V_h)\Vert_{\mcEh}$ & 
$\Vert \mathbf{u}\cdot\mathbf{n} - U_h\Vert_{\mcFh}$ & 
$\Vert \mathbf{u}\cdot\mathbf{n} - V_h\Vert_{\mcFh}$ \\
\hline
CG(SD)    & 1.127 & 5.5e-16 & 0.615 & 0.089 \\
CG(WD)    & 0.288 & 6.4e-17 & 0.183 & 0.163 \\
CG(RD)  & 0.278 & 1.3e-15 & 0.179 & 0.144 \\
\hline
\end{tabular} 
}

\end{table}

\begin{table}[bpt]
\caption{\textbf{Consistency tests}. Integrated flux along vertical mesh lines for different flux approximations and mesh lines, $\gamma_i$.}
\label{tab:gamma_flux}
\begin{center}
\footnotesize

\subfloat[Distorted 2D grid.]{
\label{tab:gamma_flux_d}
\begin{tabular}{|ll|ccccc|}
\hline
 & & $\gamma_1$ & $\gamma_2$ & $\gamma_3$ & $\gamma_4$ & $\gamma_5$ \\ 
\multicolumn{2}{|c|}{Flux} & $(x=0)$ & $(x\approx 0.25)$ & $(x\approx 0.5)$ & $(x\approx 0.75)$ & $(x=1)$ \\ \hline
Exact & $\int_{\gamma_i} \bfu\cdot\bfn$ & 0 & 0.4980 & 1.0850 & 1.4400 & 2 \\ \hline
\multirow{2}{*}{CG(SD)} & $\int_{\gamma_i} U_h$ & 0.2551  & 0.5198 & 1.0331 & 1.4919 & 1.7196 \\
                        & $\int_{\gamma_i} V_h$ & 0.0033  & 0.4947 & 1.0817 & 1.4367 & 1.9967 \\ \hline
\multirow{2}{*}{CG(WD)} & $\int_{\gamma_i} U_h$ & 4.2e-15 & 0.5334 & 1.0330 & 1.4785 & 2 \\
                        & $\int_{\gamma_i} V_h$ & -0.0020  & 0.5001 & 1.0870 & 1.4420 & 2.0020 \\ \hline 
\multirow{2}{*}{CG(RD)} & $\int_{\gamma_i} U_h$ & 1.6e-15 & 0.5198 & 1.0331 & 1.4919 & 2 \\
                        & $\int_{\gamma_i} V_h$ & 1.6e-15 & 0.4980 & 1.0850 & 1.4400 & 2 \\ \hline         
\end{tabular} }

\subfloat[Nonmatching 2D grid.]{
\label{tab:gamma_flux_nm}
\begin{tabular}{|ll|ccc|}
\hline
 & & $\gamma_1$ & $\gamma_2$ & $\gamma_3$ \\ 
\multicolumn{2}{|c|}{Flux} & $(x=0)$ & $(x=0.5)$ & $(x=1)$ \\ \hline
Exact & $\int_{\gamma_i} \bfu\cdot\bfn$ & 0 & 1 & 2 \\ \hline
\multirow{2}{*}{CG(SD)} & $\int_{\gamma_i} U_h$ & 0.4130  & 1 & 1.5870 \\
                       & $\int_{\gamma_i} V_h$ & 9.7e-17 & 1 & 2 \\ \hline
\multirow{2}{*}{CG(WD)} & $\int_{\gamma_i} U_h$ & 1.9e-15 & 1 & 2 \\
                       & $\int_{\gamma_i} V_h$ & 4.5e-15 & 1 & 2 \\ \hline 
\multirow{2}{*}{CG(RD)} & $\int_{\gamma_i} U_h$ & 8.7e-16 & 1 & 2 \\
                       & $\int_{\gamma_i} V_h$ & 8.7e-16 & 1 & 2 \\ \hline         
\end{tabular} }

\end{center}
\end{table}

\subsection{Convergence Tests}

To verify the convergence estimates in Eq.\ \eqref{eq:conv_estimate_pressure} and Lemma \ref{lemma:ee} numerically (objective (iv)), we consider a time dependent problem with analytic solution. Let $\Omega=(0,1)^2$, $\mathbf{K}= \mathbb{I}$, $\beta=1.0$ and $\phi= 1.0$. For the coupled flow and transport problem \eqref{eq:flow_eq}--\eqref{eq:transport_eq} we choose right hand sides and boundary conditions such that
\begin{align}
p=\cos(t + x-y),\qquad c=\cos(t + x-y)
\end{align}
are the analytic solutions. One may easily verify that $q=2\cos\alpha-\sin\alpha$ and $f=\left(1+4\sin\alpha\right)\cos\alpha$ with $\alpha=t + x-y$. For the flow problem, we impose Dirichlet conditions on $x=\{0,1\}$ and Neumann conditions on $y=\{0,1\}$. The numerical solution at $t=0.1$ on a fine grid can viewed in Fig.\ \ref{fig:conv_test_solution}.

First, the domain $\Omega$ is discretized into uniform quadratic grids of size $n\times n$ with $n=2^{i},\,i=2,3,4,5$. Equivalently, $h=\frac{1}{n}=2^{-i}$. The end time is $T=0.1$, and the time step size is chosen small enough to not effect the convergence rates and is recursively refined such that $\Delta t = \frac{1}{5\cdot 4^{i-1}} = \frac{4}{5}h^2$. The transport solver is run with three different flux approximations: (i) CG flux ($U_h$); (ii) postprocessed CG flux ($V_h$); and (iii) analytic flux ($\bfu\cdot\bfn$). Dirichlet conditions are imposed strongly, CG(SD). Convergence tables for flow and transport quantities are shown in Table \ref{tab:conv}.

\begin{figure}[bpt]
\centering
\subfloat[Pressure.]{
\includegraphics[height=0.4\textwidth]{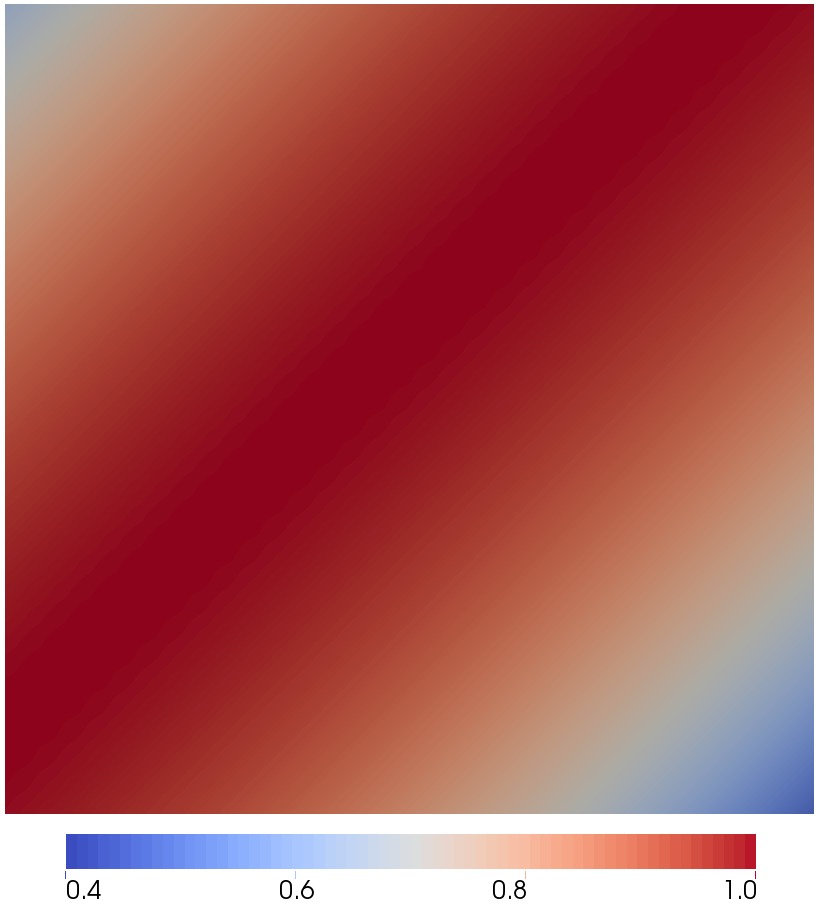}
}
\hspace*{3mm}
\subfloat[Concentration.]{
\includegraphics[height=0.4\textwidth]{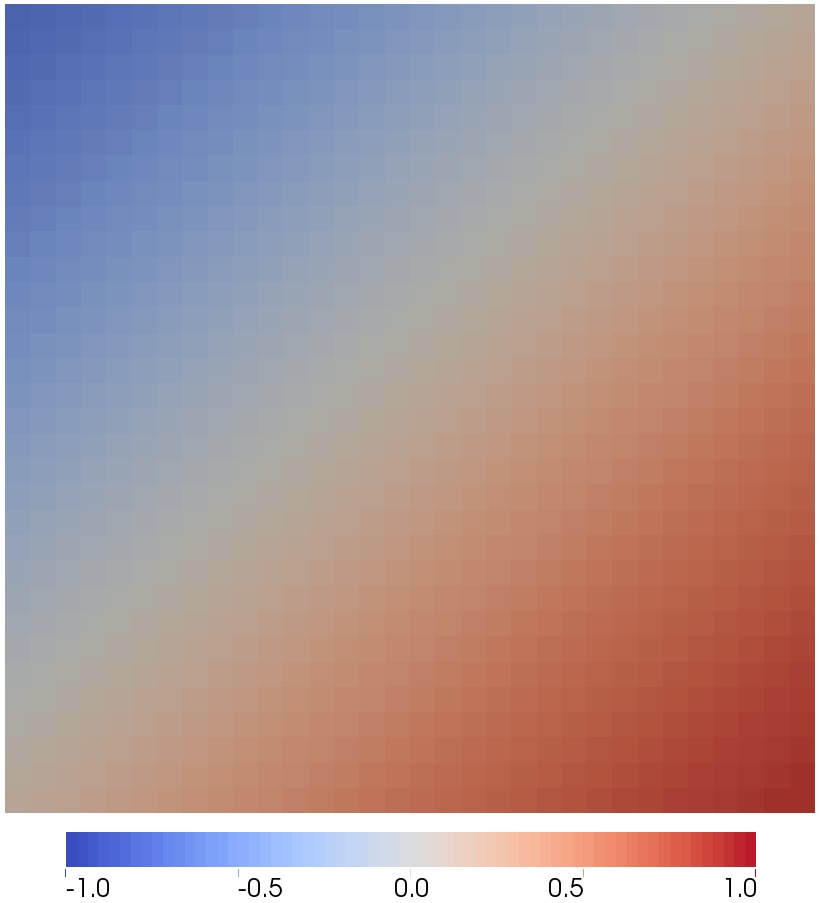}
}
\caption{\textbf{Convergence tests}. Pressure (a) and concentration (b) solution at $t = 0.1$ on the finest grid level, $1/h=32$. The postprocessed flux, $V_h$, is used in the transport solver to calculate the transport solution.}
\label{fig:conv_test_solution}
\end{figure}

We observe that the error in $p$ is of order 1 in the energy norm in accordance with the error estimate in Eq.\ \eqref{eq:conv_estimate_pressure}. Furthermore, we see that the postprocessed flux, $V_h$, converges with order $1/2$ larger than the CG flux, $U_h$. The residual, $\mathcal{R}(U_h)$, converges to zero with one order lower than $U_h$. These results are in accordance with Lemma \ref{lemma:ee}. Finally, we observe that the residual is zero (down to machine precision) for the postprocessed flux.

For the concentration solution, all simulations converge with order 1. The differences in concentration due to different flux calculations are small in this example. However, we show later that cases involving heterogeneous permeability may result in much larger differences.

Next, the same examples were run but with Dirichlet flux recovery, CG(RD). The convergence table for flow and transport variables are displayed in Table \ref{tab:conv_dfr}. We see that the order of the error in $U_h$ increases by $1/2$ compared to CG(SD), while the residual now converges to zero with rate $1.5$. This appears to be due to better flux approximation on the Dirichlet boundary as this is the only difference. The postprocessed flux has the same order as the CG flux, so the net effect is nearly the same as without Dirichlet flux recovery (cf.\ Table \ref{tab:conv_flow}). In the remaining examples of this work, we will therefore only consider strong Dirichlet conditions, CG(SD).

\begin{table}[bpt]
\setlength{\tabcolsep}{5pt}
\caption{\textbf{Convergence tests}. Error and convergence rates for flow variables (a) and concentration solution (b). A recursively refined quadratic grid with element size $h$ is used. Dirichlet boundary conditions are imposed strongly, CG(SD).}
\label{tab:conv}
\begin{center}
\footnotesize

\subfloat[Flow variables.]{
\label{tab:conv_flow}
\begin{tabular}{|c|cc|cc|cc|cc|c|}
\hline
$1/h$ &
\multicolumn{2}{c|}{$\Vert p-p_h\Vert_a$} & 
\multicolumn{2}{c|}{$\Vert \mathbf{u}\cdot\mathbf{n}-U_h\Vert_{h,\mcFh}$} & 
\multicolumn{2}{c|}{$\Vert \mathbf{u}\cdot\mathbf{n}-V_h\Vert_{h,\mcFh}$} & 
\multicolumn{2}{c|}{$\Vert \mathcal{R}(U_h)\Vert_{\mcEh}$} & 
\multicolumn{1}{c|}{$\Vert \mathcal{R}(V_h)\Vert_{\mcEh}$} \\ \hline
4 & 0.0941 & - & 0.08211 & - & 0.00809 & - & 0.3000 & - & 2.8e-16\\ 
8 & 0.0470 & 1.00 & 0.02804 & 1.55 & 0.00197 & 2.04 & 0.2125 & 0.50 & 3.1e-16 \\ 
16 & 0.0235 & 1.00 & 0.00967 & 1.54 & 0.00049 & 2.00 & 0.1503 & 0.50 & 4.0e-16\\ 
32 & 0.0117 & 1.00 & 0.00337 & 1.52 & 0.00012 & 2.00 & 0.1063 & 0.50 & 5.2e-16\\ 
\hline
\end{tabular} }

\subfloat[Concentration solution with different flux (in parenthesis).]{
\label{tab:conv_transport}
\begin{tabular}{|c|cc|cc|cc|} 
\hline
$1/h$ & 
\multicolumn{2}{c|}{$\Vert c-c_h\Vert_{\mcEh}\ (\mathbf{u}\cdot\mathbf{n})$} & 
\multicolumn{2}{c|}{$\Vert c-c_h\Vert_{\mcEh}\ (U_h)$} & 
\multicolumn{2}{c|}{$\Vert c-c_h\Vert_{\mcEh}\ (V_h)$} \\ \hline
4 &  0.09502 & - & 0.09631 & - & 0.09507 & - \\ 
8 &  0.04765 & 1.00 & 0.04850 & 0.99 & 0.04766 & 1.00 \\ 
16 & 0.02385 & 1.00 & 0.02436 & 0.99 & 0.02385 & 1.00 \\ 
32 & 0.01193 & 1.00 & 0.01218 & 1.00 & 0.01193 & 1.00 \\ 
\hline
\end{tabular} }

\end{center}
\end{table}

\begin{table}[bpt]
\setlength{\tabcolsep}{5pt}
\caption{\textbf{Convergence tests}. Error and convergences rates for flow variables (a) and concentration solution (b). A recursively refined quadratic grid with element size $h$ is used. The Dirichlet flux recovery technique, CG(RD), is used.}
\label{tab:conv_dfr}
\begin{center}
\footnotesize

\subfloat[Flow variables.]{
\label{tab:conv_flow_dfr}
\begin{tabular}{|c|cc|cc|cc|cc|c|}
\hline
$1/h$ &
\multicolumn{2}{c|}{$\Vert p-p_h\Vert_a$} & 
\multicolumn{2}{c|}{$\Vert \mathbf{u}\cdot\mathbf{n}-U_h\Vert_{h,\mcFh}$} & 
\multicolumn{2}{c|}{$\Vert \mathbf{u}\cdot\mathbf{n}-V_h\Vert_{h,\mcFh}$} & 
\multicolumn{2}{c|}{$\Vert \mathcal{R}(U_h)\Vert_{\mcEh}$} & 
\multicolumn{1}{c|}{$\Vert \mathcal{R}(V_h)\Vert_{\mcEh}$} \\ \hline
4  & 0.0941 & -    & 0.00965 & -    & 0.00719 & -    & 0.0207 & -    & 1.6e-13\\ 
8  & 0.0470 & 1.00 & 0.00250 & 1.95 & 0.00160 & 2.17 & 0.0077 & 1.42 & 5.1e-13\\ 
16 & 0.0235 & 1.00 & 0.00064 & 1.97 & 0.00037 & 2.11 & 0.0028 & 1.49 & 7.7e-13\\ 
32 & 0.0117 & 1.00 & 0.00016 & 1.99 & 0.00009 & 2.06 & 0.0010 & 1.50 & 1.4e-12\\ 
\hline
\end{tabular} }

\subfloat[Concentration solution with different flux (in parenthesis).]{
\label{tab:conv_transport_dfr}
\begin{tabular}{|c|cc|cc|cc|} 
\hline
$1/h$ & 
\multicolumn{2}{c|}{$\Vert c-c_h\Vert_{\mcEh}\ (\mathbf{u}\cdot\mathbf{n})$} & 
\multicolumn{2}{c|}{$\Vert c-c_h\Vert_{\mcEh}\ (U_h)$} & 
\multicolumn{2}{c|}{$\Vert c-c_h\Vert_{\mcEh}\ (V_h)$} \\ \hline
4 & 0.09503 & - & 0.09514 & - & 0.09510 & - \\ 
8 & 0.04765 & 1.00 & 0.04769 & 1.00 & 0.04767 & 1.00 \\ 
16 & 0.02385 & 1.00 & 0.02386 & 1.00 & 0.02386 & 1.00 \\ 
32 & 0.01193 & 1.00 & 0.01194 & 1.00 & 0.01193 & 1.00 \\
\hline
\end{tabular} }

\end{center}
\end{table}

At last, we consider the same problem but evaluate convergence on a family of distorted and non-conforming grid. Let $\mathcal{M}_0$ be the base grid as displayed in Fig.\ \ref{fig:grid_refinement_1}. Then, we iteratively refine the base grid globally by dividing each element into four by connecting midpoints of the four faces. This results in a family of refined grids, $\mathcal{M}_i,\ i=0,1,\ldots,4$, where the three first grids are displayed in Fig.~\ref{fig:grid_refinement}. The time steps are now $\Delta t = \frac{1}{5\cdot 4^{i+1}}$. Convergence results are shown in Table \ref{tab:conv_grid}. 
We still observe that the order of $V_h$ is the same as for $U_h$, although we have to let $h$ be very small for the rate to converge towards 1. Notice that $\Vert \mathbf{u}\cdot\mathbf{n}-V_h\Vert_{h,\mcFh} < \Vert \mathbf{u}\cdot\mathbf{n}-U_h\Vert_{h,\mcFh}$ for all cases studied in this section. This example demonstrates that our method works and that the error estimates hold for general grids.

\begin{figure}[bpt]
\centering
\subfloat[$\mathcal{M}_0$.]{
\includegraphics[width=0.25\textwidth]{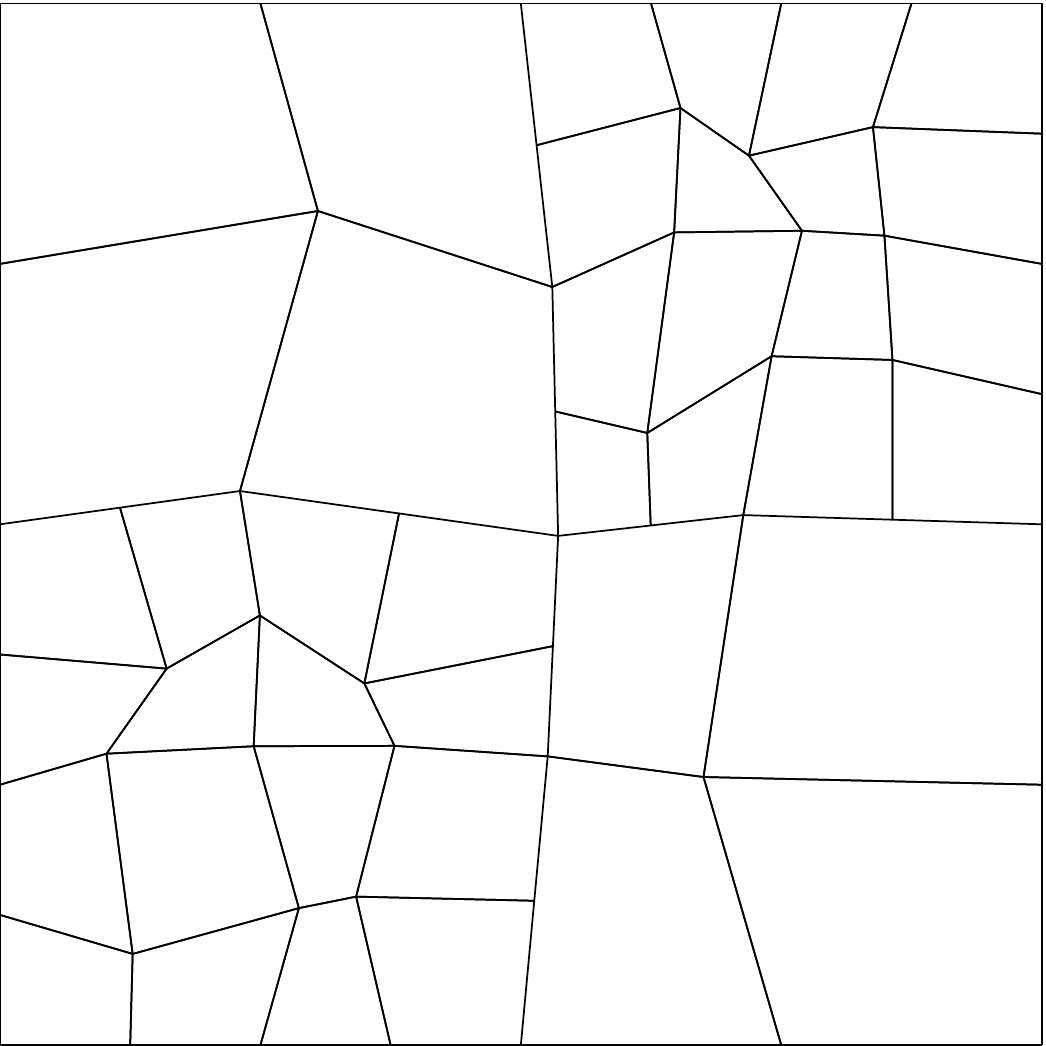}
\label{fig:grid_refinement_1}
}
\hspace*{3mm}
\subfloat[$\mathcal{M}_1$.]{
\includegraphics[width=0.25\textwidth]{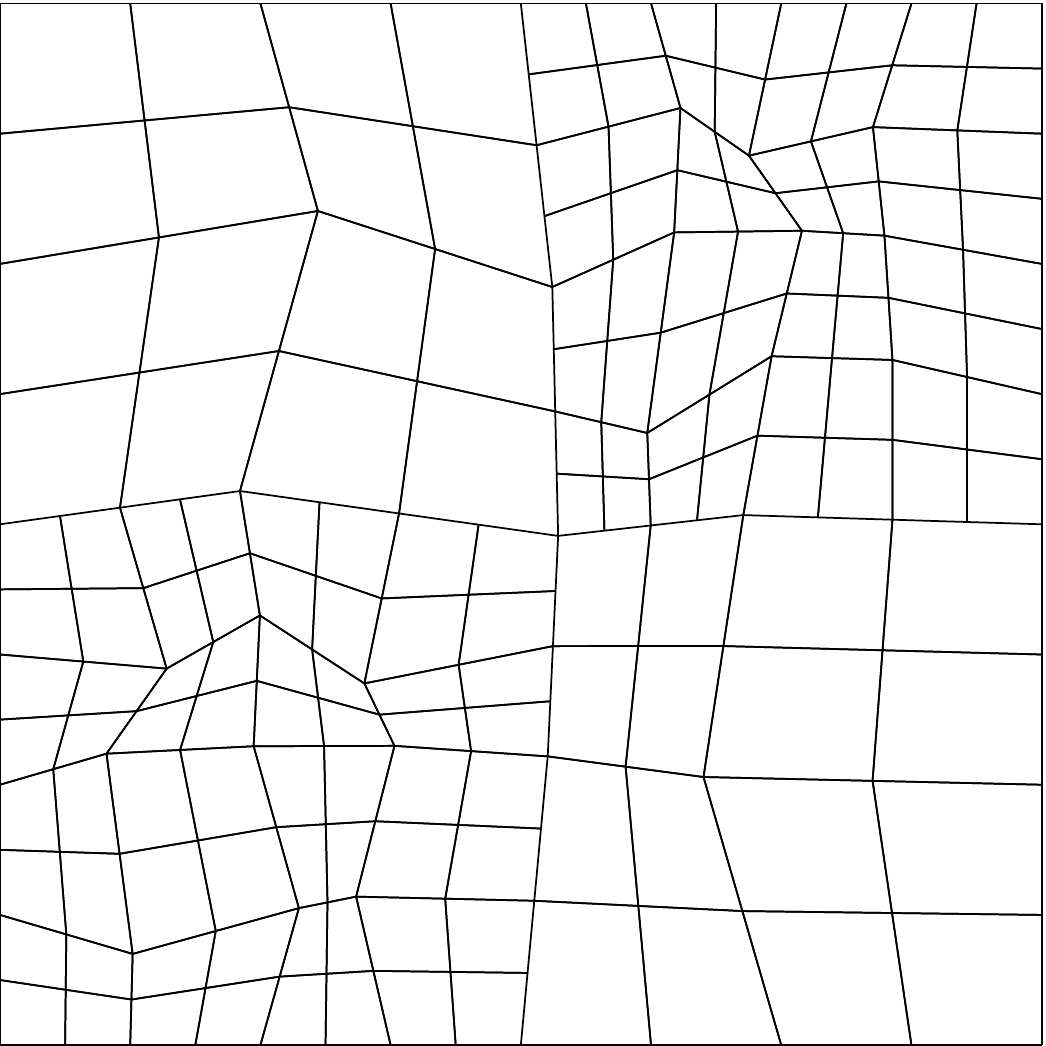}
}
\hspace*{3mm}
\subfloat[$\mathcal{M}_2$.]{
\includegraphics[width=0.25\textwidth]{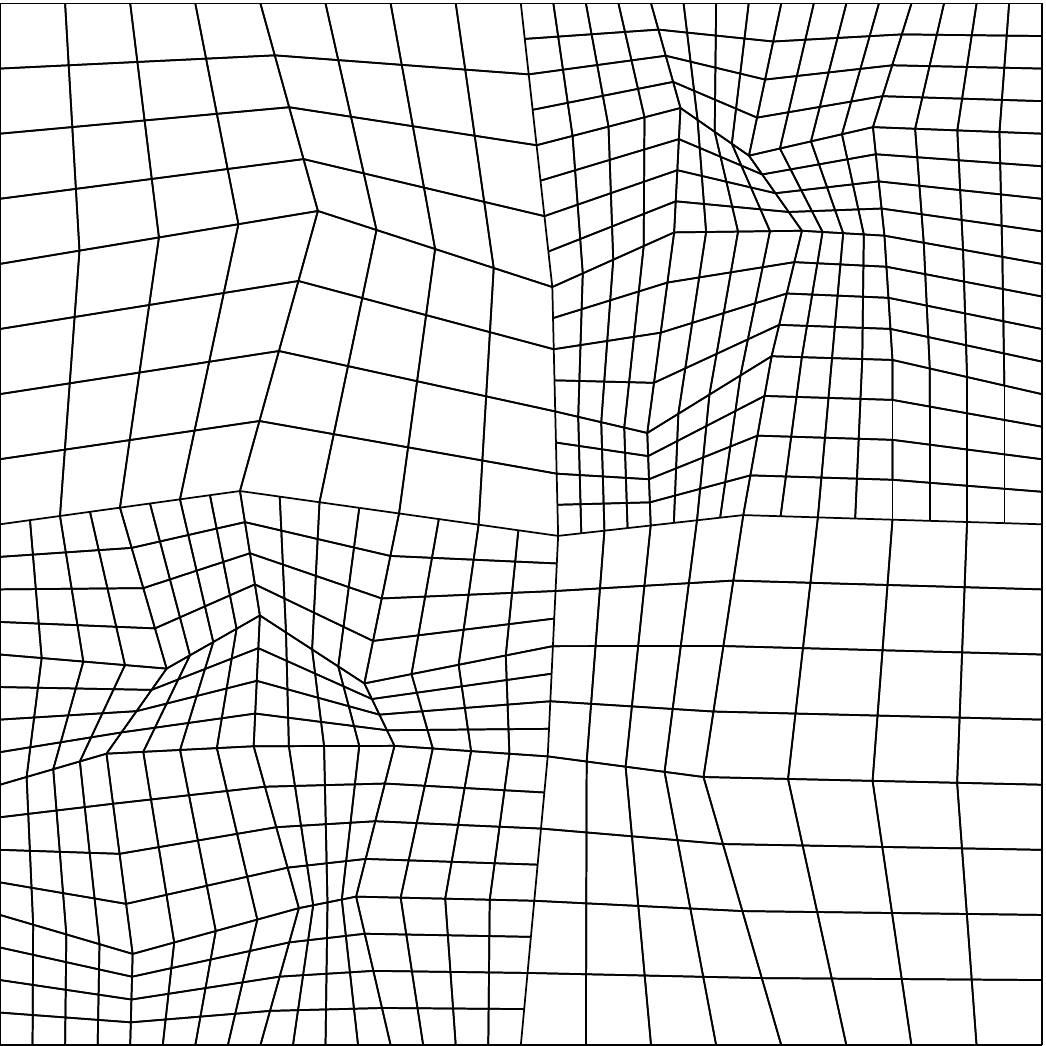}
}
\caption{\textbf{Convergence tests}. Base grid (left) and the first two recursively refined grids used for convergence test for distorted and non-conforming grids. All cells are divided in four in each refinement cycle.}
\label{fig:grid_refinement}
\end{figure}

\begin{table}[bpt]
\setlength{\tabcolsep}{5pt}
\caption{\textbf{Convergence tests}. Error and convergences rates for flow variables (a) and concentration solution (b) for the recursively refined grids shown in Fig.~\ref{fig:grid_refinement}. Dirichlet boundary conditions are imposed strongly, CG(SD).}
\label{tab:conv_grid}
\begin{center}
\footnotesize

\subfloat[Flow variables.]{
\label{tab:conv_flow_grid}
\begin{tabular}{|c|cc|cc|cc|cc|c|}
\hline
Grid & 
\multicolumn{2}{c|}{$\Vert p-p_h\Vert_a$} & 
\multicolumn{2}{c|}{$\Vert \mathbf{u}\cdot\mathbf{n}-U_h\Vert_{h,\mcFh}$} & 
\multicolumn{2}{c|}{$\Vert \mathbf{u}\cdot\mathbf{n}-V_h\Vert_{h,\mcFh}$} & 
\multicolumn{2}{c|}{$\Vert \mathcal{R}(U_h)\Vert_{\mcEh}$} &
$\Vert \mathcal{R}(V_h)\Vert_{\mcEh}$
\\ \hline
$\mathcal{M}_0$ & 0.08331 & -    & 0.116103 & -    & 0.044114 & -    & 0.5096 & -    & 7.3e-16\\ 
$\mathcal{M}_1$ & 0.04057 & 1.04 & 0.041738 & 1.48 & 0.017211 & 1.36 & 0.3153 & 0.69 & 7.1e-16\\
$\mathcal{M}_2$ & 0.02006 & 1.02 & 0.015556 & 1.42 & 0.007820 & 1.14 & 0.2233 & 0.50 & 6.9e-16\\
$\mathcal{M}_3$ & 0.00998 & 1.01 & 0.006116 & 1.35 & 0.003840 & 1.03 & 0.1585 & 0.49 & 1.0e-15\\ 
$\mathcal{M}_4$ & 0.00498 & 1.00 & 0.002556 & 1.26 & 0.001921 & 1.00 & 0.1122 & 0.50 & 1.4e-15\\ 
$\mathcal{M}_5$ & 0.00248 & 1.00 & 0.001133 & 1.17 & 0.000963 & 1.00 & 0.0793 & 0.50 & 2.4e-15\\ 
$\mathcal{M}_6$ & 0.00124 & 1.00 & 0.000527 & 1.10 & 0.000483 & 1.00 & 0.0560 & 0.50 & 4.6e-15\\
\hline
\end{tabular} }

\subfloat[Concentration solution with different flux (in parenthesis).]{
\label{tab:conv_transport_grid}
\begin{tabular}{|c|cc|cc|cc|} 
\hline
Grid & 
\multicolumn{2}{c|}{$\Vert c-c_h\Vert_{\mcEh}\ (\mathbf{u}\cdot\mathbf{n})$} & 
\multicolumn{2}{c|}{$\Vert c-c_h\Vert_{\mcEh}\ (U_h)$} & 
\multicolumn{2}{c|}{$\Vert c-c_h\Vert_{\mcEh}\ (V_h)$} \\ \hline
$\mathcal{M}_0$ & 0.07595 & - & 0.07864 & - & 0.07598 & - \\ 
$\mathcal{M}_1$ & 0.03821 & 0.99 & 0.03977 & 0.98 & 0.03822 & 0.99 \\ 
$\mathcal{M}_2$ & 0.01919 & 0.99 & 0.02007 & 0.99 & 0.01920 & 0.99 \\ 
$\mathcal{M}_3$ & 0.00964 & 0.99 & 0.01006 & 1.00 & 0.00964 & 0.99 \\ 
\hline
\end{tabular} }

\end{center}
\end{table}

\subsection{Barrier Problem}

In the next example we consider flow and transport through a barrier (low permeability region) and study the objectives (i), (v), (vi) and (vii). The problem is illustrated in Fig.~\ref{fig:barrier_problem}. Let $\Omega=(0,1)^2$, $\beta=0$ and use boundary conditions $p(0,y) = 1$, $p(1,y)=0$ and $\mathbf{u}\cdot\mathbf{n} = 0$ on $y=\{0,1\}$. For the transport problem, $\phi=1$,  $\Gamma_{\text{in}}=\{x=0\}\cap\partial\Omega$, $c_B=1$ and $c_0= 0$. The steady state pressure and velocity solution from the CG scheme on a fine grid is shown in Fig.~\ref{fig:pressure_barrier_problem}.

\begin{figure}[bpt]
\centering
\subfloat[Problem definition. Geometry, boundary conditions and permeability distribution. The permeability is given as $\mathbf{K}=k\mathbb{I}$.] {
\includegraphics[width=0.36\textwidth]{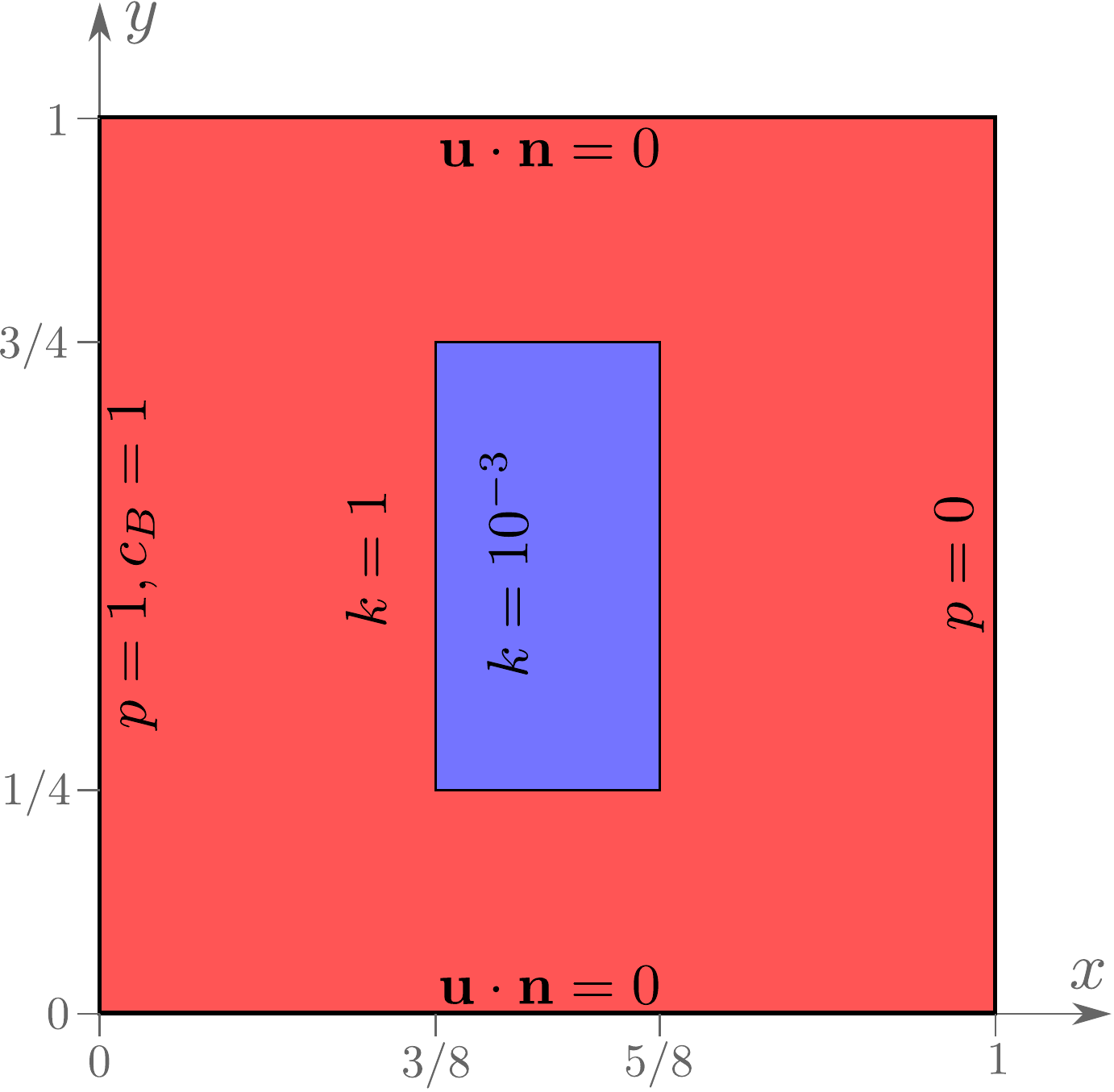}
}
\hspace*{5mm}
\subfloat[Steady state pressure and velocity solution from the CG scheme on a quadratic grid with $h=1/32$.] {
\includegraphics[width=0.4\textwidth]{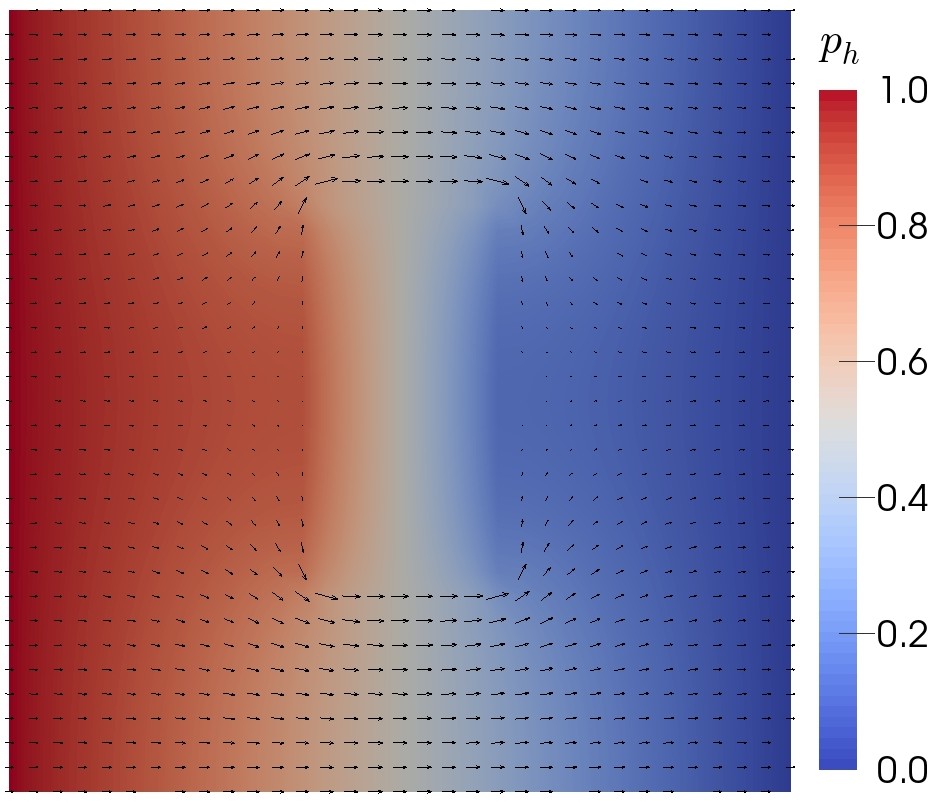}
\label{fig:pressure_barrier_problem}
}
\caption{\textbf{Barrier problem}. Problem definition (a) and numerical pressure solution (b).}
\label{fig:barrier_problem}
\end{figure}

First, consider the case when the standard average $\theta=1/2$ is used for flux calculations. The concentration solution with $\Delta t=0.01$ at $t=1$ and $t=2$ is shown in Fig.~\ref{fig:concentration_barrier_problem}, both for CG(SD,1/2) and PP(SD,1/2,L2). Furthermore, the concentration along the curve $y=0.735$ is plotted in the same figure. The solutions are close at $t=1$, although we observe some small unphysical oscillation close to the barrier interface for CG(SD,1/2). Both solutions are in the (physical) valid range $[0,1]$. However, at $t=2$, CG(SD,1/2) gives an unphysical solution as $c>1.0$ in some cells and since the solution oscillates close to the barrier interface. The solution with PP(SD,1/2,L2) is in the range $[0,1]$ and without oscillations. 

\begin{figure}[bpt]
\centering
\subfloat[CG(SD,1/2), $t=1$.]{
\includegraphics[width=0.3\textwidth]{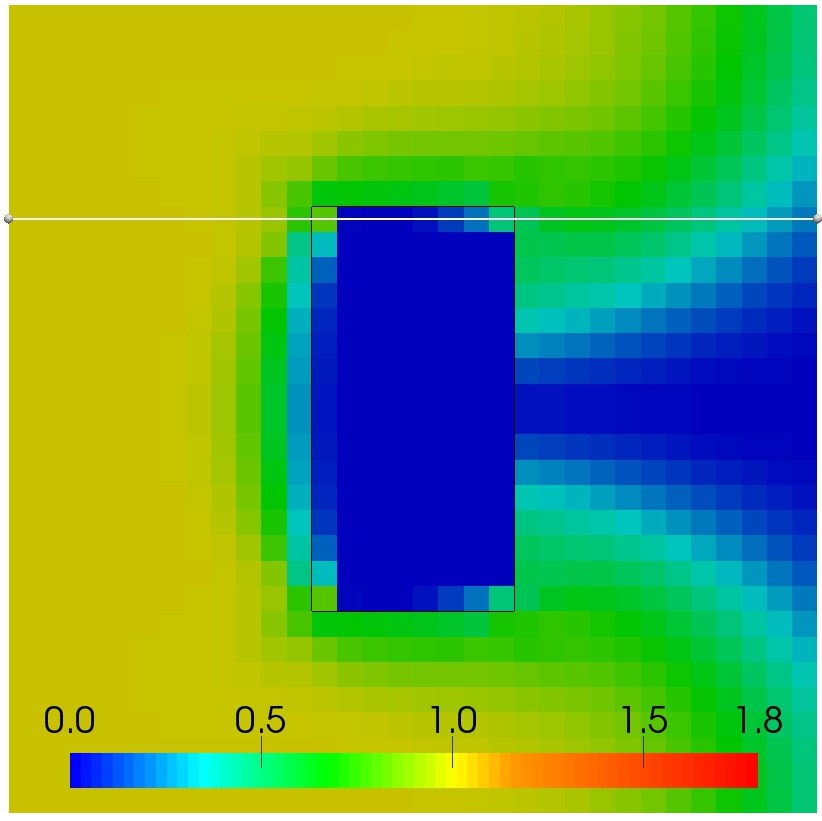}
}
\subfloat[PP(SD,1/2,L2), $t=1$.]{
\includegraphics[width=0.3\textwidth]{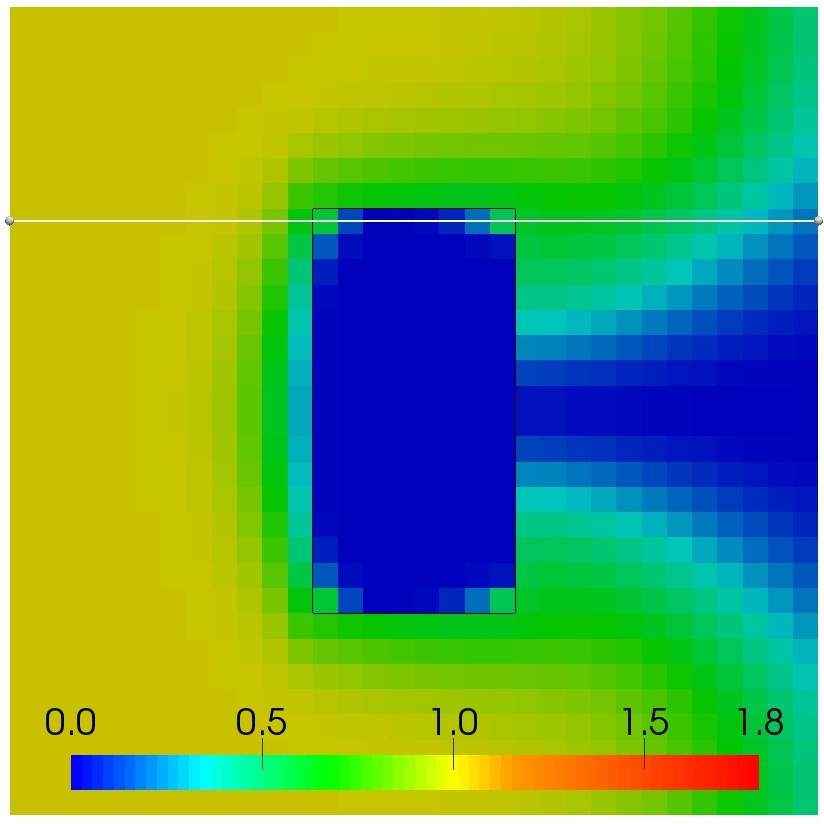}
}
\subfloat[Concentration over line, $t=1$.]{
\includegraphics[width=0.37\textwidth]{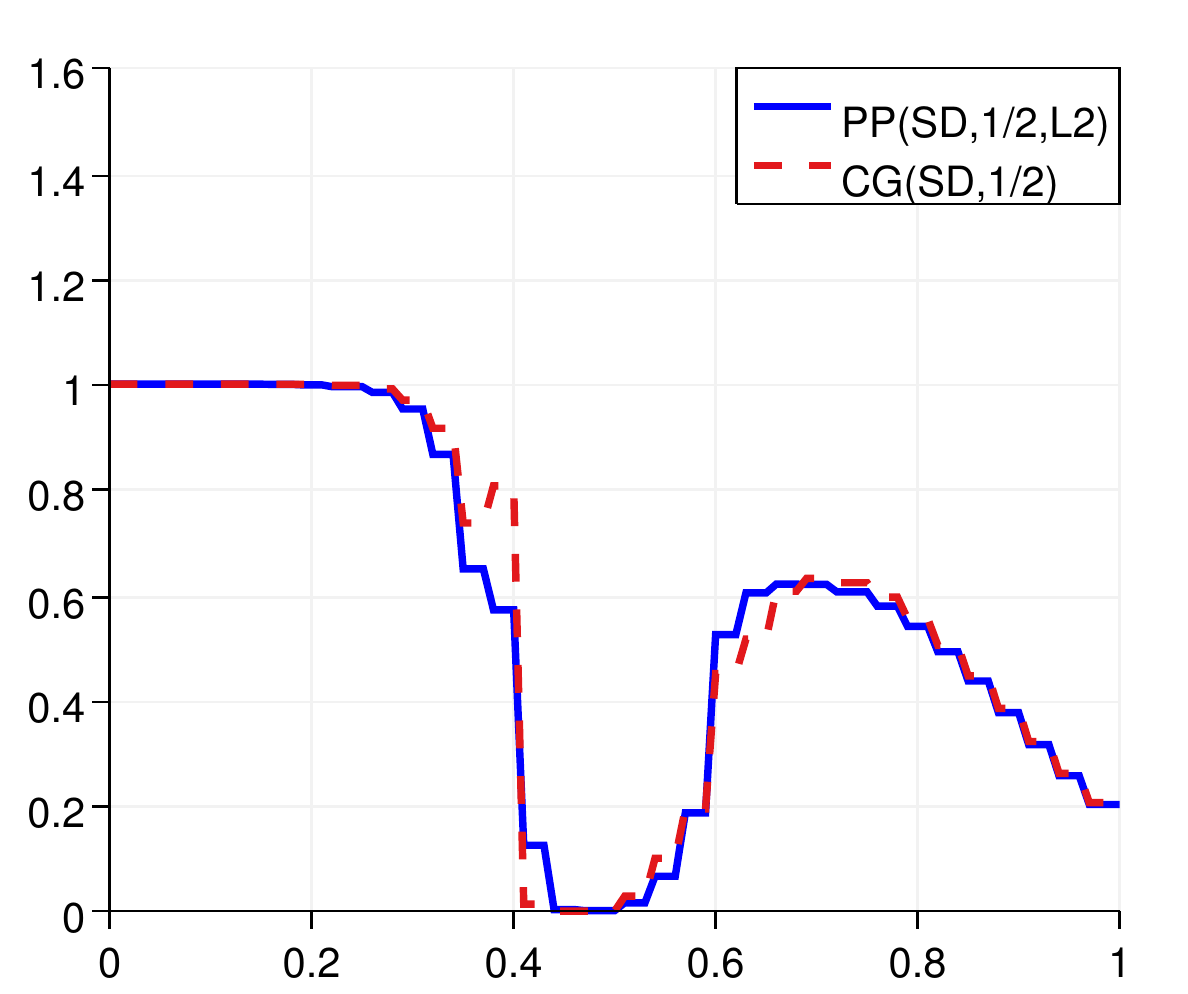}
} \\
\subfloat[CG(SD,1/2), $t=2$.]{
\includegraphics[width=0.3\textwidth]{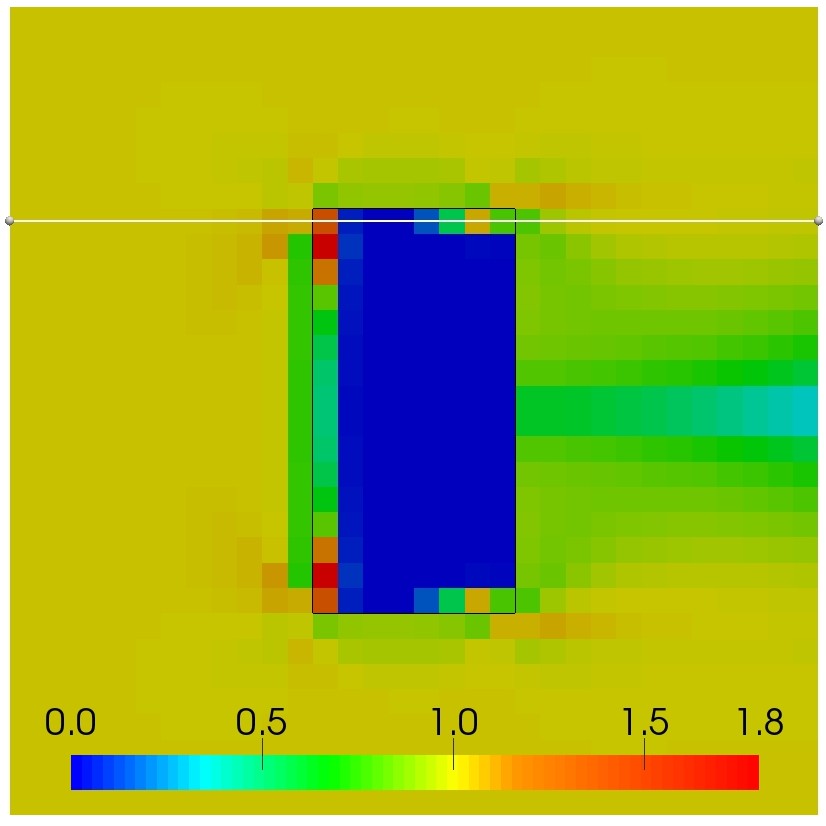}
}
\subfloat[PP(SD,1/2,L2), $t=2$.]{
\includegraphics[width=0.3\textwidth]{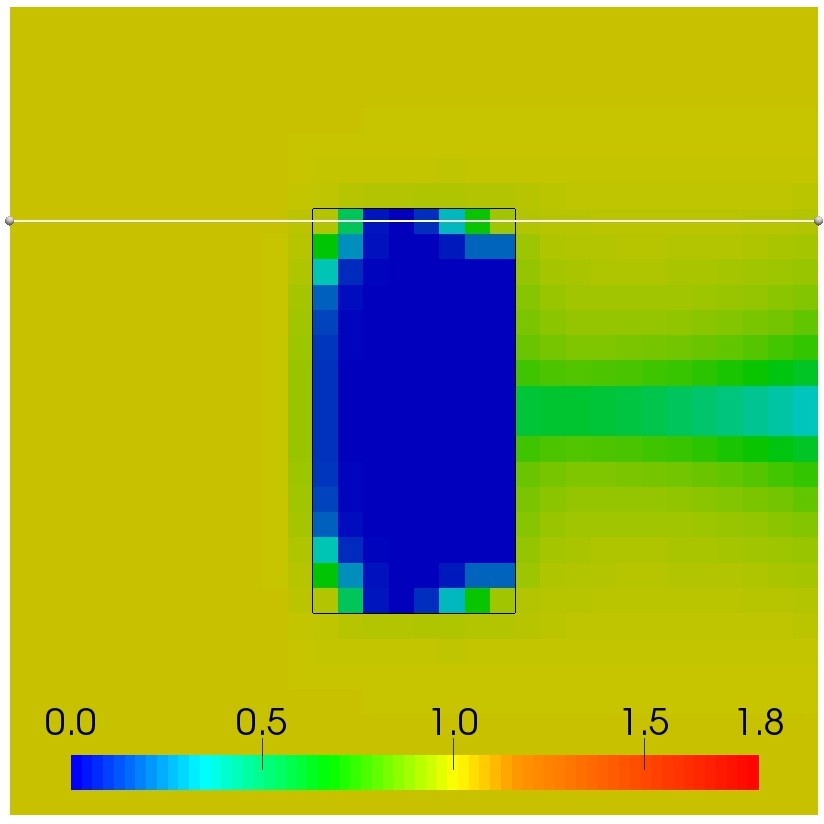}
}
\subfloat[Concentration over line, $t=2$.]{
\includegraphics[width=0.37\textwidth]{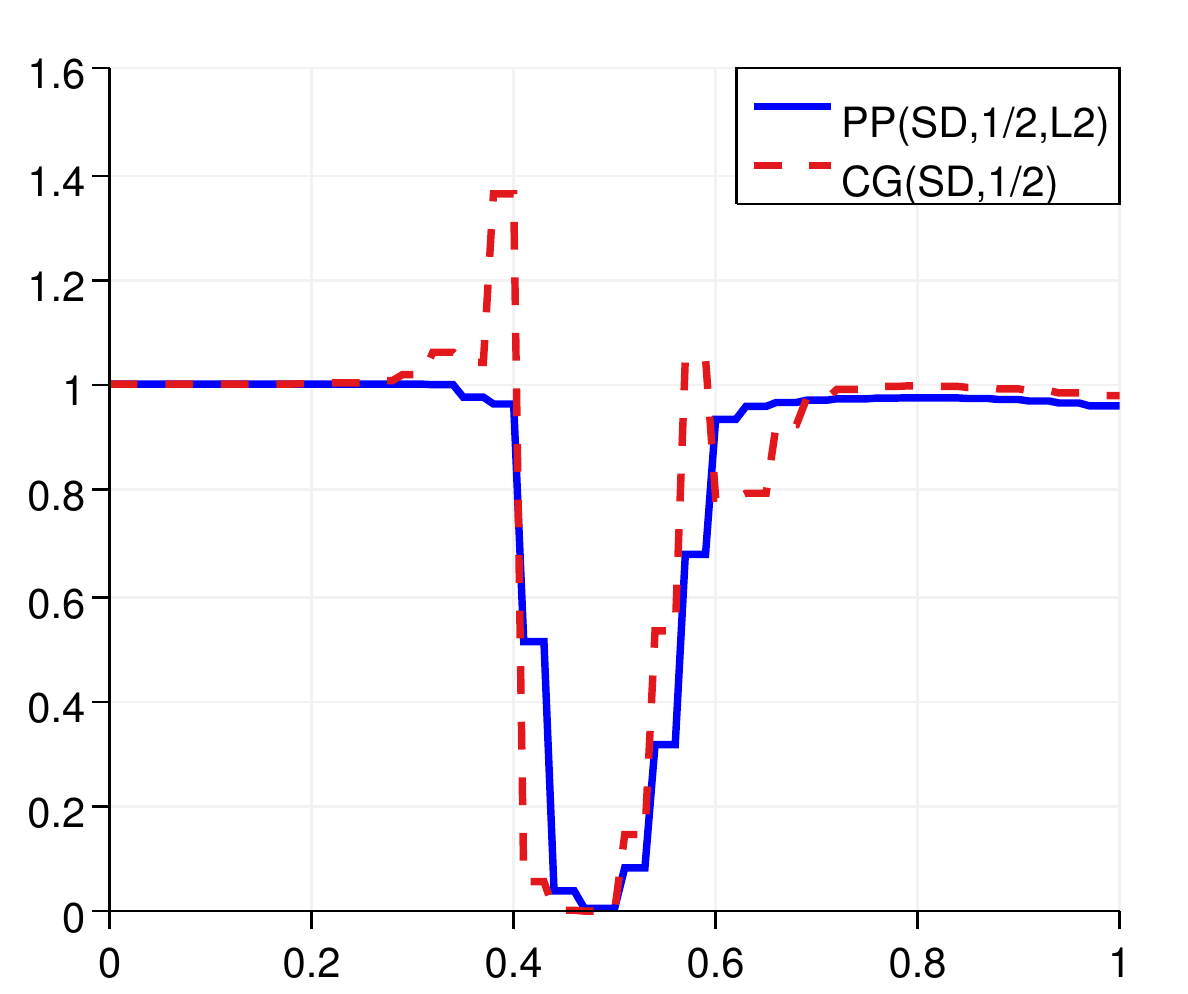}
}
\caption{\textbf{Barrier problem}. Concentration solution at two different times, with and without postprocessing. The standard weights $\theta=1/2$ is used for the average in calculations of CG flux, $U_h$. The solution along the white line ($y=0.735$) is plotted to the right. The low permeability region in inscribed in the black box.}
\label{fig:concentration_barrier_problem}
\end{figure}

Since the contrast in permeability is three orders of magnitude, we would expect very little flow into the barrier region. However, we see from Fig.~\ref{fig:concentration_barrier_problem} that the concentration in the corners of the barrier region is rather large. To cope with this we use harmonic averaging of the permeability, thus set $\theta=\vartheta$ in the flux averaging. Similar results as with $\theta=1/2$ are displayed in Fig.~\ref{fig:concentration_barrier_problem_harm}. Clearly, harmonic averaging reduces the inflow into the barrier region when we use CG flux, but still we get an unphysical solution (Fig.~\ref{fig:concentration_barrier_problem_harm_cg1} and \ref{fig:concentration_barrier_problem_harm_cg2}). However, when we postprocess this flux with minimization in the standard $L^2$ norm, the effect of harmonic averaging reduces since the concentration in the corners is now high (Fig.~\ref{fig:concentration_barrier_problem_harm_pp1} and \ref{fig:concentration_barrier_problem_harm_pp2}). If we instead postprocess with minimization in the weighted $L^2$ norm, we see that the barrier region is much less permeable  (Fig.~\ref{fig:concentration_barrier_problem_harm_pp1} and \ref{fig:concentration_barrier_problem_harm_pp2}). This clearly demonstrates that using the weighted $L^2$ norm is necessary to preserve low permeable interfaces and should be used in combination with harmonic averaging of the CG flux.

\begin{figure}[bpt]
\centering
\subfloat[CG(SD,$\vartheta$), $t=1$.]{
\includegraphics[width=0.3\textwidth]{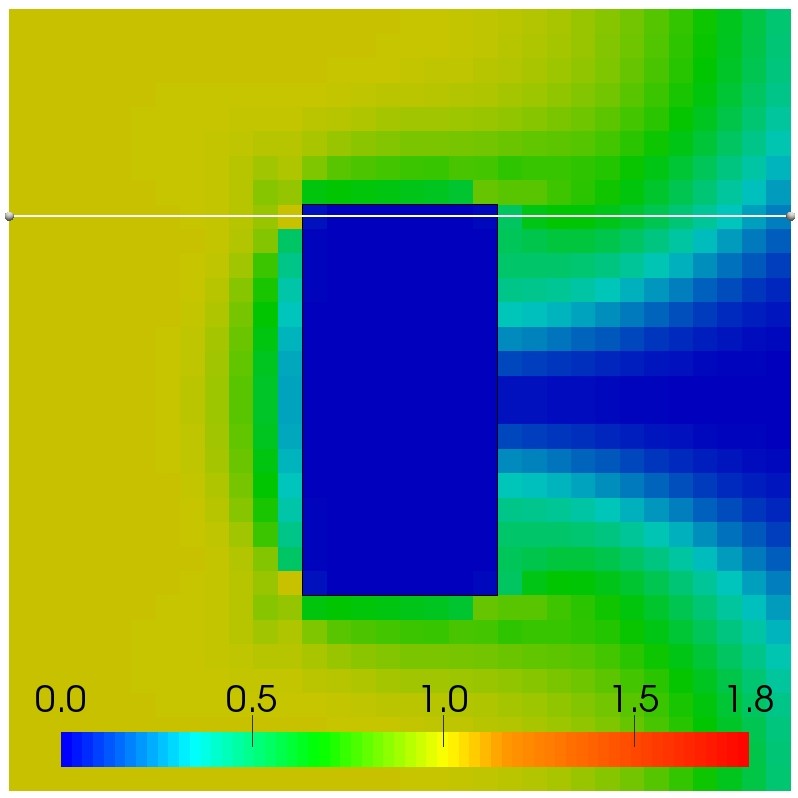}
\label{fig:concentration_barrier_problem_harm_cg1}
}
\subfloat[PP(SD,$\vartheta$,L2), $t=1$.]{
\includegraphics[width=0.3\textwidth]{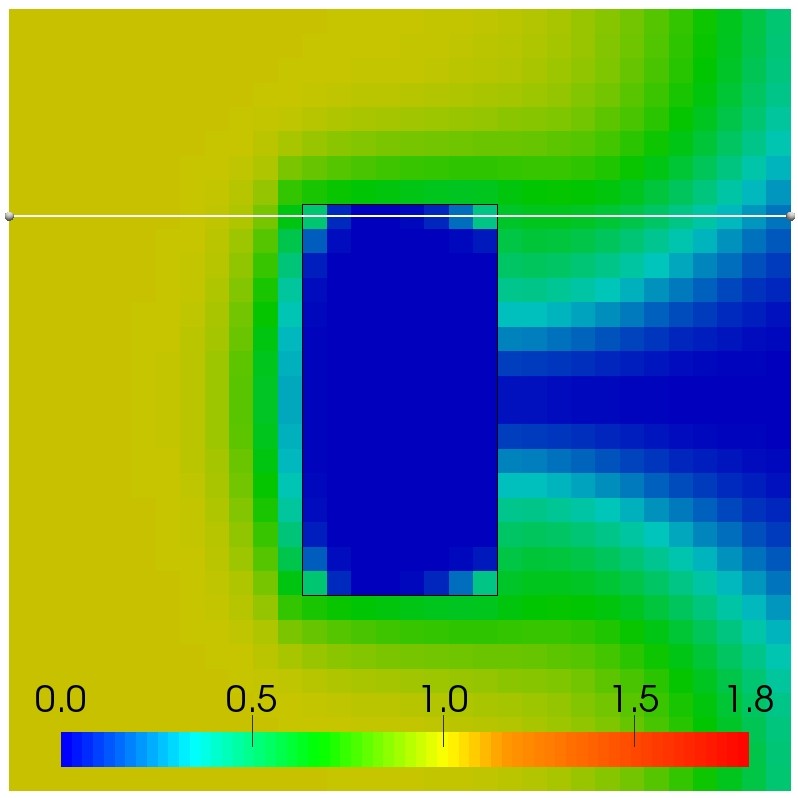}
\label{fig:concentration_barrier_problem_harm_pp1}
}
\subfloat[PP(SD,$\vartheta$,wL2), $t=1$.]{
\includegraphics[width=0.3\textwidth]{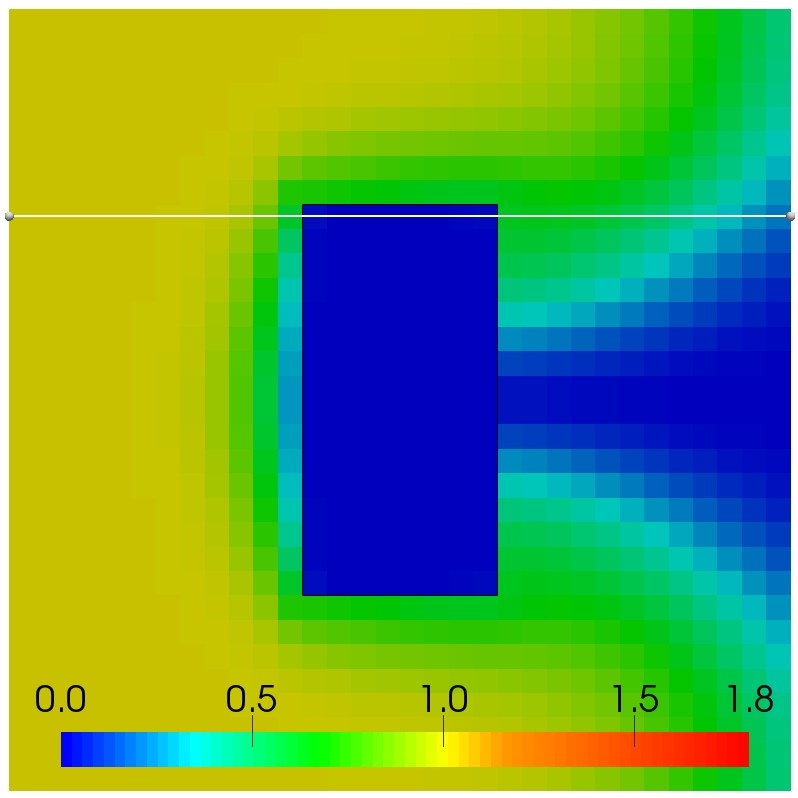}
\label{fig:concentration_barrier_problem_harm_ppw1}
} \\
\subfloat[CG(SD,$\vartheta$), $t=2$.]{
\includegraphics[width=0.3\textwidth]{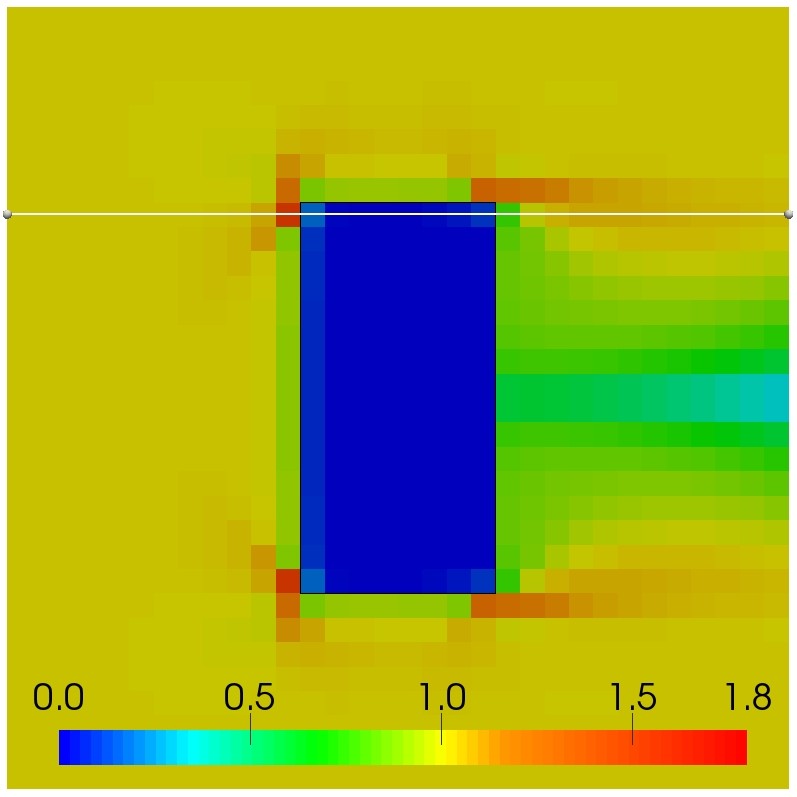}
\label{fig:concentration_barrier_problem_harm_cg2}
}
\subfloat[PP(SD,$\vartheta$,L2), $t=2$.]{
\includegraphics[width=0.3\textwidth]{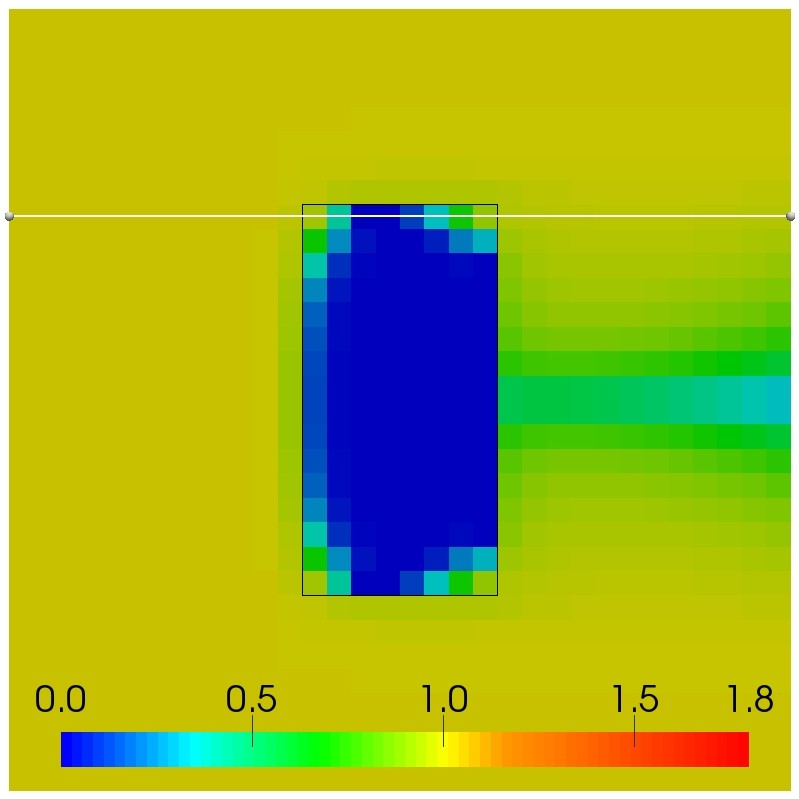}
\label{fig:concentration_barrier_problem_harm_pp2}
}
\subfloat[PP(SD,$\vartheta$,wL2), $t=2$.]{
\includegraphics[width=0.3\textwidth]{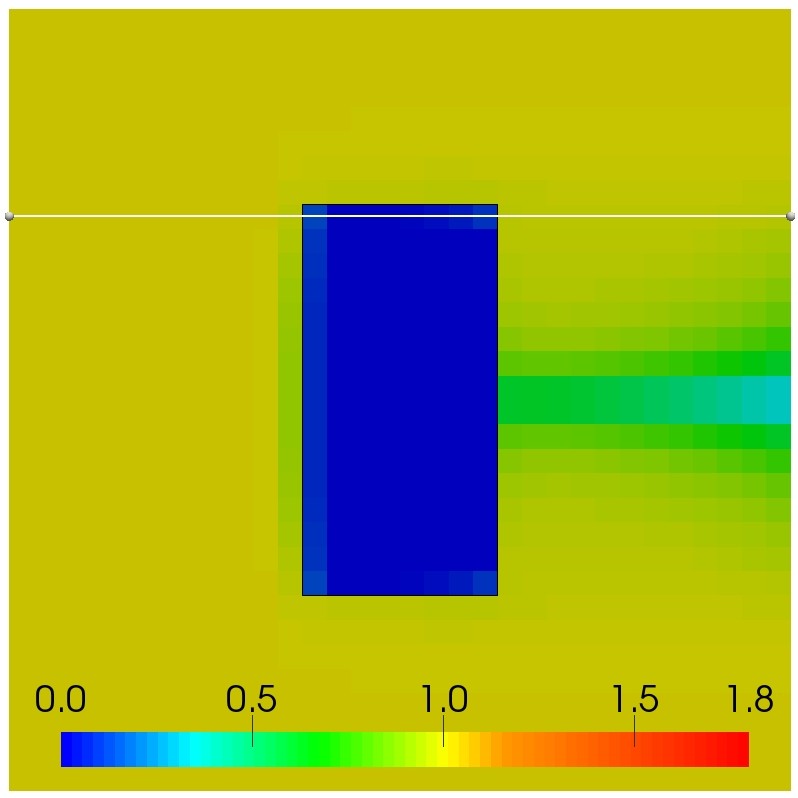}
\label{fig:concentration_barrier_problem_harm_ppw2}
} \\
\subfloat[Concentration over line, $t=1$.]{
\includegraphics[width=0.37\textwidth]{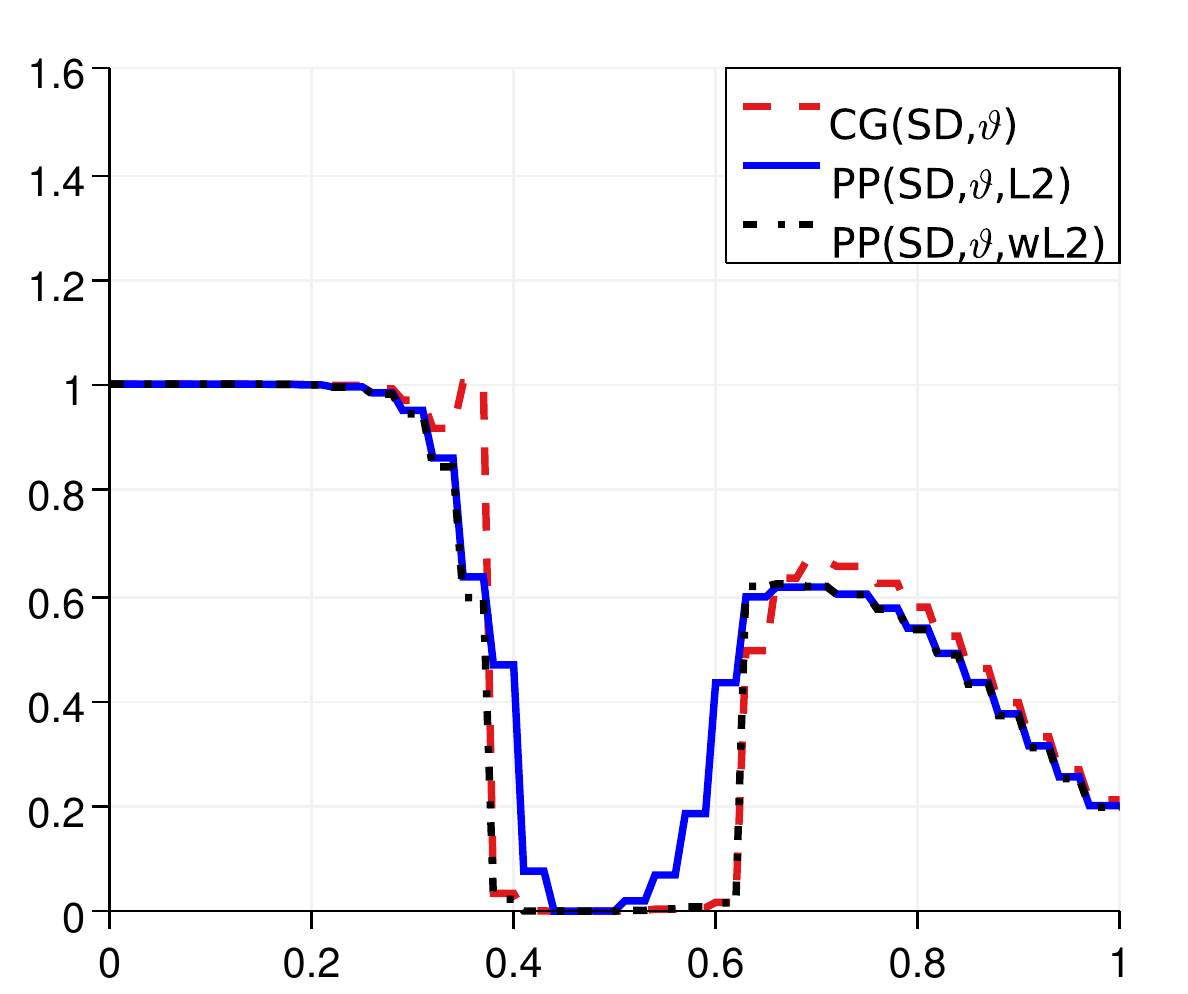}
}
\hspace*{1mm}
\subfloat[Concentration over line, $t=2$.]{
\includegraphics[width=0.37\textwidth]{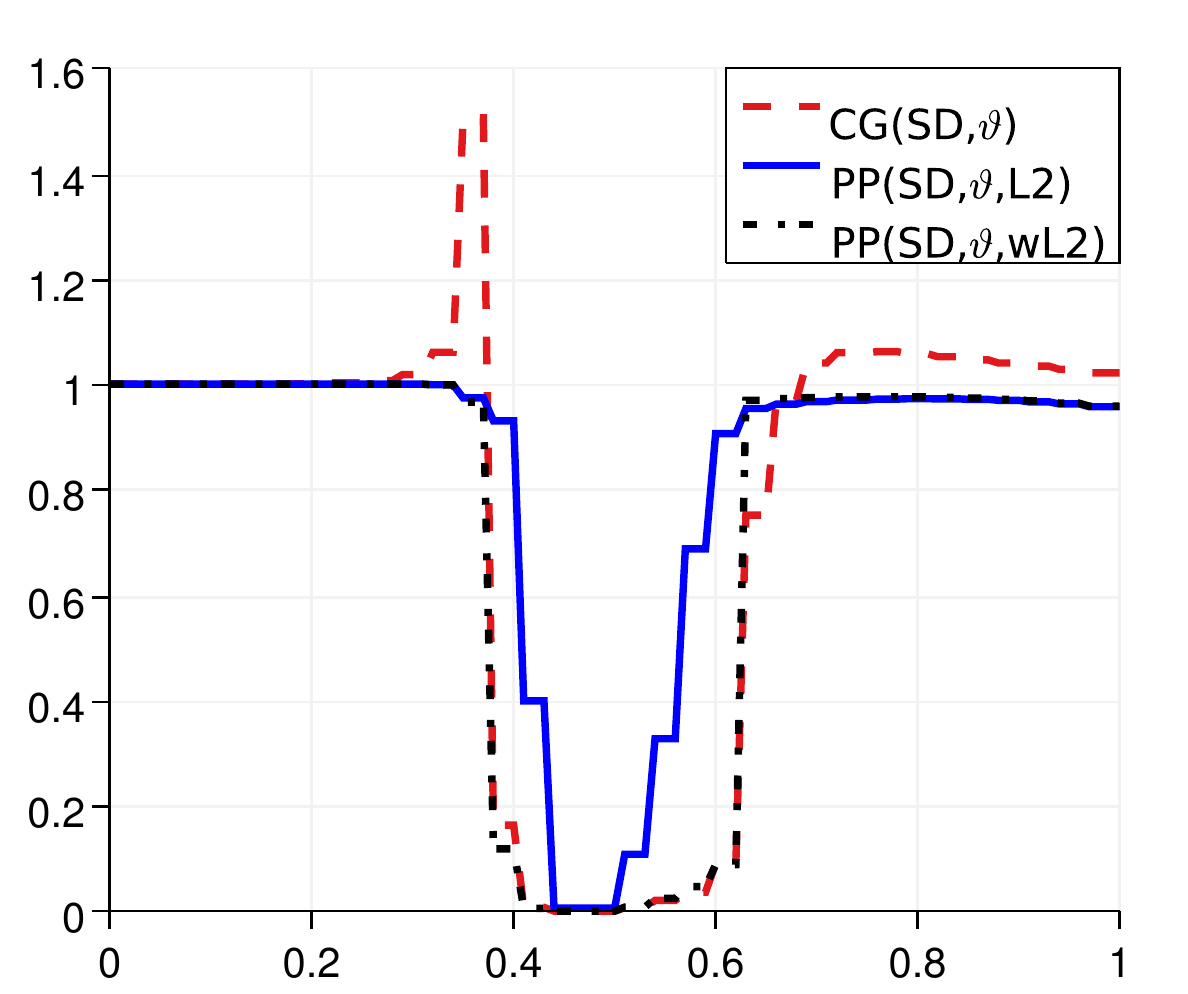}
}
\caption{\textbf{Barrier problem}. Concentration solution at two different times, with CG flux and postprocessed flux with the standard $L^2$ norm and the weighted $L^2$ norm. The harmonic weights $\theta=\vartheta$ is used for the average in calculations of CG flux, $U_h$. The solution along the white line ($y=0.735$) is plotted in the bottom row. The low permeability region in inscribed in the black box.}
\label{fig:concentration_barrier_problem_harm}
\end{figure}

The overshoot quantity, $\mathcal{O}(c_h)$, the minimum and maximum of $c_h$ and the norm of the residual is reported in Table \ref{tab:barrier_overshoot} for the different cases studied above. We see that for all postprocessing cases, $\mathcal{R}(V_h)$ and $\mathcal{O}(c_h)$ is zero down to machine precision, and that $c_h\in[0,1]$. This is not satisfied with CG flux, which is not locally conservative.

\begin{table}[bpt]
\caption{\textbf{Barrier problem}. Norm of residual, $\Vert \mathcal{R}(\cdot)\Vert_{\mcEh}$, overshoot, $\mathcal{O}(c_h)$, and minimum and maximum value of concentration solution at $t=2$ for different flux approximations.}
\label{tab:barrier_overshoot}
\begin{center}
\footnotesize
\begin{tabular}{lrrrrr}
\hline
Method &
$\Vert \mathcal{R}(U_h)\Vert_{\mcEh}$ &
$\Vert \mathcal{R}(V_h)\Vert_{\mcEh}$ &
$\mathcal{O}(c_h)$ &
min($c_h$) &
max($c_h$)
\\ \hline
CG(SD,1/2)    			& 1.184   & -       & 0.04107  & 2.1e-12  & 1.822 \\
CG(SD,$\vartheta$)      & 1.895   & -       & 0.03285  & 3.1e-13  & 1.505 \\
PP(SD,1/2,L2) 			& -       & 4.8e-16 & 3.2e-17  & 2.2e-11  & 1.000 \\
PP(SD,$\vartheta$,L2)   & -       & 9.7e-16 & 1.6e-17  & 1.4e-10  & 1.000 \\
PP(SD,$\vartheta$,wL2)  & -       & 2.7e-15 & 4.8e-17  & 3.0e-13  & 1.000 \\
\hline
\end{tabular}
\end{center}
\end{table}

Next, we compare the postprocessing step with the CG solver in terms of efficiency and computational complexity. Both the CG problem (Eq.~\eqref{eq:cg_scheme_strong}) and the postprocessing problem (Eq.~\eqref{eq:pp_variational}) are symmetric and positive definite, so we use the conjugate gradient method as linear solver. In Table \ref{tab:barrier_timing} we report on degrees of freedom (DoF), condition number ($\kappa$)\footnote{The condition numbers are estimated by routines in the deal.II library.}, number of iterations in the linear solver (it), and the CPU time used by the linear solver (time). This is done for the CG problem and the postprocessing problem both with and without weighting for recursively refined regular Cartesian grids. We consider both the standard conjugate gradient solver and the preconditioned conjugate gradient with a symmetric successive overrelaxation preconditioner, SSOR(1.5). For all cases we use strong Dirichlet conditions and harmonic weighting of the CG flux. 

Without preconditioning, we see that PP(SD,$\vartheta$,L2) is much less costly to solve for than CG(SD,$\vartheta$), both in terms of the condition number and solver time. PP(SD,$\vartheta$,wL2) is more expensive, and the solution time is $\sim 70\%$ of that of CG(SD,$\vartheta$). This is because weighting introduces high aspect ratios in the system matrix, see eq.~\eqref{eq:pp_matrix_entries}. However, if we apply a relatively simple preconditioner as SSOR, the condition numbers and solution times drop remarkably for CG(SD,$\vartheta$) and PP(SD,$\vartheta$,wL2), such that the computational complexity of PP(SD,$\vartheta$,L2) and PP(SD,$\vartheta$,wL2) are almost similar. Still, the additional cost of the postprocessing step is significant ($\sim55\%$ for PP(SD,$\vartheta$,L2) and $\sim 60\%$ for PP(SD,$\vartheta$,wL2)).

Finally, we test the sensitivity of the computational complexity with respect to the permeability contrast. This is done by keeping the grid resolution fixed at $1/h=64$ and then vary the permeability in the low permeable block, denoted $k_b$. These results are reported in Table \ref{tab:barrier_timing_perm}. For the case without preconditioning, we see that the condition number and linear solver time for CG(SD,$\vartheta$) and PP(SD,$\vartheta$,wL2) scales badly with the permeability contrast, whereas PP(SD,$\vartheta$,L2) is nearly unaffected. This is as expected since the system matrix for PP(SD,$\vartheta$,L2) is independent on the permeability, while for CG(SD,$\vartheta$) and PP(SD,$\vartheta$,wL2) it is not. However, if we look at the preconditioned system, we see that the effect of the permeability contrast almost vanishes. Hence, for this problem, the SSOR preconditioner is able to remove 
the effect of the permeability contrast on the condition number.

\begin{table}[bpt]
\caption{\textbf{Barrier problem}. Computational complexity for different problems; DoF: Degrees of Freedom, $\kappa$: condition number, it: number of iterations in linear solver, time: CPU time used by linear solver including initialization of the preconditioner (median value over 11 runs). The linear solver is the (preconditioned) conjugate gradient method with residual tolerance $10^{-12}$.}
\label{tab:barrier_timing}
\centering
\subfloat[Without preconditioning.]{
\footnotesize
\begin{tabular}{|r|rrrr|rrrr|rrrr|}
\hline
& \multicolumn{4}{|c|}{CG(SD,$\vartheta$)}
& \multicolumn{4}{|c|}{PP(SD,$\vartheta$,L2)}
& \multicolumn{4}{|c|}{PP(SD,$\vartheta$,wL2)} \\
$1/h$ & DoF   & $\kappa$ & it    & time   & DoF    & $\kappa$ & it    & time   & DoF   & $\kappa$ & it   & time \\ \hline
16    & 289   & 5505     & 151   & 0.0107 & 256    & 58       & 40    & 0.0018 & 256   & 3611     & 136  & 0.0073 \\
32    & 1089  & 21114    & 443   & 0.1039 & 1024   & 220      & 85    & 0.0136 & 1024  & 12748    & 416  & 0.0766 \\    
64    & 4225  & 83607    & 1203  & 0.4129 & 4096   & 856      & 163   & 0.0413 & 4096  & 49475    & 1037 & 0.2893 \\
128   & 16641 & 333602   & 2915  & 3.4805 & 16384  & 3372     & 307   & 0.2758 & 16384 & 196428   & 2350 & 2.2415 \\ \hline  
\end{tabular}
} \\
\subfloat[With SSOR(1.5) precondtioner.]{
\footnotesize
\begin{tabular}{|r|rrrr|rrrr|rrrr|}
\hline
& \multicolumn{4}{|c|}{CG(SD,$\vartheta$)}
& \multicolumn{4}{|c|}{PP(SD,$\vartheta$,L2)}
& \multicolumn{4}{|c|}{PP(SD,$\vartheta$,wL2)} \\
$1/h$ & DoF   & $\kappa$ & it    & time   & DoF    & $\kappa$ & it    & time   & DoF   & $\kappa$ & it   & time \\ \hline
16    & 289   & 9.1      & 27    & 0.0040 & 256    & 11.6     & 25    & 0.0026 & 256   & 10.6     & 27   & 0.0028 \\
32    & 1089  & 30.2     & 43    & 0.0221 & 1024   & 39.2     & 38    & 0.0125 & 1024  & 33.8     & 41   & 0.0135 \\    
64    & 4225  & 110.6    & 77    & 0.0731 & 4096   & 146.2    & 62    & 0.0375 & 4096  & 121.8    & 69   & 0.0444 \\
128   & 16641 & 424.7    & 147   & 0.4327 & 16384  & 567.7    & 109   & 0.2116 & 16384 & 465.3    & 126  & 0.2459 \\ \hline  
\end{tabular}
}
\end{table}

\begin{table}[bpt]
\caption{\textbf{Barrier problem}. Computational complexity for different problems; DoF: Degrees of Freedom, $\kappa$: condition number, it: number of iterations in linear solver, time: CPU time used by linear solver including initialization of the preconditioner (median value over 11 runs). The linear solver is the (preconditioned) conjugate gradient method with residual tolerance $10^{-12}$. The grid resolution is kept constant at $1/h=64$, but the permeability in the low permeable block, $k_b$, is varied.}
\label{tab:barrier_timing_perm}
\centering
\subfloat[Without preconditioning.]{
\footnotesize
\begin{tabular}{|r|rrrr|rrrr|rrrr|}
\hline
& \multicolumn{4}{|c|}{CG(SD,$\vartheta$)}
& \multicolumn{4}{|c|}{PP(SD,$\vartheta$,L2)}
& \multicolumn{4}{|c|}{PP(SD,$\vartheta$,wL2)} \\
$k_b$     & DoF   & $\kappa$ & it    & time   & DoF    & $\kappa$ & it    & time   & DoF   & $\kappa$ & it   & time \\ \hline
$10^{-1}$ & 4225  & 1885     & 275   & 0.1085 & 4096   & 856      & 161   & 0.0484 & 4096  & 1331     & 221  & 0.0665 \\
$10^{-3}$ & 4225  & 83607    & 1203  & 0.4127 & 4096   & 856      & 163   & 0.0413 & 4096  & 49475    & 1037 & 0.2875 \\    
$10^{-5}$ & 4225  & 8328390  & 2565  & 0.9362 & 4096   & 856      & 163   & 0.0375 & 4096  & 4931220  & 2364 & 0.6897 \\ \hline  
\end{tabular}
} \\
\subfloat[With SSOR(1.5) precondtioner.]{
\footnotesize
\begin{tabular}{|r|rrrr|rrrr|rrrr|}
\hline
& \multicolumn{4}{|c|}{CG(SD,$\vartheta$)}
& \multicolumn{4}{|c|}{PP(SD,$\vartheta$,L2)}
& \multicolumn{4}{|c|}{PP(SD,$\vartheta$,wL2)} \\
$k_b$     & DoF   & $\kappa$ & it    & time   & DoF    & $\kappa$ & it    & time   & DoF   & $\kappa$ & it   & time \\ \hline
$10^{-1}$ & 4225  & 111      & 77    & 0.0649 & 4096   & 146      & 61    & 0.0369 & 4096  & 123      & 68   & 0.0372 \\
$10^{-3}$ & 4225  & 111      & 77    & 0.0809 & 4096   & 146      & 62    & 0.0353 & 4096  & 122      & 69   & 0.0395 \\    
$10^{-5}$ & 4225  & 111      & 77    & 0.0934 & 4096   & 146      & 62    & 0.0467 & 4096  & 122      & 69   & 0.0506 \\ \hline  
\end{tabular}
}
\end{table}

\subsection{Channel Problem}
To further investigate the importance of harmonic averaging (objective (v)), consider now flow and transport through a channel with corners, see Fig.~\ref{fig:channel_problem}. The problem parameters are the same as for the barrier problem, except for the permeability distribution, which now forms a channel through the domain, and the boundary concentration, $c_B$, which is one into the channel and zero elsewhere. The channel has permeability $k=1$, while the surroundings have permeability $k=k_s\ll 1$, so we expect most of the flow to be in the channel. We only consider harmonic averaging ($\theta=\vartheta$), but use both the standard $L^2$ norm and the weighted $L^2$ norm for minimization in the postprocessing method, PP(SD,$\vartheta$,L2) and PP(SD,$\vartheta$,wL2), respectively. We study the cases $k_s=10^{-2}$ and $k_s=10^{-5}$, and set $\Delta t = 0.005$ and $T=2$.

The concentration solutions for the different scenarios are displayed in Fig.~\ref{fig:concentration_channel_problem}, and residuals, overshoot and minimum and maximum values are reported in Table \ref{tab:channel_overshoot}. 
For $k_s=10^{-2}$, we get $c_h\gg 0$ in some areas outside but close to the channel. This seems reasonable, as the contrast in permeability is two orders of magnitude. However, for $k_s=10^{-5}$ the interface should be close to impermeable, and we expect very low concentrations outside the channel. For CG(SD,$\vartheta$), we observe that $c_h\sim 0$ outside the channel for $k=10^{-5}$, but that $c_h>1$ in many elements due to lack of local conservation (Fig.\ \ref{fig:concentration_channel_problem_cg-2} and \ref{fig:concentration_channel_problem_cg-5}). For the case PP(SD,$\vartheta$,L2), we see that the difference in solution for $k_s=10^{-2}$ and $k=10^{-5}$ is rather small, and that $1>c_h\gg 0$ for some elements outside the channel also for $k=10^{-5}$ (Fig.\ \ref{fig:concentration_channel_problem_pp-2} and \ref{fig:concentration_channel_problem_pp-5}). This is problematic, since the interface should be close to impermeable. If we instead minimize in the weighted $L^2$ norm, PP(SD,$\vartheta$,wL2), we are able to resolve this issue so that the interface is close to impermeable (Fig.\ \ref{fig:concentration_channel_problem_pp-2} and \ref{fig:concentration_channel_problem_pp-5}).

The shortcoming of postprocessing with the standard $L^2$ norm is that it does not take the permeability contrast into account. Let $F$ be a face on the boundary of the channel. With harmonic averaging, $U_h\vert_F\sim 0$. However, in the minimization step without weighting, we allow for a flux correction that is small in absolute value compared to fluxes on faces inside the channel, but still relatively large compared to $U_h\vert_F$. Thus, $V_h\vert_F$ might be orders of magnitude larger than $U_h\vert_F$, resulting in a more permeable interface. When we use the weighted $L^2$ norm, $F$ is given a large weight (the inverse of the effective permeability, $k_e$), so that we do not allow for such large relative correction.

\begin{figure}[bpt]
\centering
\includegraphics[width=0.4\textwidth]{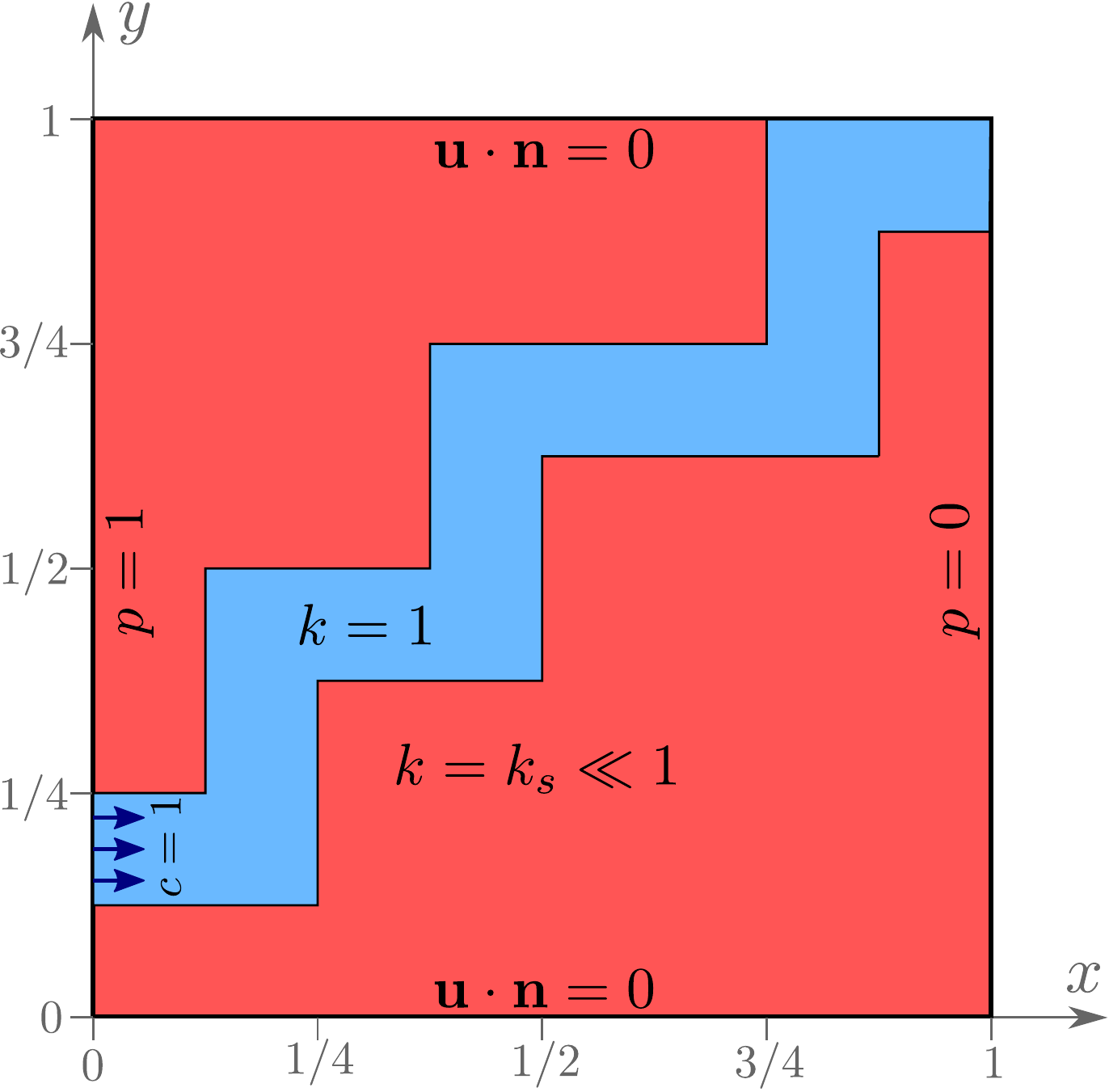}
\caption{\textbf{Channel problem}. Problem definition. Boundary conditions are $p=1$ on the left, $p=0$ on the right, and $\bfu\cdot\bfn =0$ on the bottom and top. The boundary concentration is $c=1$ into the channel only, and $0$ elsewhere.}
\label{fig:channel_problem}
\end{figure}

\begin{figure}[bpt]
\centering
\subfloat[CG(SD,$\vartheta$), $k_s=10^{-2}$.]{
\includegraphics[width=0.27\textwidth]{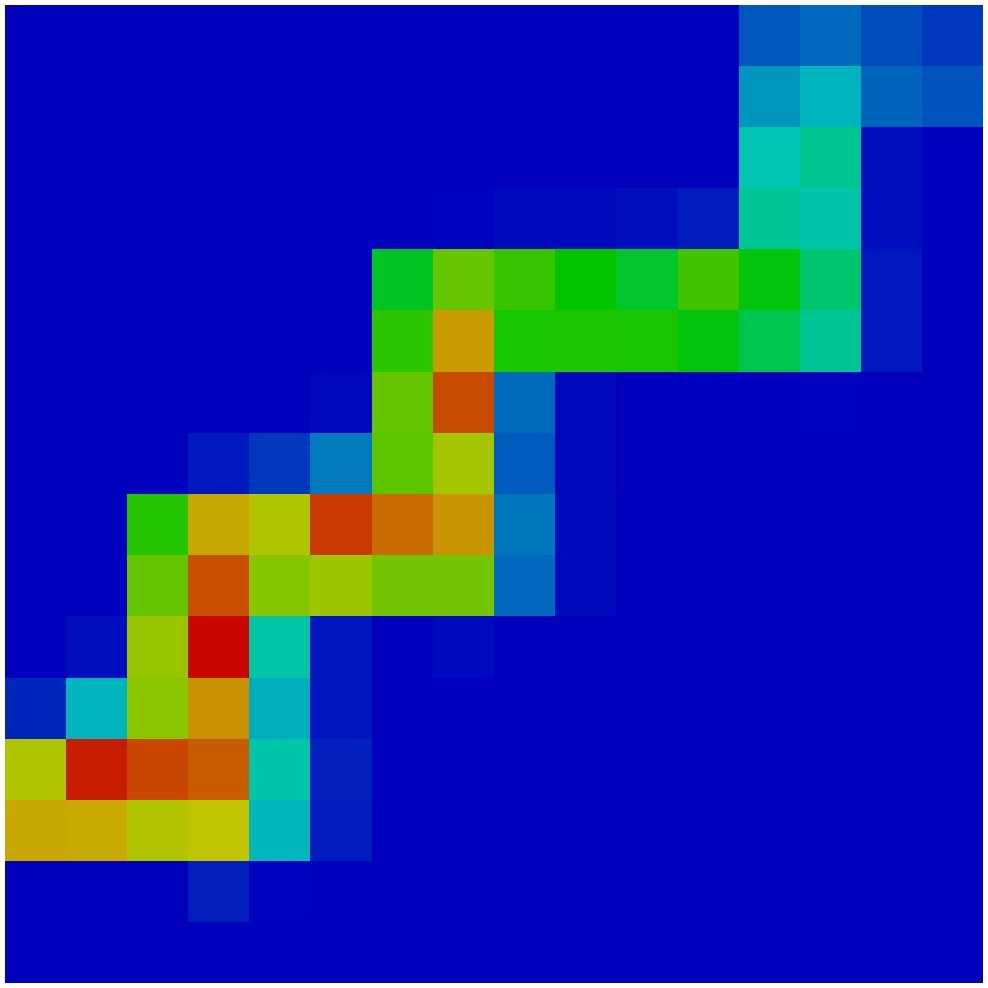}
\label{fig:concentration_channel_problem_cg-2}
}
\subfloat[PP(SD,$\vartheta$,L2), $k_s=10^{-2}$.]{
\includegraphics[width=0.27\textwidth]{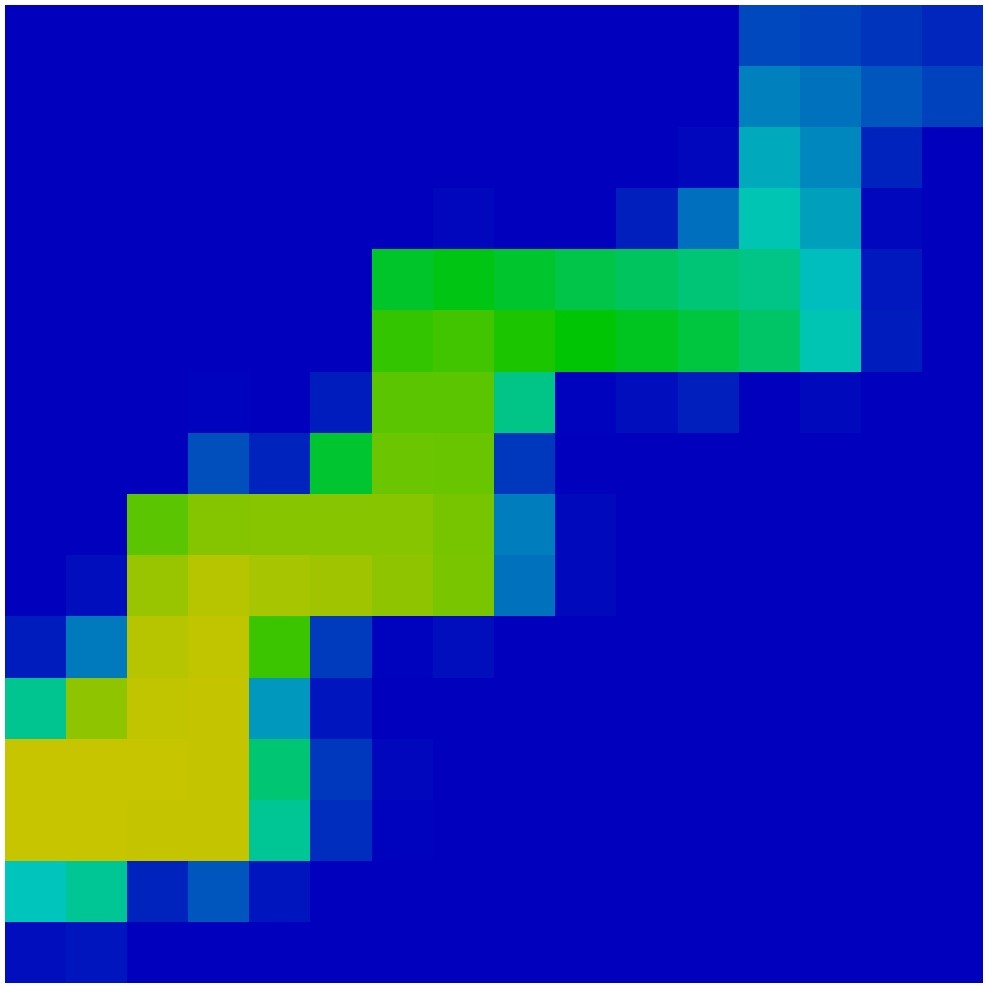}
\label{fig:concentration_channel_problem_pp-2}
}
\subfloat[PP(SD,$\vartheta$,wL2), $k_s=10^{-2}$.]{
\includegraphics[width=0.27\textwidth]{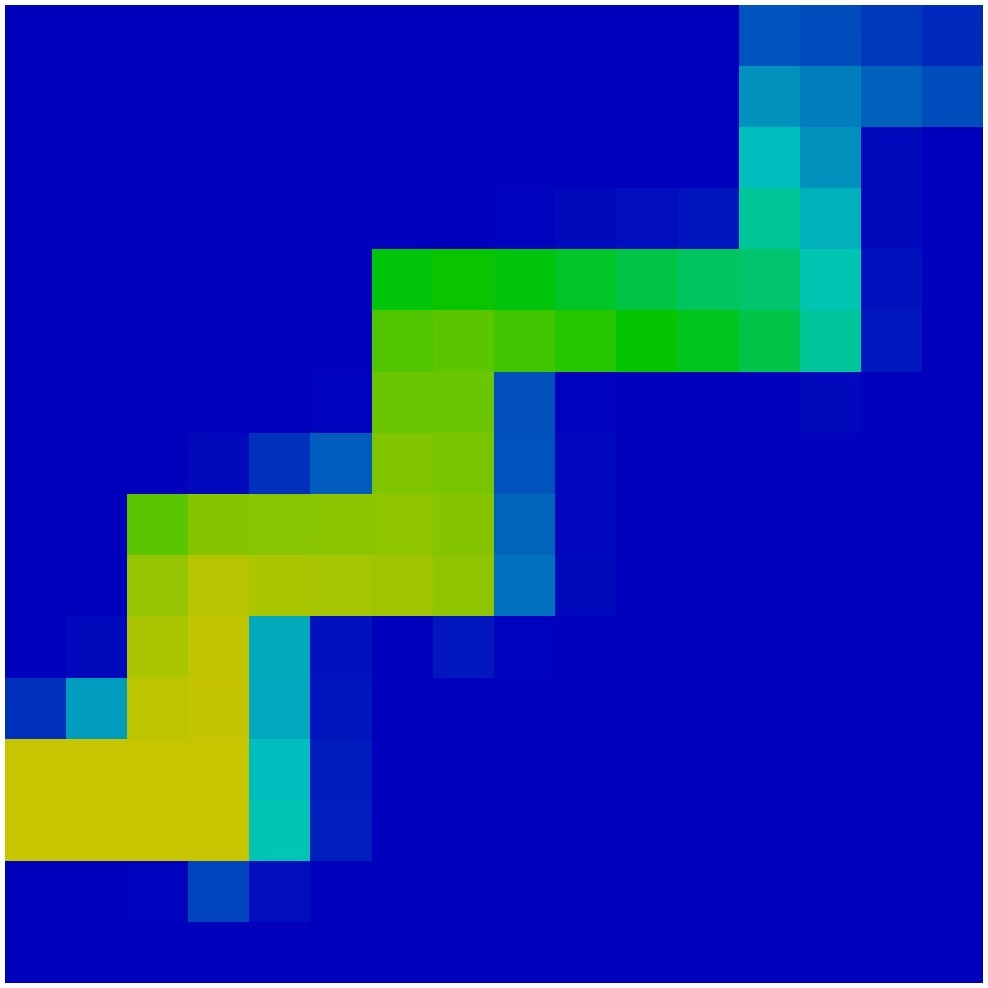}
\label{fig:concentration_channel_problem_ppw-2}
} 
\subfloat{
\includegraphics[width=0.05\textwidth]{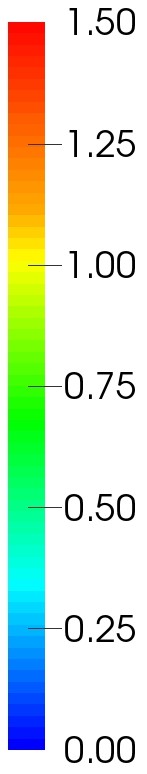}
} 
\\
\addtocounter{subfigure}{-1}
\subfloat[CG(SD,$\vartheta$), $k_s=10^{-5}$.]{
\includegraphics[width=0.27\textwidth]{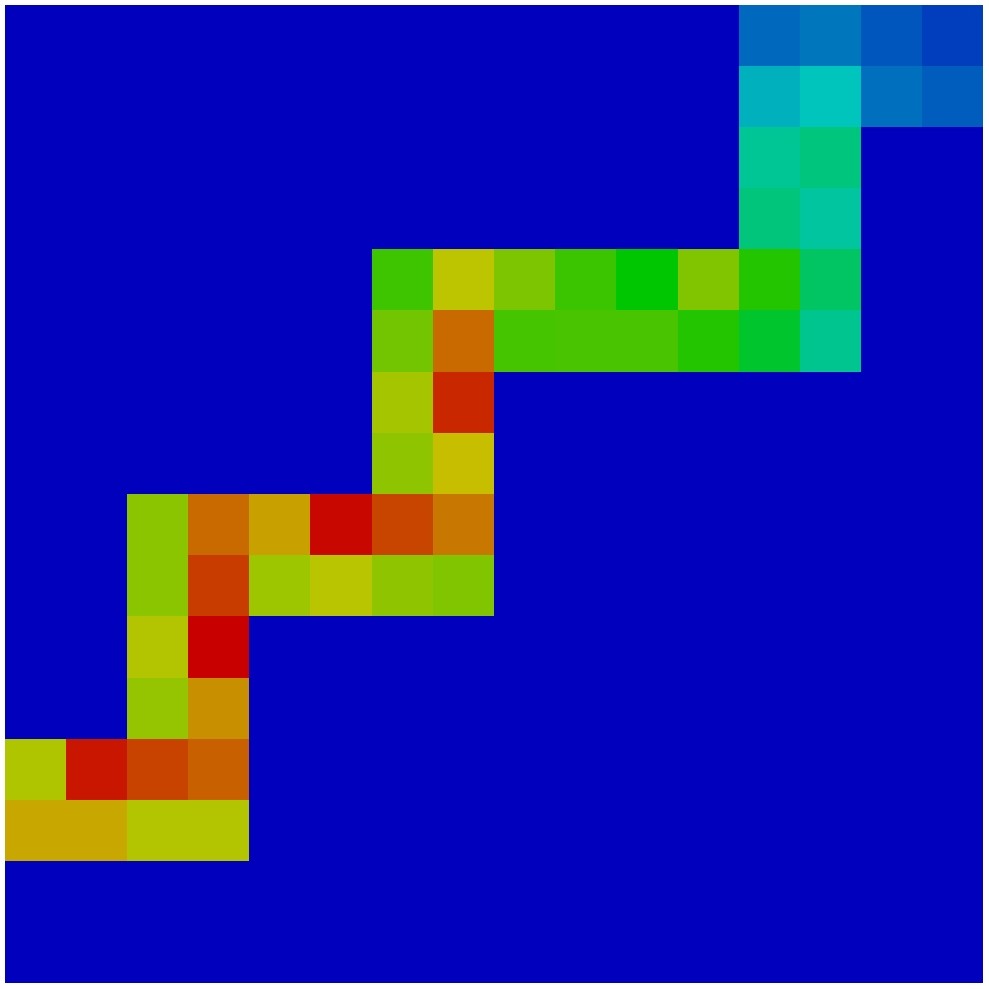}
\label{fig:concentration_channel_problem_cg-5}
}
\subfloat[PP(SD,$\vartheta$,L2), $k_s=10^{-5}$.]{
\includegraphics[width=0.27\textwidth]{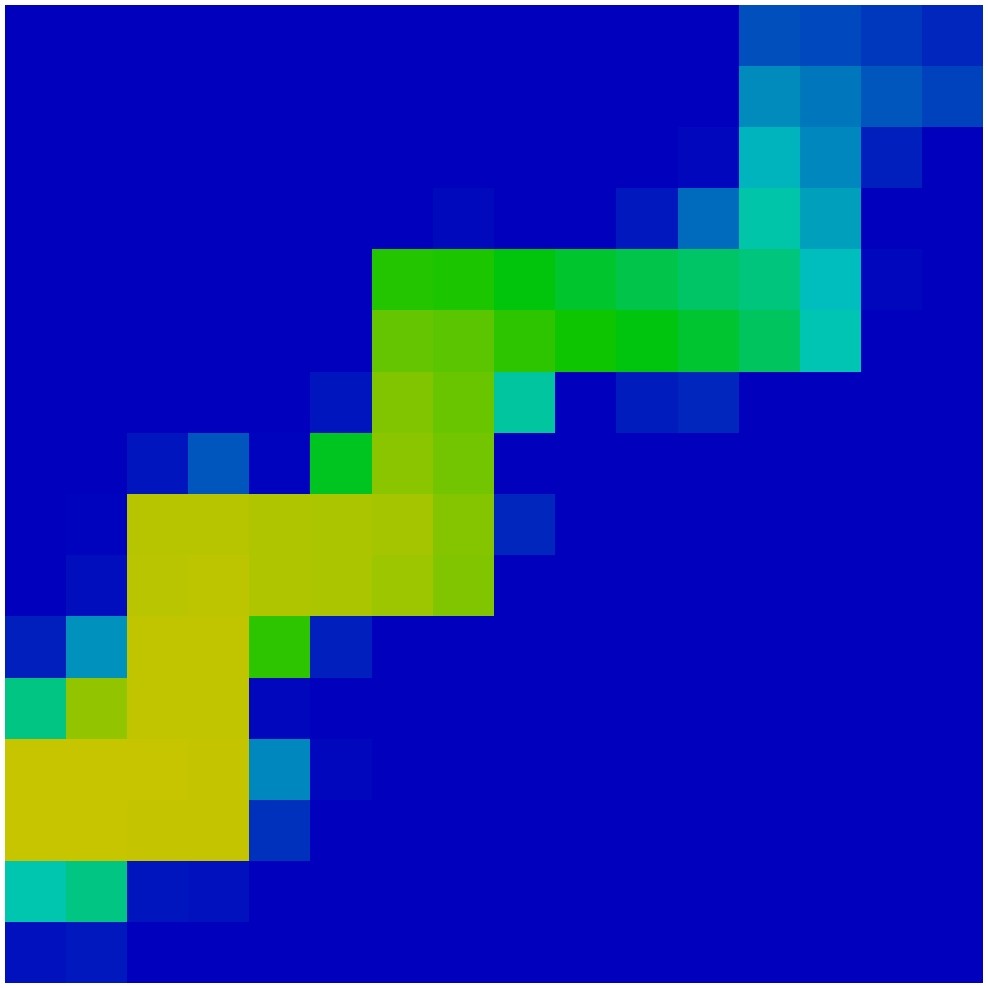}
\label{fig:concentration_channel_problem_pp-5}
}
\subfloat[PP(SD,$\vartheta$,wL2), $k_s=10^{-5}$.]{
\includegraphics[width=0.27\textwidth]{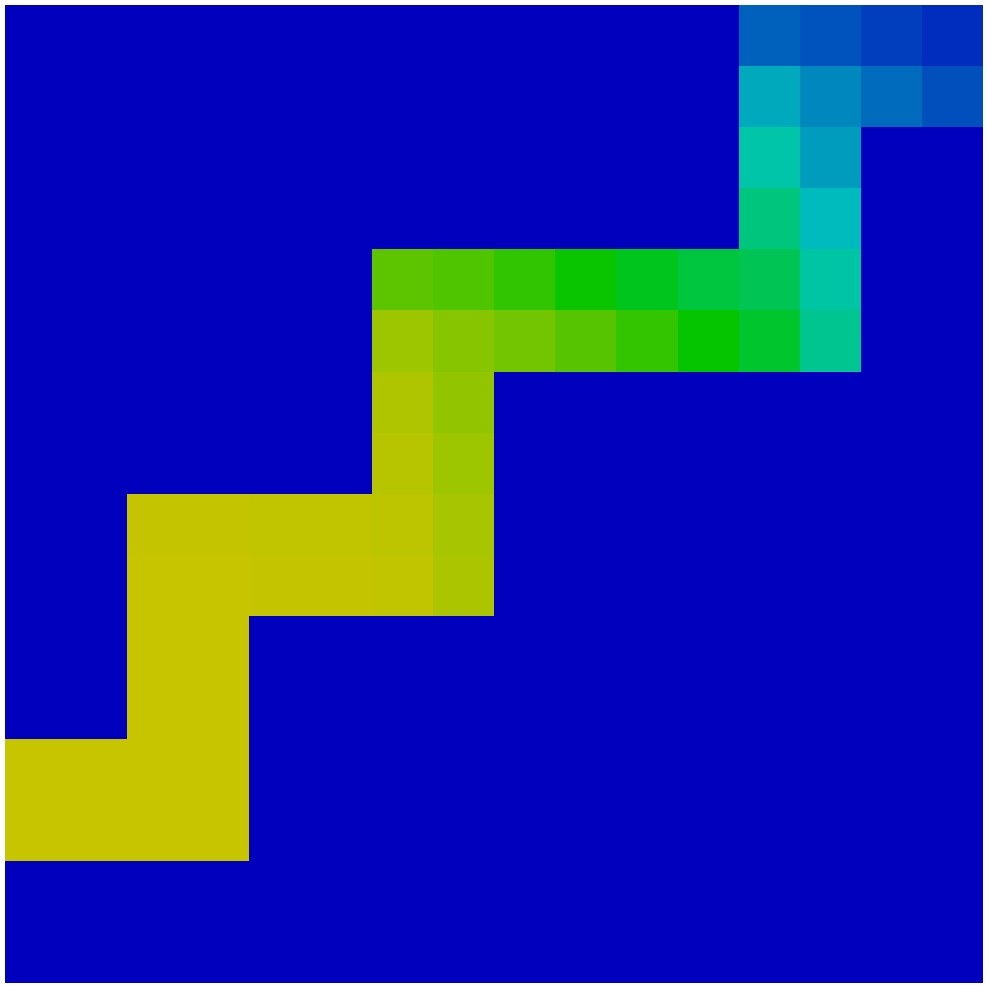}
\label{fig:concentration_channel_problem_ppw-5}
}
\subfloat{
\includegraphics[width=0.05\textwidth]{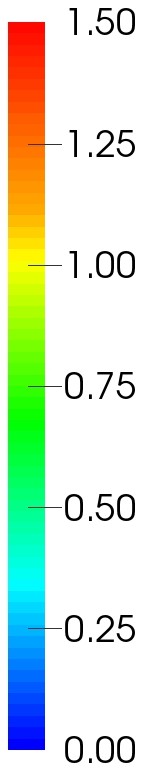}
} 
\caption{\textbf{Channel problem}. Concentration solutions at $t=2$ for different flux approximations (left to right) and different permeability outside channel (top and bottom). Harmonic averaging is used for calculations of the CG flux, $U_h$, in all cases.}
\label{fig:concentration_channel_problem}
\end{figure}

\begin{table}[bpt]
\setlength{\tabcolsep}{5pt}
\caption{\textbf{Channel problem}. Norm of residual, $\Vert \mathcal{R}(\cdot)\Vert_{\mcEh}$, overshoot, $\mathcal{O}(c_h)$, and minimum and maximum value of concentration solution at $t=2$ for different flux approximations.}
\label{tab:channel_overshoot}
\begin{center}
\footnotesize
\begin{tabular}{llrrrrr}
\hline
$k_s$ &
Method &
$\Vert \mathcal{R}(U_h)\Vert_{\mcEh}$ &
$\Vert \mathcal{R}(V_h)\Vert_{\mcEh}$ &
$\mathcal{O}(c_h)$ &
min($c_h$) &
max($c_h$) \\ 
\hline
1e-2 & CG(SD,$\vartheta$)     & 0.9646  & -       & 0.05715  & 0  & 1.478 \\
     & PP(SD,$\vartheta$,L2)  & -       & 3.6e-16 & 0        & 0  & 1.000 \\
     & PP(SD,$\vartheta$,wL2) & -       & 6.7e-16 & 0        & 0  & 1.000 \\
\hline
1e-5 & CG(SD,$\vartheta$)     & 0.9915  & -       & 0.06951  & 0 & 1.502 \\
     & PP(SD,$\vartheta$,L2)  & -       & 4.2e-15 & 0        & 0 & 1.000 \\
     & PP(SD,$\vartheta$,wL2) & -       & 6.7e-16 & 0        & 0 & 1.000 \\
\hline
\end{tabular}
\end{center}
\end{table}

\subsection{Well Pair Problem}

Next, we consider a simplified well scenario, and focus on objective (vii) for a problem with non-zero right hand side. Still, we let $\beta = 0$ and $\Omega=(0,1)^2$, but now $\mathbf{K} = k\mathbb{I}$, where $k=1$ if $x\le 0.5$ and $k=10^{-3}$ otherwise. Next, we model a injector/producer well pair by setting $q=100$ in the lower left corner and $q=-100$ in the upper right corner. See Fig.~\ref{fig:well_problem} for a sketch. The initial condition is $c_0= 0$ and the concentration of the injected fluid, $c_w=1.0$. We assume a pure Neumann boundary with $u_B= 0$. The coupled flow and transport problem is solved on quadratic grids with $h=\{1/16,1/32,1/64 \}$ and $\Delta t=0.01$. We only consider harmonic average in the calculations of the CG flux and use the weighted $L^2$ norm for minimization in the postprocessing method (CG(SD,$\vartheta$) and PP(SD,$\vartheta$,wL2)).

\begin{figure}[bpt]
\centering
\includegraphics[width=0.4\textwidth]{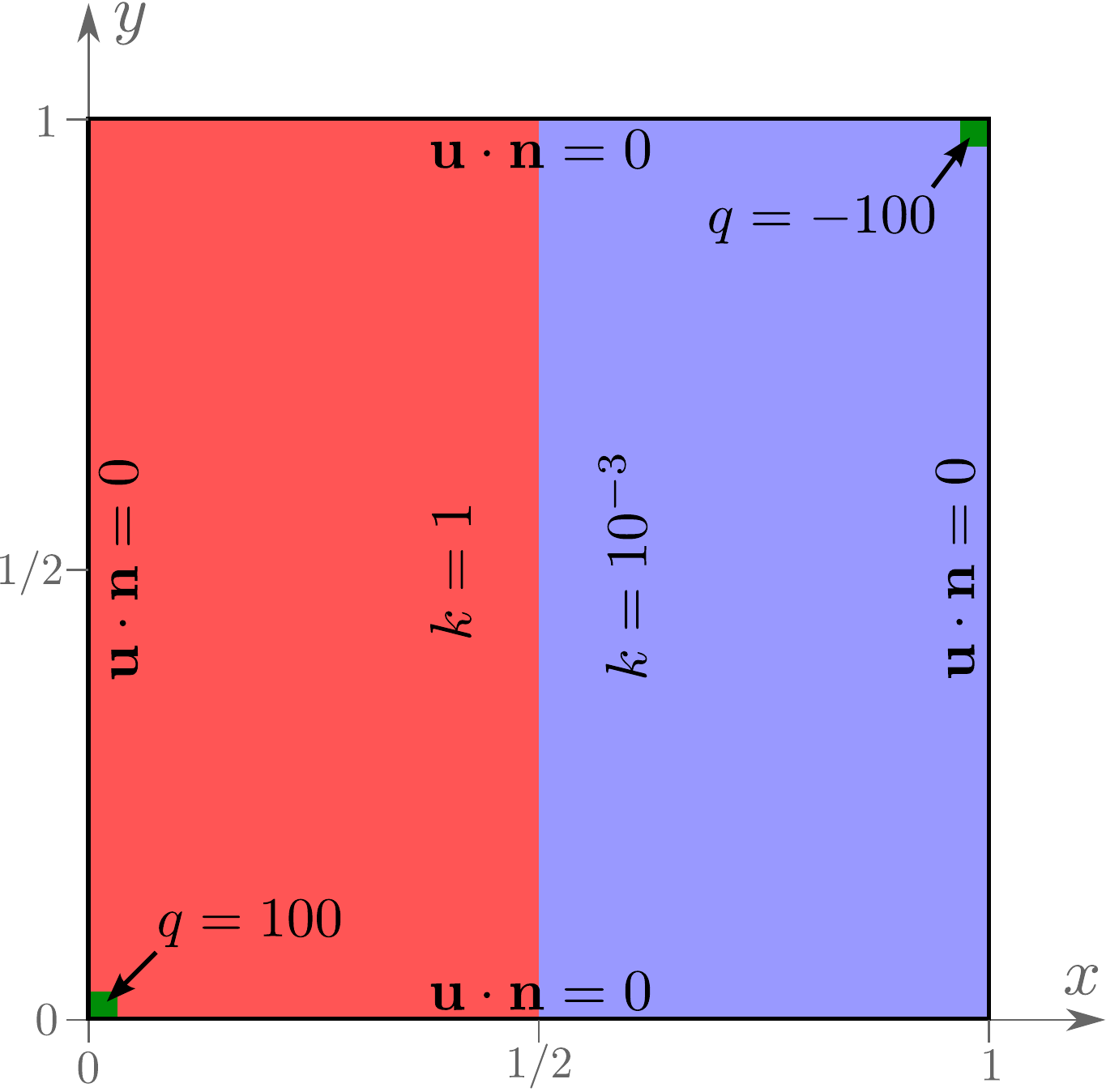}
\caption{\textbf{Well pair problem}. Problem definition. The green squares where $q\neq 0$ in the lower left and upper right corner have size $1/32\times 1/32$.}
\label{fig:well_problem}
\end{figure}

The concentration solution at different times for the grid with $h=1/32$ is shown in Fig.~\ref{fig:well_conc_time}. The concentration is produced in the lower left corner and moves towards the source in the upper right corner. The difference between $\text{CG}(\text{SD},\vartheta)$ and $\text{PP}(\text{SD},\vartheta,\text{wL2})$ is significant and the maximum principle $c_h\le 1$ is violated for $\text{CG}(\text{SD},\vartheta)$. Postprocessing is necessary to produce an acceptable concentration solution.

Similar results at $t=10$ for quadratic grids with $h=\{1/16,1/32,1/64 \}$ are shown in Fig.~\ref{fig:well_conc}. Furthermore, residuals, overshoot and minimum and maximum values are given in Table~\ref{tab:well_overshoot}. Evidently, the difference in concentration solution is smaller for smaller $h$. This is as expected since CG converges to the true solution, which is locally conservative. The area where $c_h>1$ seems to cluster around the sink and source for $h=1/64$. 

\begin{figure}[btp]
\centering
\subfloat[$\text{CG}(\text{SD},\vartheta),\ t=1$.]{
\includegraphics[height=0.25\textwidth]{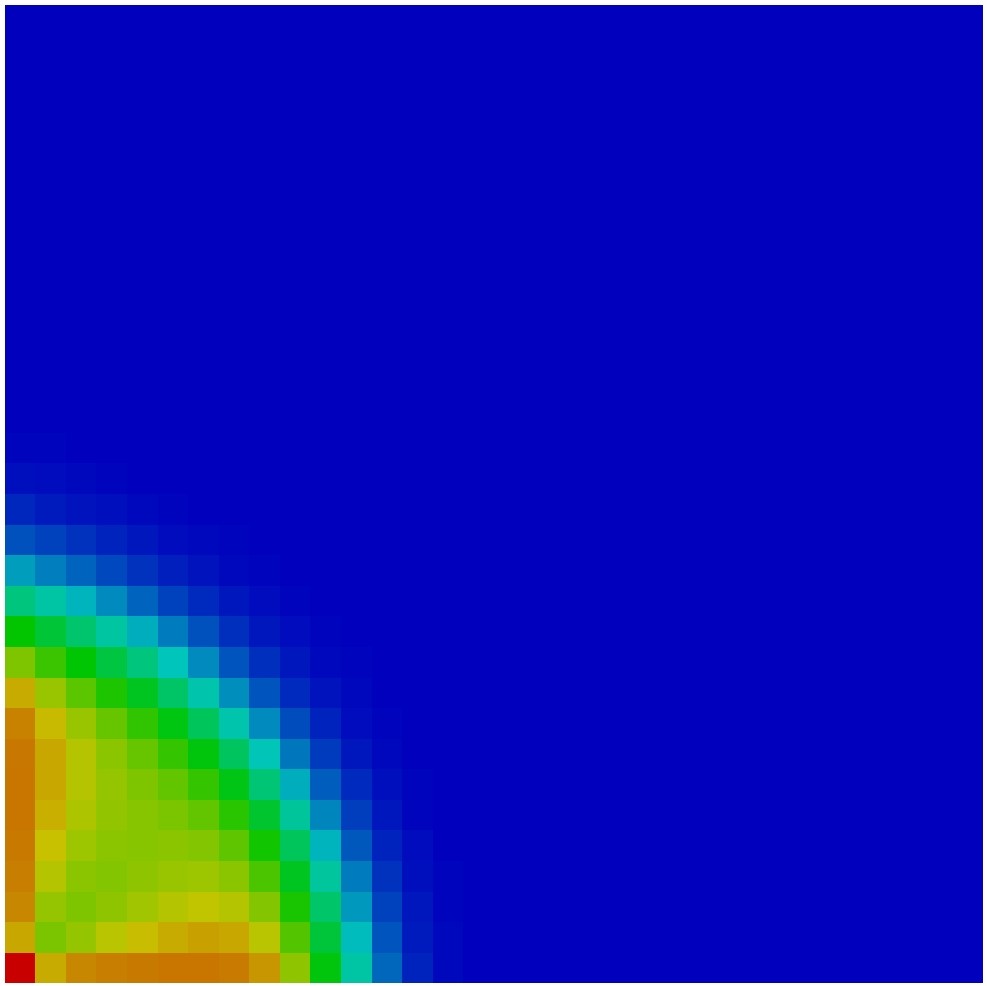}}
\subfloat[$\text{CG}(\text{SD},\vartheta),\ t=6$.]{
\includegraphics[height=0.25\textwidth]{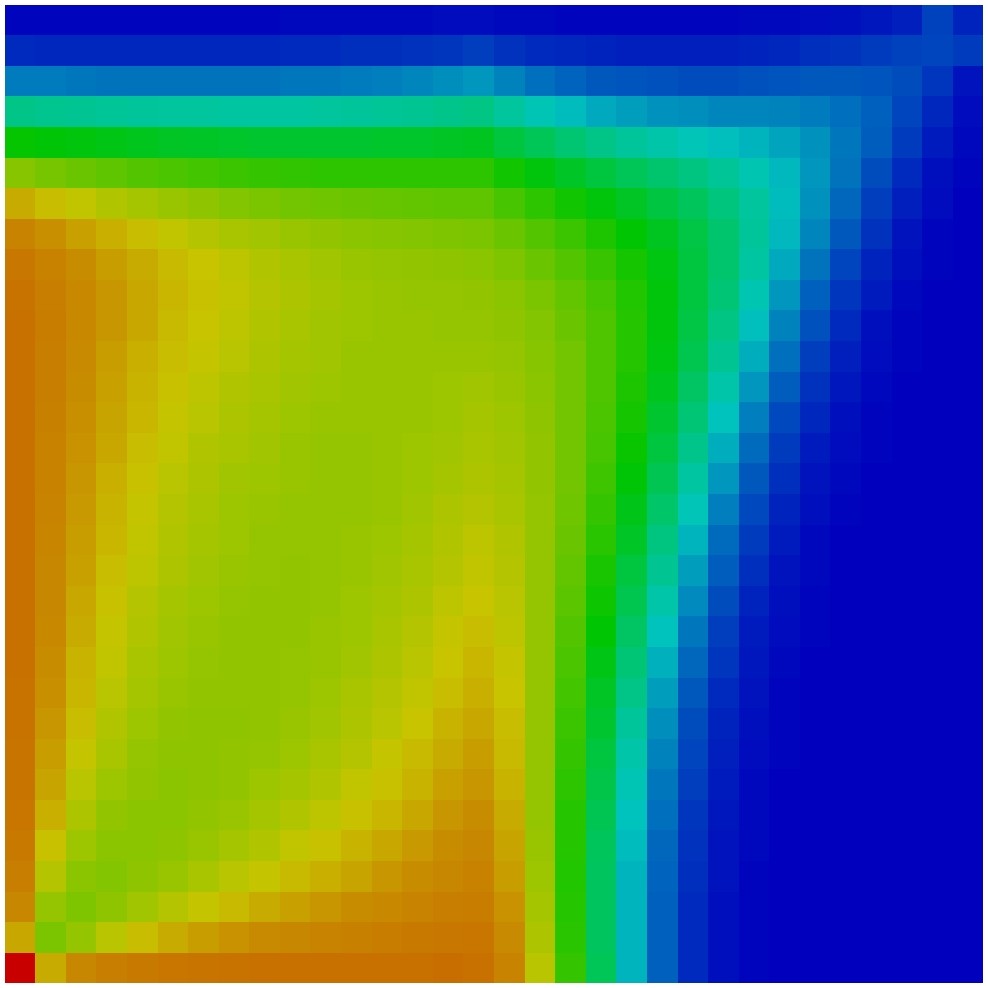}}
\subfloat[$\text{CG}(\text{SD},\vartheta),\ t=20$.]{
\includegraphics[height=0.25\textwidth]{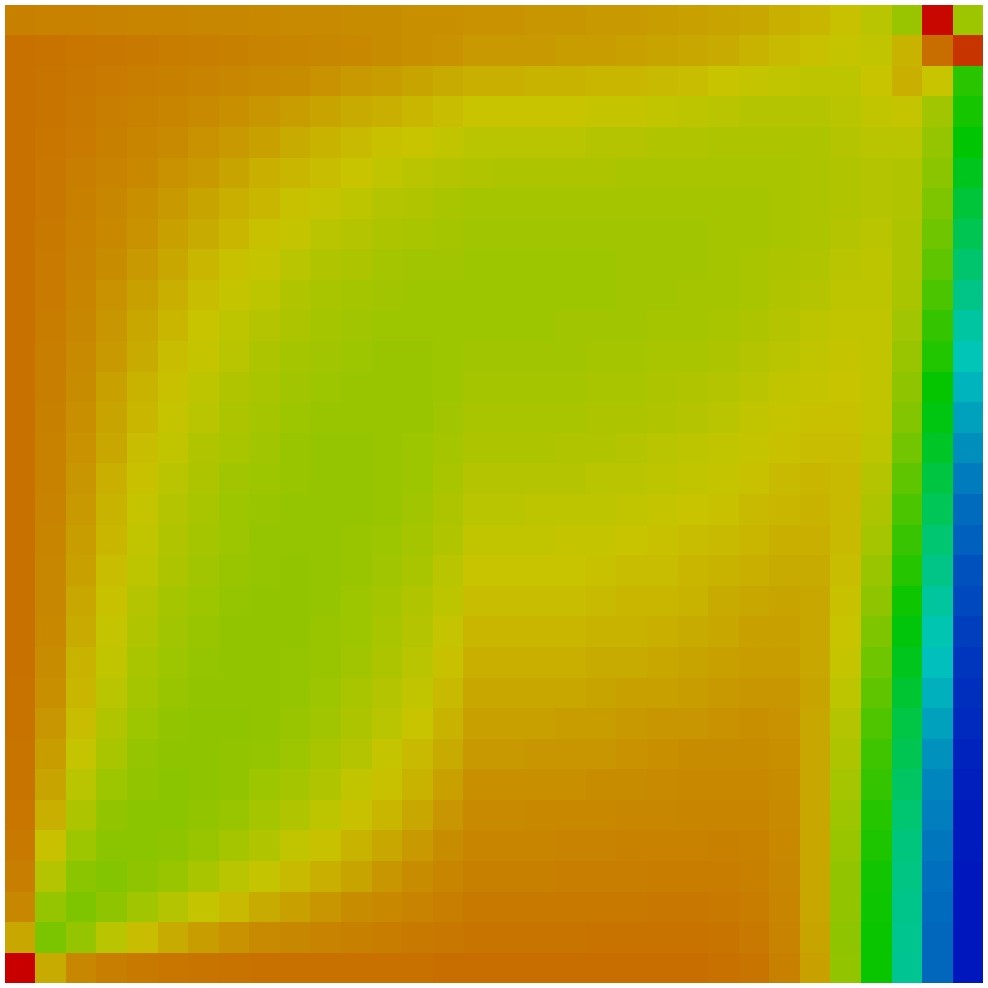}}
\subfloat{
\includegraphics[height=0.25\textwidth]{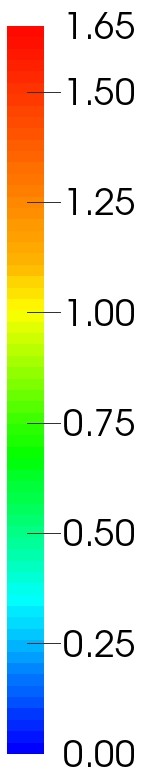}} \\
\addtocounter{subfigure}{-1}
\subfloat[$\text{PP}(\text{SD},\vartheta,\text{wL2}),\ t=1$.]{
\includegraphics[height=0.25\textwidth]{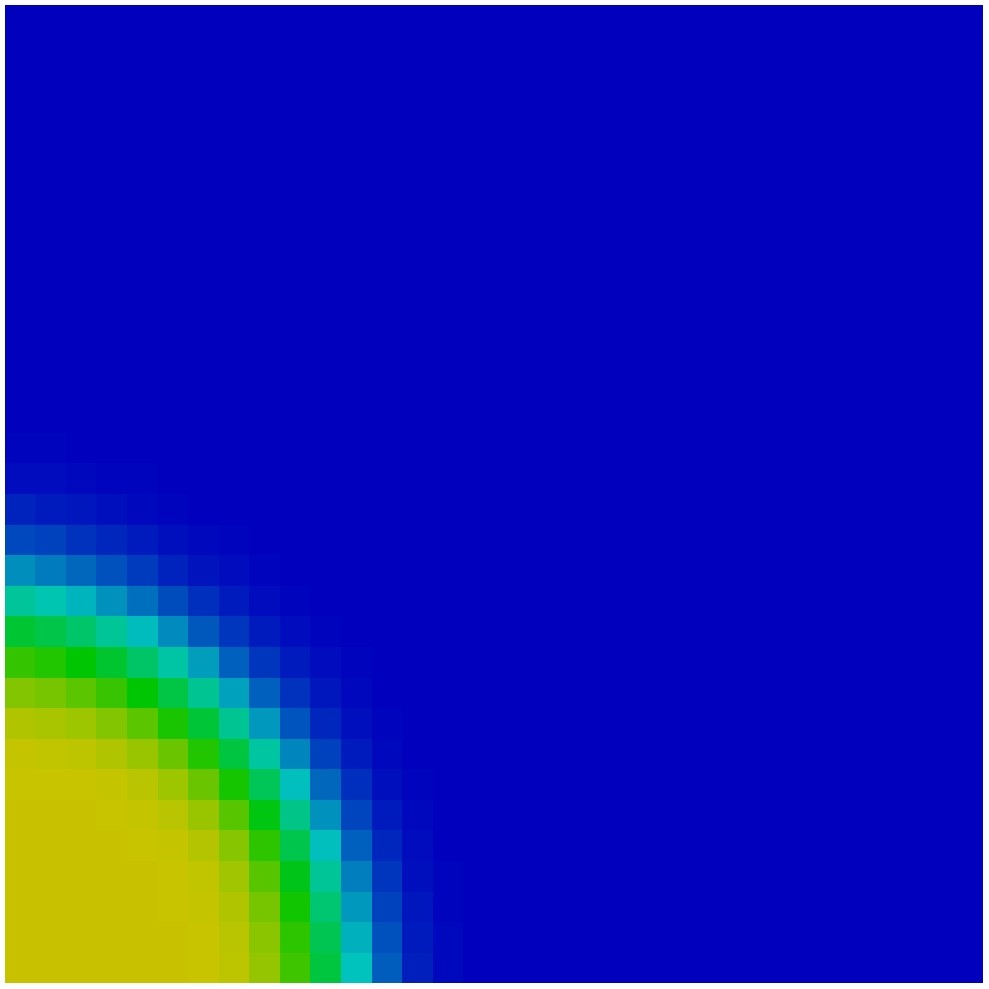}}
\subfloat[$\text{PP}(\text{SD},\vartheta,\text{wL2}),\ t=6$.]{
\includegraphics[height=0.25\textwidth]{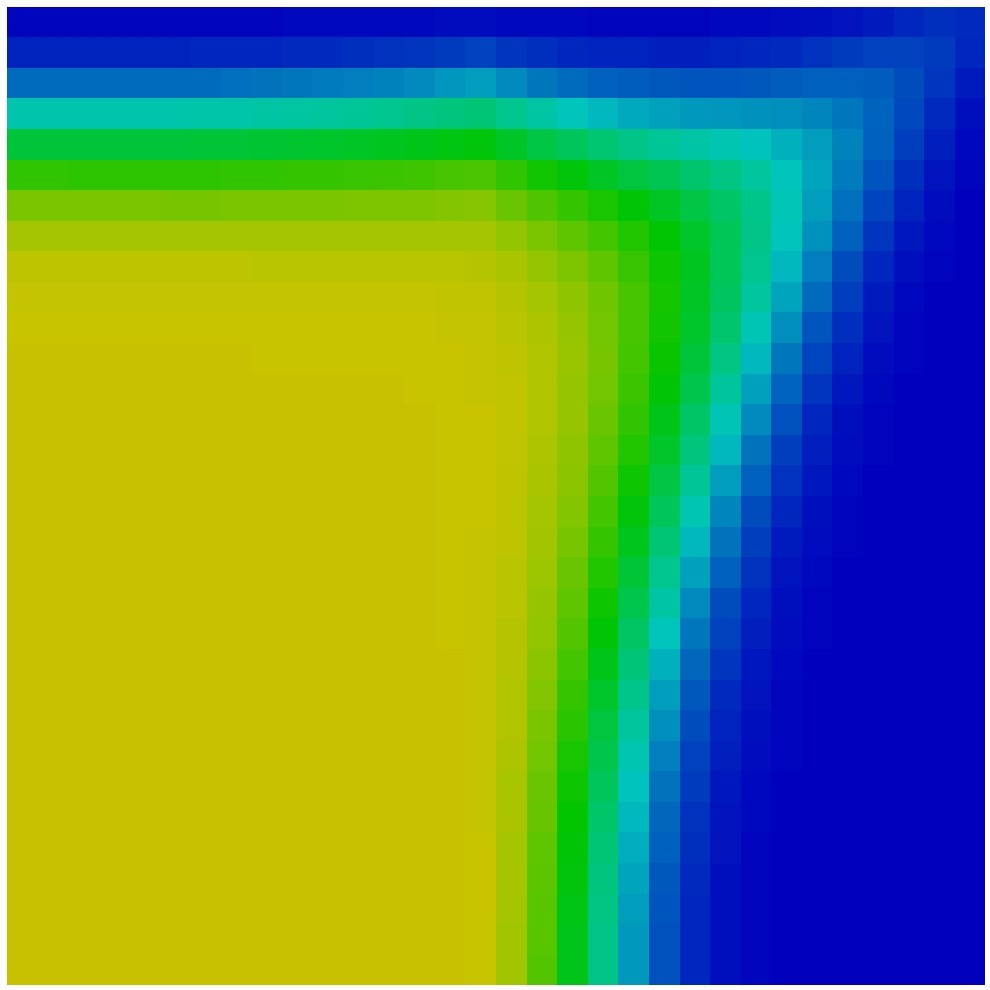}}
\subfloat[$\text{PP}(\text{SD},\vartheta,\text{wL2}),\ t=20$.]{
\includegraphics[height=0.25\textwidth]{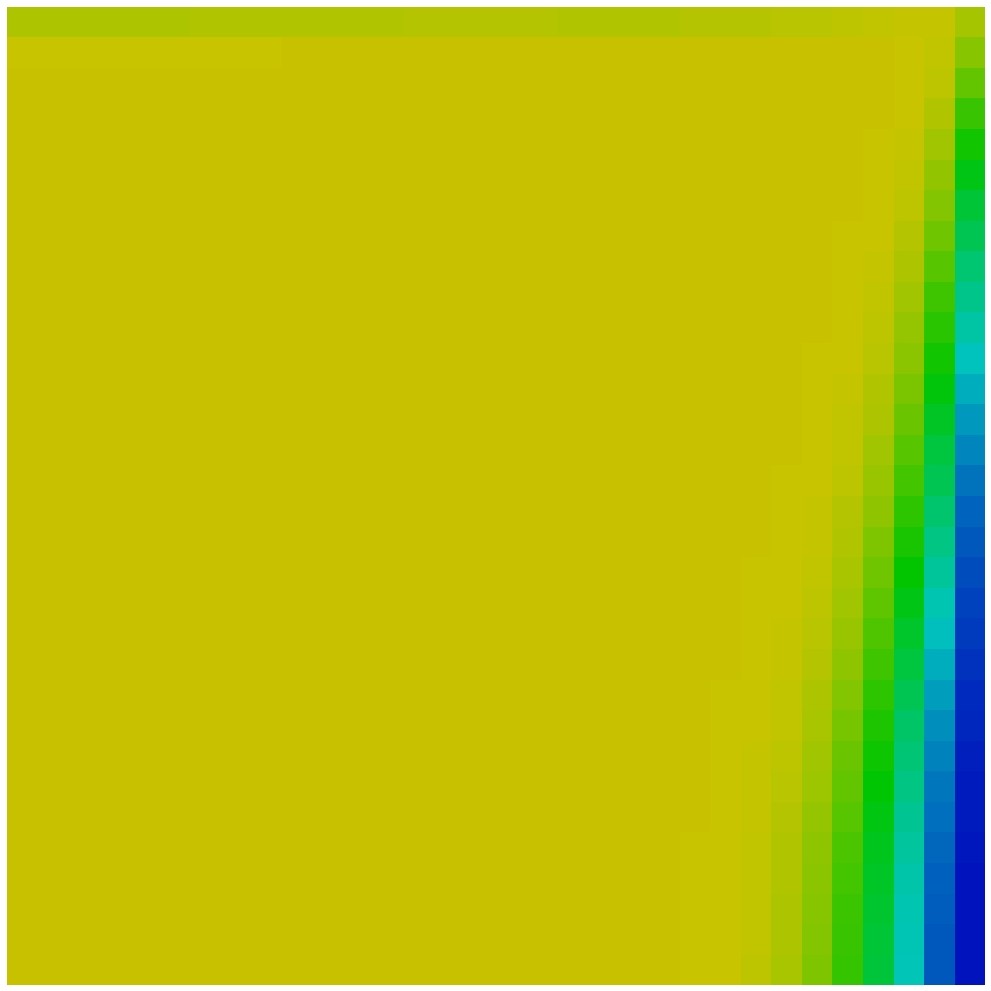}}
\subfloat{
\includegraphics[height=0.25\textwidth]{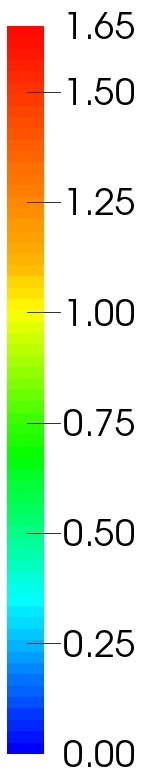}}
\caption{\textbf{Well pair problem}. Concentration solution with CG flux (top row), and postprosessed flux (bottom row) at different times (left to right) on a quadratic grid with $h=1/32$.}
\label{fig:well_conc_time}
\end{figure}

\begin{figure}[bpt]
\centering
\subfloat[CG(SD,$\vartheta$), $h=1/16$.]{
\includegraphics[height=0.25\textwidth]{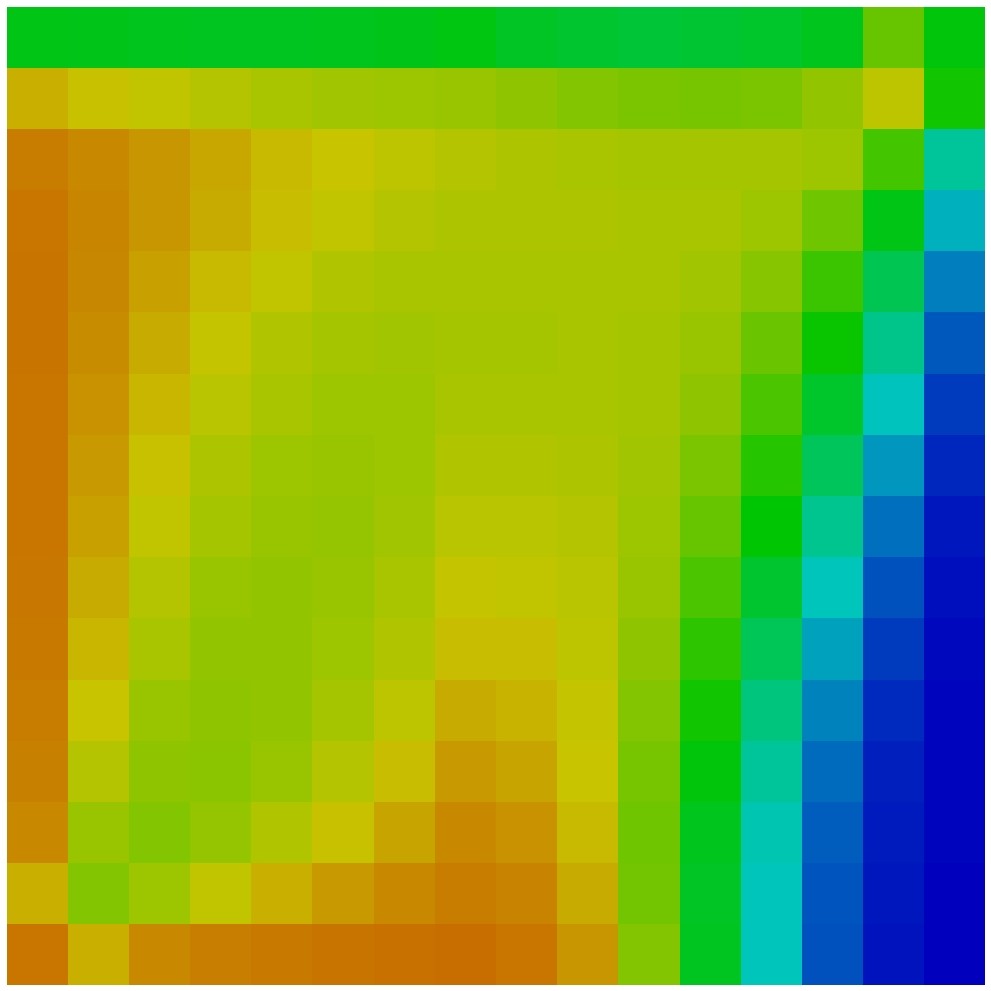}}
\subfloat[CG(SD,$\vartheta$), $h=1/32$.]{
\includegraphics[height=0.25\textwidth]{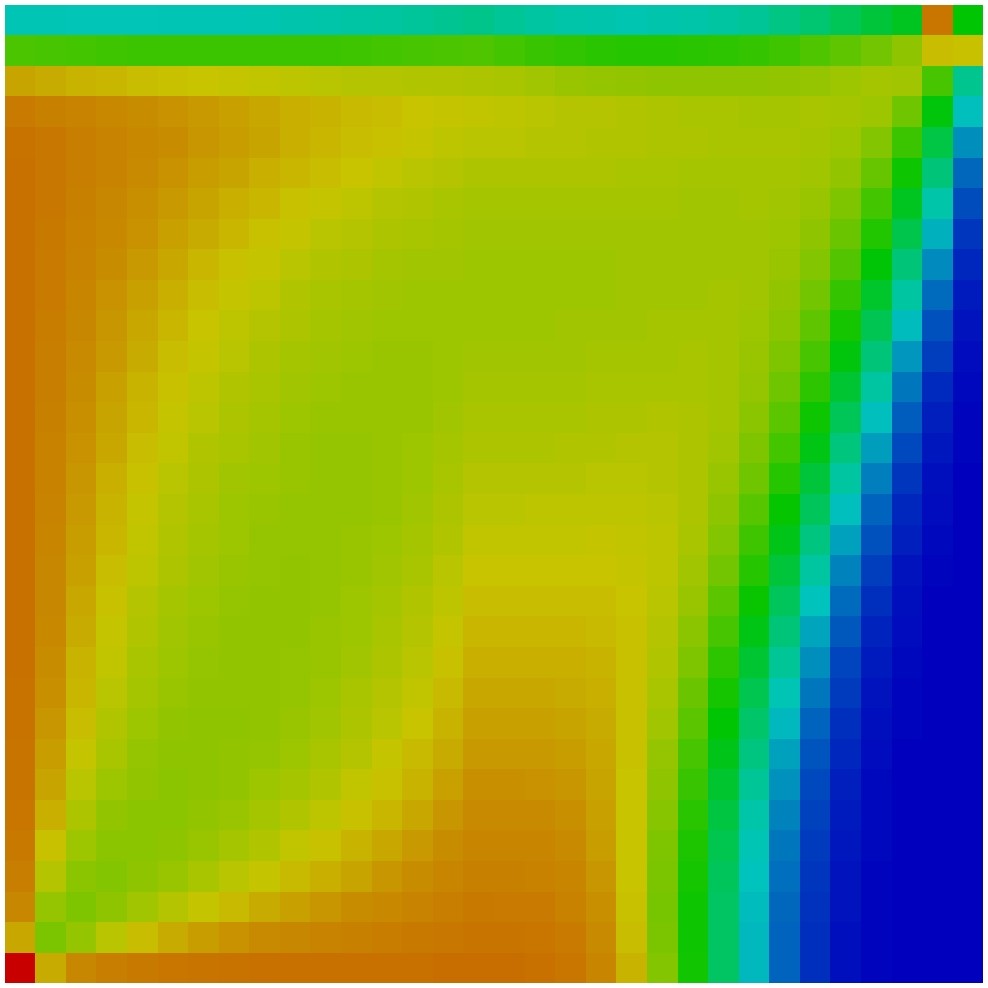}}
\subfloat[CG(SD,$\vartheta$), $h=1/64$.]{
\includegraphics[height=0.25\textwidth]{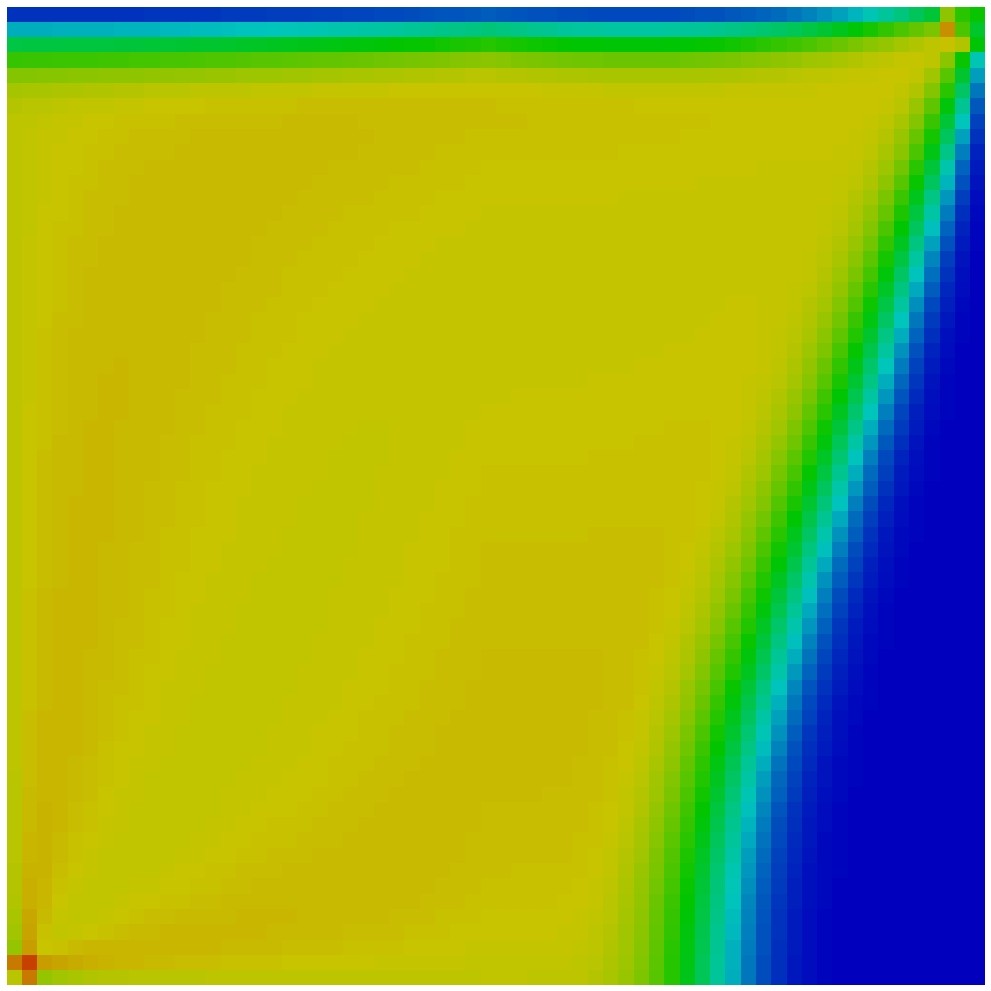}}
\subfloat{
\includegraphics[height=0.25\textwidth]{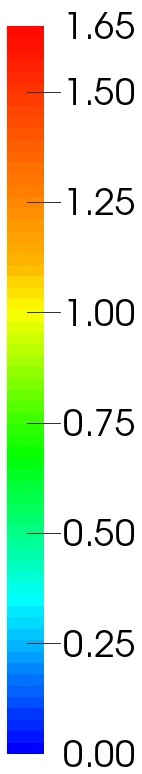}} \\
\addtocounter{subfigure}{-1}
\subfloat[PP(SD,$\vartheta$,wL2), $h=1/16$.]{
\includegraphics[height=0.25\textwidth]{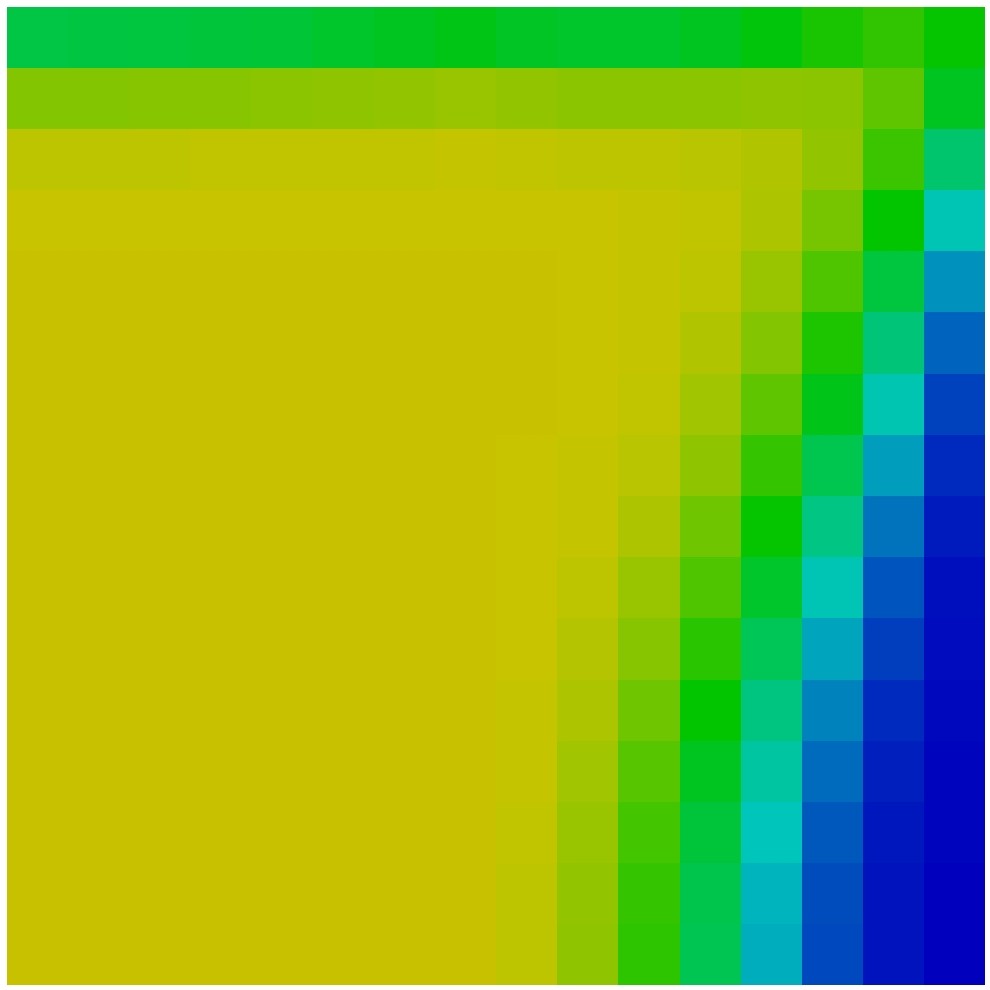}}
\subfloat[PP(SD,$\vartheta$,wL2), $h=1/32$.]{
\includegraphics[height=0.25\textwidth]{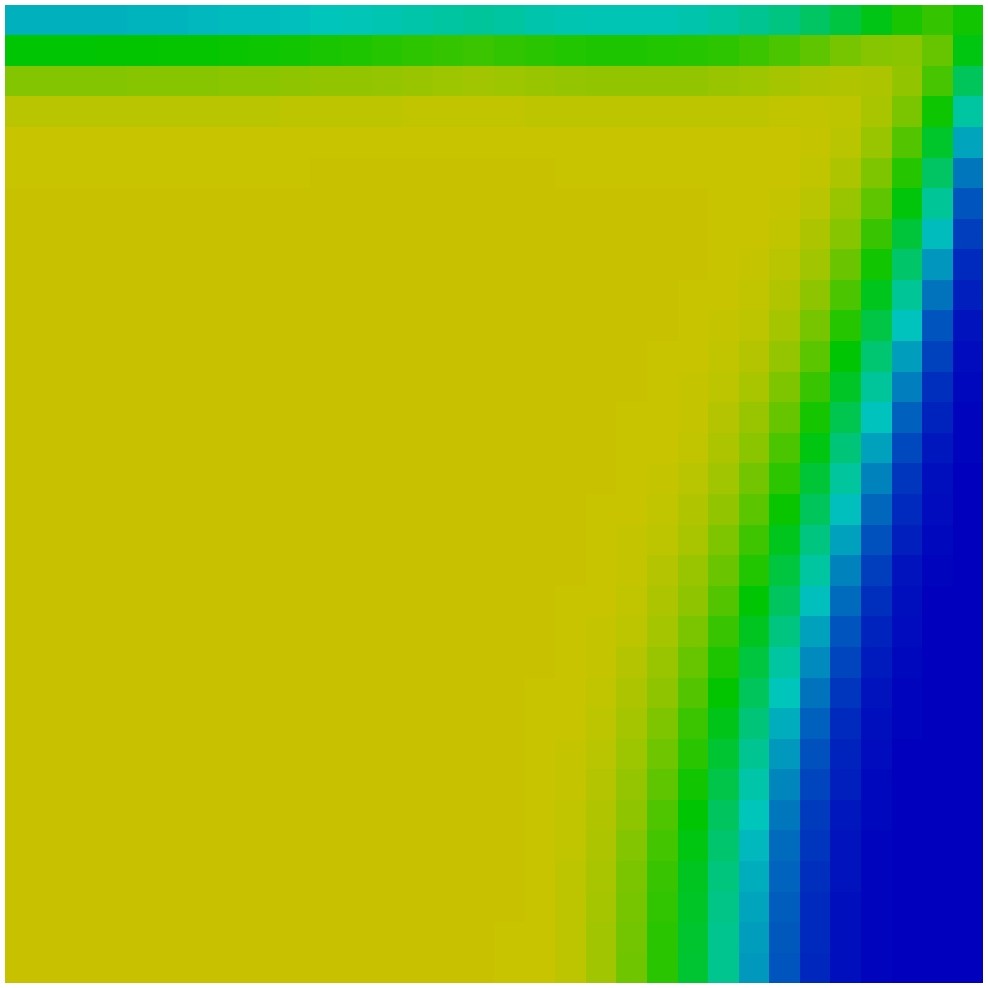}}
\subfloat[PP(SD,$\vartheta$,wL2), $h=1/64$.]{
\includegraphics[height=0.25\textwidth]{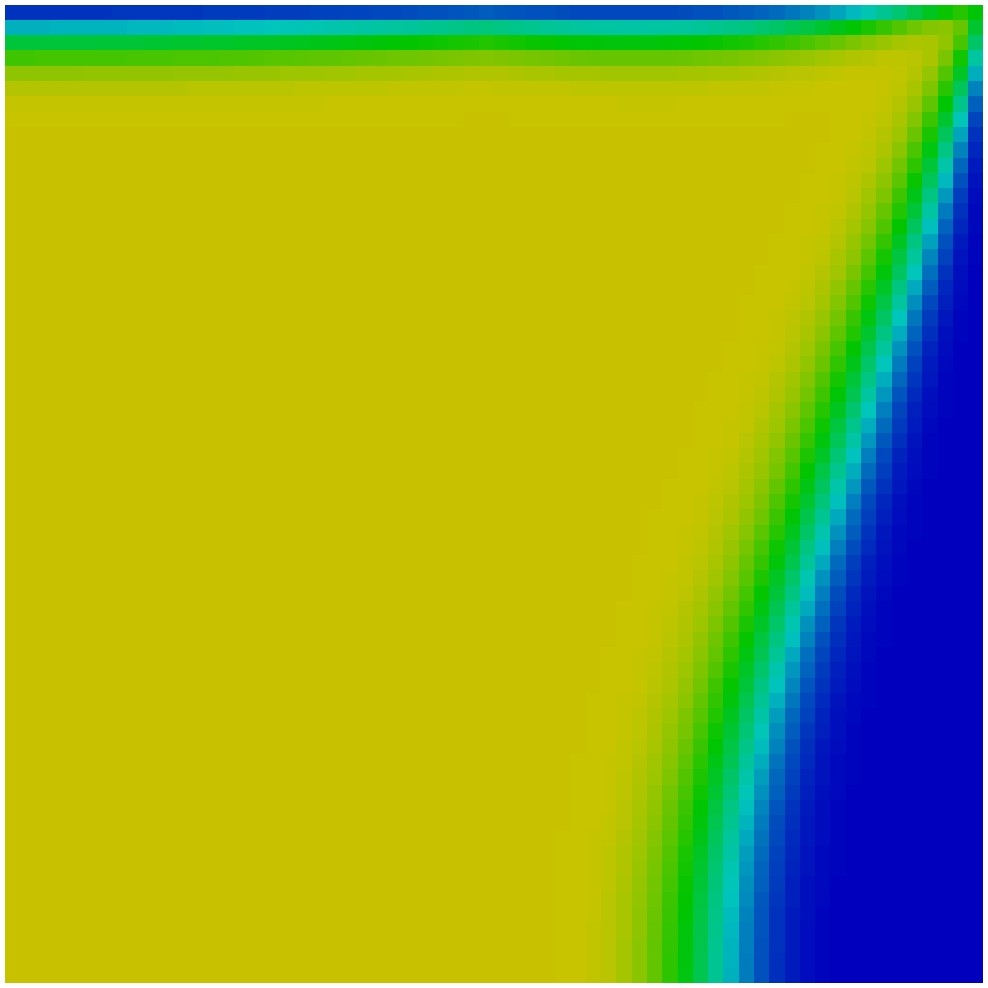}}
\subfloat{
\includegraphics[height=0.25\textwidth]{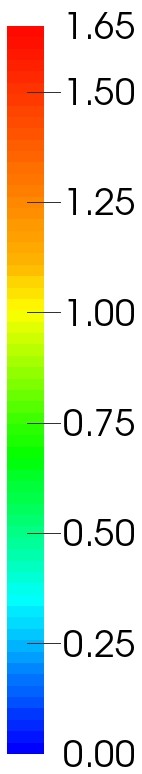}}
\caption{\textbf{Well pair problem}. Concentration solution without (top row) and with (bottom row) postprocessing at $t=10$ on quadratic grids with different $h$.}
\label{fig:well_conc}
\end{figure}

\begin{table}[bpt]
\setlength{\tabcolsep}{5pt}
\caption{\textbf{Well pair problem}. Norm of residual, $\Vert \mathcal{R}(\cdot)\Vert_{\mcEh}$, overshoot, $\mathcal{O}(c_h)$, and minimum and maximum value of concentration solution at $t=10$ for different flux approximations.}
\label{tab:well_overshoot}
\begin{center}
\footnotesize
\begin{tabular}{llrrrrr}
\hline
$h$ &
Method &
$\Vert \mathcal{R}(U_h)\Vert_{\mcEh}$ &
$\Vert \mathcal{R}(V_h)\Vert_{\mcEh}$ &
$\mathcal{O}(c_h)$ &
min($c_h$) &
max($c_h$) \\ 
\hline
1/16 & CG(SD,$\vartheta$)     & 0.3162  & -       & 0.0558  & 0.00508  & 1.217 \\
     & PP(SD,$\vartheta$,wL2) &  -      & 1.6e-16 & 0       & 0.00477  & 1.000 \\
\hline
1/32 & CG(SD,$\vartheta$)     & 2.0928  & -       & 0.0616  & 2.3e-5  & 1.652 \\
     & PP(SD,$\vartheta$,wL2) & -       & 1.4e-15 & 1.4e-17 & 2.0e-5  & 1.000 \\
\hline
1/64 & CG(SD,$\vartheta$)     & 1.5247  & -       & 0.0102  & 3.9e-10 & 1.399 \\
     & PP(SD,$\vartheta$,wL2) & -       & 1.6e-15 & 2.8e-15 & 4.2e-10 & 1.000 \\
\hline
\end{tabular}
\end{center}
\end{table}

A quantity of interest for such well problem is the production rate at the producer,
\begin{align}
\text{PR}(t) &= \frac{1}{\Delta t}\int_{t-\Delta t}^t \int_{\Omega_w} qc,
\end{align}
where $\Omega_w$ is the sink part of $\Omega$, i.e., $\Omega_w=\{\mathbf{x}\in\Omega : q(\mathbf{x})<0\}$. For this example $\Omega_w =\left[\frac{31}{32},1\right]^2$. The production rate is plotted against time for different $h$ in Fig.~\ref{fig:well_production}, where a reference curve from a simulation with $h=1/256$ is included. Although not prominent, we see that we get different curves whether we use CG flux or postprocessed flux, and that this effect is largest for the coarsest grid. We get a earlier breakthrough (smallest $t$ where $\text{PR}(t)>0$) for larger $h$. This is due to numerical dispersion.

\begin{figure}[bpt]
\centering
\includegraphics[width=0.8\textwidth]{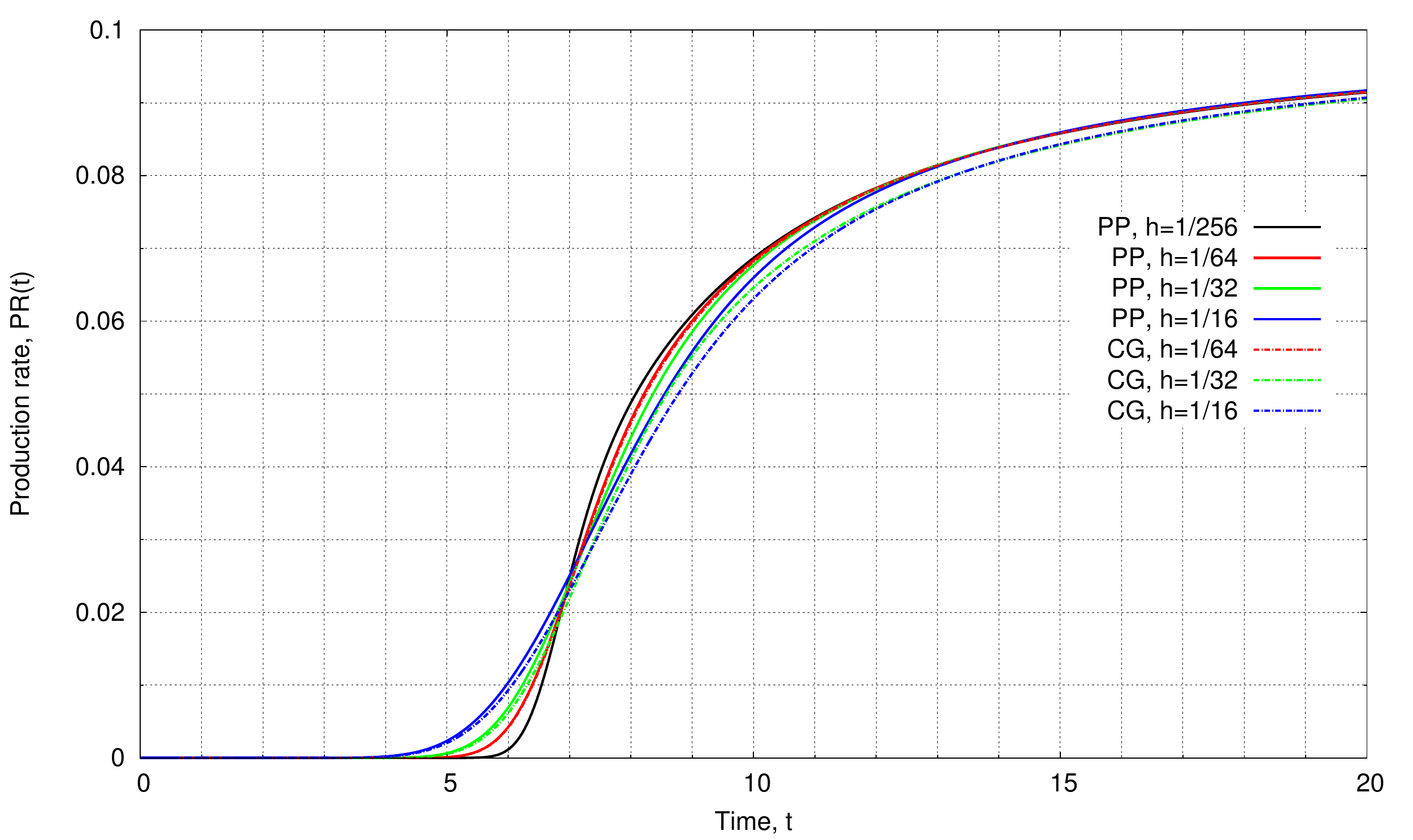}
\caption{\textbf{Well pair problem}. Production rate, $\text{PR}(t)$, for different $h$ and flux.}
\label{fig:well_production}
\end{figure}

\subsection{SPE-10 Model}

Our last example is based on the SPE-10 model \cite{christie2001tenth}, and serves as a test problem to verify objective (i), (v), (vi) and (vii) for a realistic 3D model. The SPE-10 model was originally introduced as a benchmark problem for upscaling, but it has also been used in many studies addressing other aspects of flow in porous media. We consider the top 35 layers of the original model, representing the Tarbert formation, see Fig.~\ref{fig:spe10}. This model is given on a Cartesian mesh with 462000 regular hexahedral elements. The permeability is cellwise constant and anisotropic such that the permeability tensor can be written as a diagonal tensor with entries $k_x,k_y,k_z$ ($k_x=k_y$). Observe from Fig.~\ref{fig:spe10} that the model is highly heterogeneous. To work with realistic data, we will set the fluid viscosity to $\mu=10^{-3}\,\text{Pa}\cdot\text{s}$, in contrast to the rest of this work.
We consider incompressible flow with no source ($\beta=0, q=0$). As boundary conditions, we set $p=10^9\,$Pa on the left boundary, $p=0$ and the right boundary, and no-flow conditions ($\bfu\cdot\bfn=0$) elsewhere. 
Regarding linear solver, we use the preconditioned conjugate gradient method with a general algebraic multigrid preconditioner (AMG) available through the Trilinos Project \cite{trilinos}.

Table \ref{tab:spe10_timing} report on the degrees of freedom (DoF), number of iterations (it), the CPU time used by the linear solver (time) and the norm of the residual, both for the CG problem and the postprocessing problem with and without the weighted norm. First observe that the residual is non-zero for the CG flux, and zero (below solver tolerance) for the postprocessed fluxes. Hence, our methods and implementations work also for this realistic 3D problem.
Furthermore, we see that the computational complexity of PP(SD,$\vartheta$,L2) is lower than PP(SD,$\vartheta$,wL2). This means that minimization in the weighted norm leads to worse conditioning of the system matrix. The time spent to solve PP(SD,$\vartheta$,L2) and PP(SD,$\vartheta$,wL2) compared to CG(SD,$\vartheta$) is about 9\% and 30\%, respectively.

To check the influence of the anisotropic permeability on the linear solver time, we run the same case but with isotropic permeability such that $k_z=k_x(=k_y)$. For this scenario the CPU time used by the linear solver was 20.34, 3.20 and 3.50 for CG(SD,$\vartheta$), PP(SD,$\vartheta$,L2) and PP(SD,$\vartheta$,wL2), respectively. Comparing with the anisotropic case (Table \ref{tab:spe10_timing}), we observe that anisotropic permeability leads to worse conditioning for CG(SD,$\vartheta$) and PP(SD,$\vartheta$,wL2). The run time for PP(SD,$\vartheta$,L2) is unchanged since the system matrix is independent on the permeability. With isotropic permeability, the linear solver time for PP(SD,$\vartheta$,wL2) is about 17\% of that of CG(SD,$\vartheta$).

\begin{figure}[bpt]
\centering
\subfloat[Porosity.]{
\includegraphics[width=0.7\textwidth]{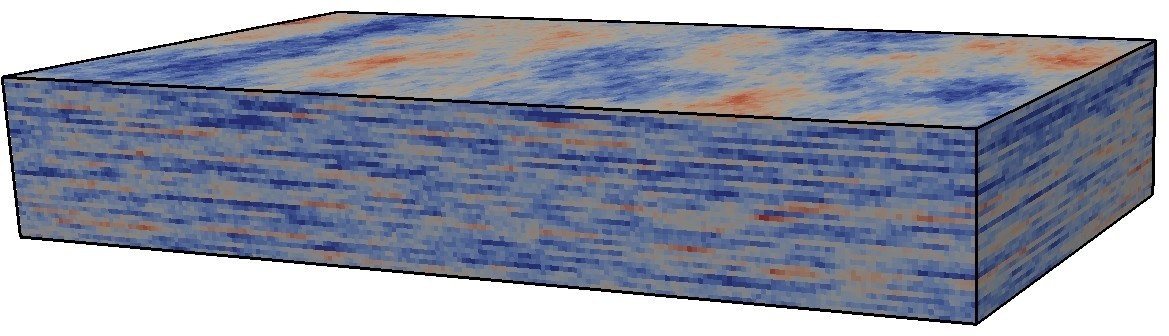}
\includegraphics[height=0.15\textwidth]{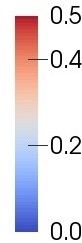}
} \\
\subfloat[Horizontal permeability ($k_x=k_y$) in milli Darcy ($1\,$mD = $9.87\cdot 10^{-16}\,\text{m}^2$) on a logarithmic scale.]{
\includegraphics[width=0.7\textwidth]{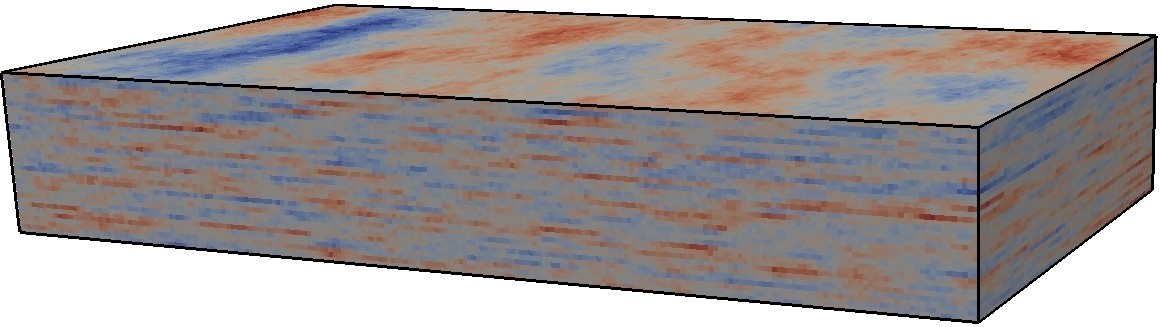}
\hspace*{1mm}
\includegraphics[height=0.16\textwidth]{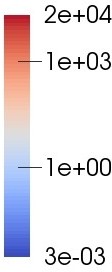}
}
\caption{\textbf{SPE-10 model}. Highly heterogeneous model given on a Cartesian mesh with $220\times60\times85=462000$ regular hexahedral elements, each of size $10\times20\times2$ feet. The model dimensions are $2200\times1200\times170$ feet (these figures are scaled by a factor 5 in the vertical direction). }
\label{fig:spe10}
\end{figure}

\begin{table}[bpt]
\caption{\textbf{SPE-10 model}. Computational complexity for different problems; DoF: Degrees of Freedom, it: number of iterations in linear solver, time: CPU time used by the linear solver including initialization of the preconditioner (median value over 11 runs). The linear solver is the conjugate gradient method with an AMG preconditioner with residual tolerance $10^{-6}$.}
\label{tab:spe10_timing}
\centering
\footnotesize
\begin{tabular}{lrrrr}
\hline
Problem                & DoF    & it     & time    & $\Vert \mathcal{R}\Vert_{\mcEh}$ \\
\hline
CG(SD,$\vartheta$)     & 485316 & 105    & 33.58   & 2.5e-2   \\
PP(SD,$\vartheta$,L2)  & 462000 & 10     & 3.14    & 2.0e-8   \\ 
PP(SD,$\vartheta$,wL2) & 462000 & 55     & 9.97    & 4.3e-8   \\ \hline
\end{tabular}
\end{table}

For the anisotropic case, we also consider the transport problem. We let $c_B=1.0$ on the inflow boundary ($x=0$) and use time steps $\Delta t = 10^4\,\text{s}$. The concentration solutions with PP(SD,$\vartheta$,L2) and PP(SD,$\vartheta$,wL2) are shown in Fig.~\ref{fig:spe10_conc_PP-L2} and \ref{fig:spe10_conc_PP-wL2}, respectively. Both solutions obey the maximum principle, but we see that without weighting (Fig.~\ref{fig:spe10_conc_PP-L2}) the vertical flow between layers with high permeabilty contrast is higher. Hence, the application of the weighted norm seems to better preserve low permeable interfaces.
We do not display similar results for CG(SD,$\vartheta$) because we get a totally unphysical solution. Instead, Fig.~\ref{fig:spe10_conc_CG}, shows the time evolution of $\text{max}(c_h)$ and $\mathcal{O}(c_h)$ with CG(SD,$\vartheta$). Clearly, the maximum principle is far from satisfied.

\begin{figure}[bpt]
\centering
\subfloat[Time step 300, $t=3\cdot 10^6\,\text{s} \approx 35\,\text{days}$.]
{\includegraphics[width=0.45\textwidth]{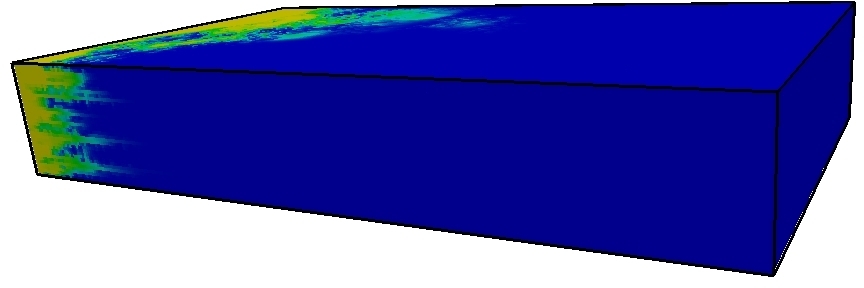}} \hspace{3mm}
\subfloat[Time step 1000, $t=1\cdot 10^7\,\text{s} \approx 116\,\text{days}$.]
{\includegraphics[width=0.45\textwidth]{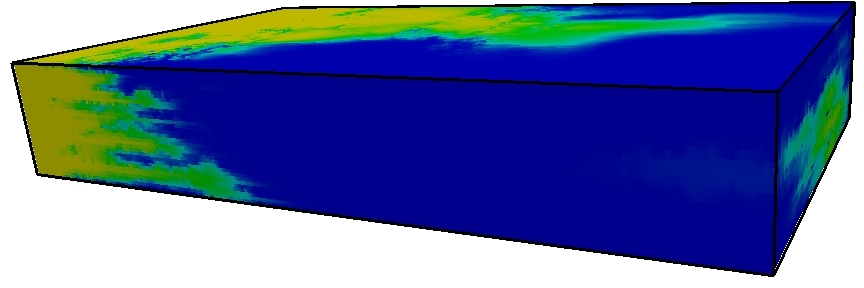}} \hspace{3mm}
\subfloat
{\hspace{0.035\textwidth}} \\
\addtocounter{subfigure}{-1}
\subfloat[Time step 3000, $t=3\cdot 10^7\,\text{s} \approx 347\,\text{days}$.]
{\includegraphics[width=0.45\textwidth]{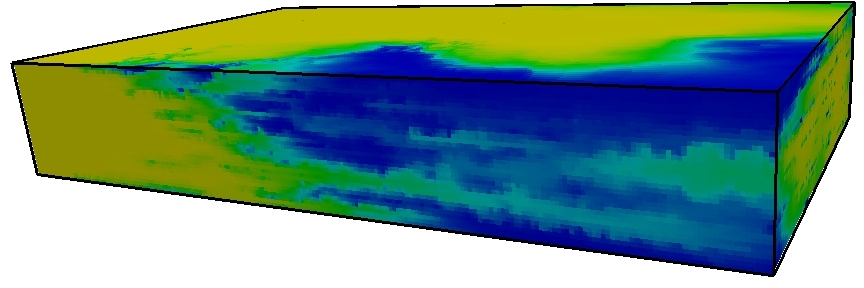}} \hspace{3mm}
\subfloat[Time step 6000, $t=6\cdot 10^7\,\text{s} \approx 694\,\text{days}$.]
{\includegraphics[width=0.45\textwidth]{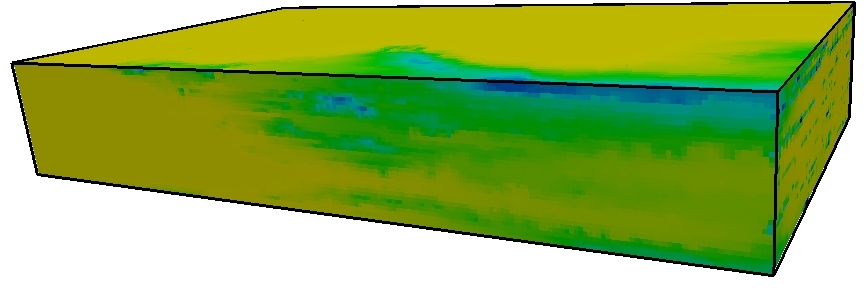}} \hspace{3mm}
\subfloat
{\includegraphics[width=0.04\textwidth]{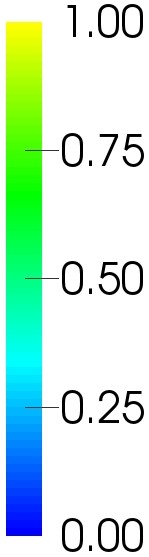}}
\caption{\textbf{SPE-10 model}. Concentration solution with postprocessed flux without weighting, PP(SD,$\vartheta$,L2).}
\label{fig:spe10_conc_PP-L2}
\end{figure}

\begin{figure}[bpt]
	\centering
	\subfloat[Time step 300, $t=3\cdot 10^6\,\text{s} \approx 35\,\text{days}$.]
	{\includegraphics[width=0.45\textwidth]{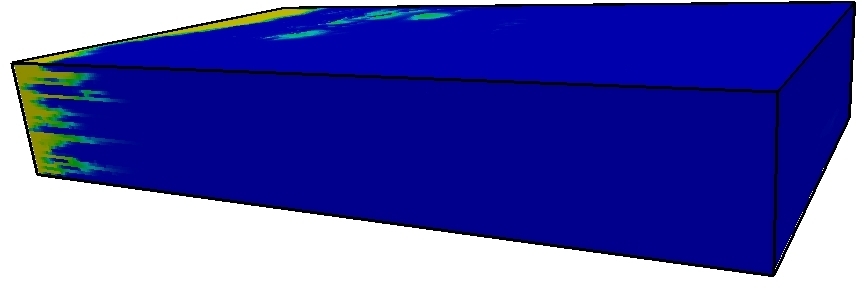}} \hspace{3mm}
	\subfloat[Time step 1000, $t=1\cdot 10^7\,\text{s} \approx 116\,\text{days}$.]
	{\includegraphics[width=0.45\textwidth]{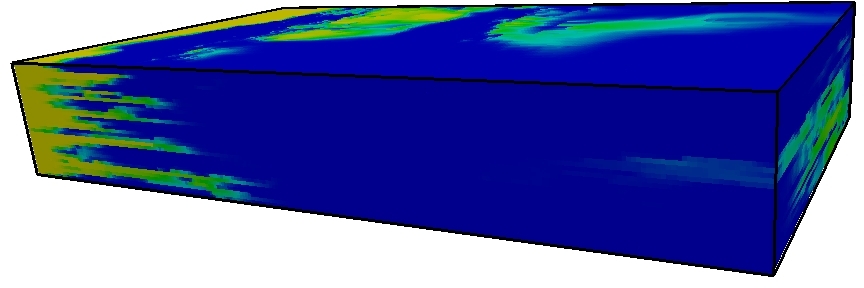}} \hspace{3mm}
	\subfloat
	{\hspace{0.035\textwidth}} \\
	\addtocounter{subfigure}{-1}
	\subfloat[Time step 3000, $t=3\cdot 10^7\,\text{s} \approx 347\,\text{days}$.]
	{\includegraphics[width=0.45\textwidth]{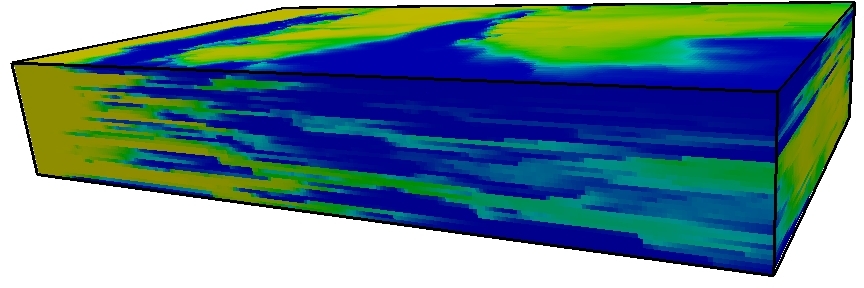}} \hspace{3mm}
	\subfloat[Time step 6000, $t=6\cdot 10^7\,\text{s} \approx 694\,\text{days}$.]
	{\includegraphics[width=0.45\textwidth]{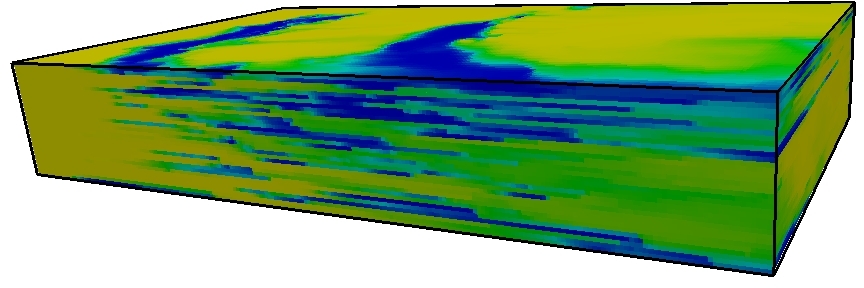}} \hspace{3mm}
	\subfloat
	{\includegraphics[width=0.04\textwidth]{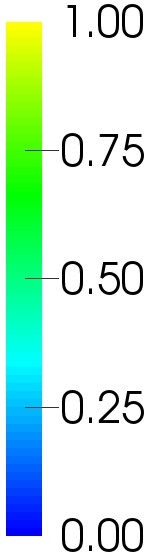}}
	\caption{\textbf{SPE-10 model}. Concentration solution with postprocessed flux with weighting, PP(SD,$\vartheta$,wL2).}
	\label{fig:spe10_conc_PP-wL2}
\end{figure}

\begin{figure}[bpt]
\centering
\includegraphics[width=0.8\textwidth]{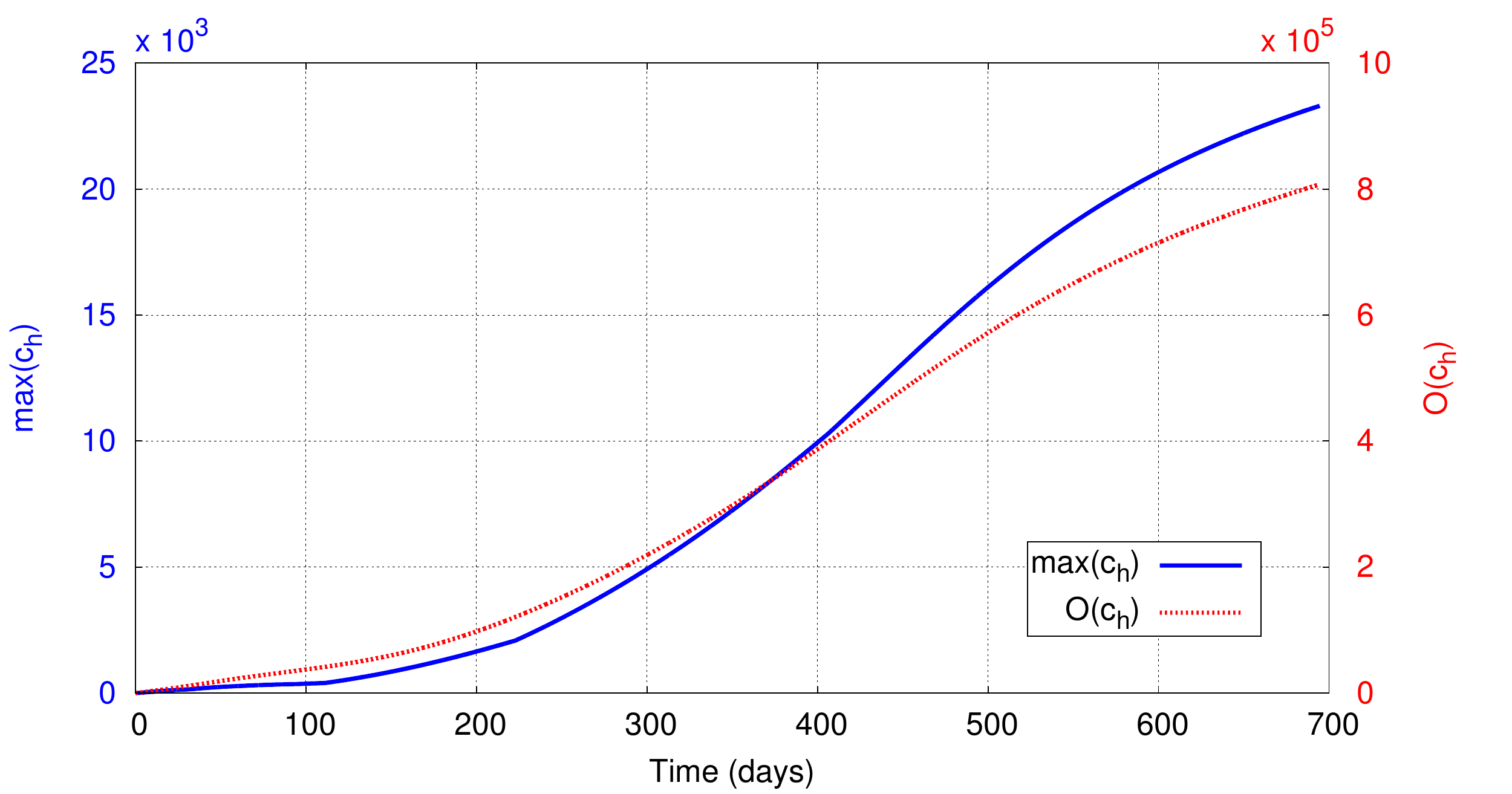}
\caption{\textbf{SPE-10 model}. Maximal concentration, $\text{max}(c_h)$, and overshoot, $\mathcal{O}(c_h)$, for concentration solution with CG(SD,$\vartheta$). For the reference, we have $\max(c_h) = 1.00006$ and $\mathcal{O}(c_h)=0.019$ at $t=694\,\text{days}$ with PP(SD, $\vartheta$,wL2).}
\label{fig:spe10_conc_CG}
\end{figure}

\section{Conclusions}
\label{sec:conclusion}

Eq.~\eqref{eq:pp_def}, p.\ \pageref{eq:pp_def}, defines a general purpose postprocessing method, where a minimal piecewise constant correction term is added to the flux. Local conservation, uniqueness and preservation of convergence order is proven and summarized in Theorem \ref{theorem}, p.\ \pageref{theorem}. Our method applies to any flux approximation in $L^1(\mcFh)$ and for a wide range of grids, including non-conforming and unstructured grids. It can also be used for the time dependent flow model. 

Through a series of numerical examples, we have demonstrated that our method produces locally conservative flux. It is verified numerically that the postprocessed flux has the same order of convergence as the original flux. Moreover, our numerical examples clearly demonstrates the importance of locally conservative flux when coupling with a DG solver for the transport equation. Lack of local conservation may produce unphysical solutions.

The postprocessing algorithm is global in the sense that a system of $N$ linear equations has to be solved, where $N$ is the number of elements (or cells). However, the system matrix is symmetric and sparse and only dependent on the permeability (through the weights) and the grid. If the grid is constant or only altered occasionally, we can allow for a preconditioner that is relatively costly to initialize. 

For flux approximations from CG, where the pressure gradient is discontinuous across element faces, it is favorable to use harmonic averaging to calculate the flux. A novelty of this work compared to \cite{larson2004conservative} and \cite{sun2006projections} is that we minimize the correction term in a weighted $L^2$ norm with weights equal to the inverse of the effective face permeability. This better preserves low permeable interfaces, and numerical examples  demonstrate that no weighting (standard $L^2$ norm) tends to weaken the effect of harmonic averaging. 

The computational complexity of solving the linear system associated with the postprocessing step compared to that of solving the linear system for the CG problem was measured. For the synthetic 2D barrier problem, the additional cost was significant ($\sim 60\%$). However, for the larger 3D SPE-10 model, the additional cost was smaller, $10$--$30\%$, depending on anisotropy and choice of weights. This indicates that the postprocessing method is reasonable also in terms of computational efficiency. The difference in computational complexity of applying the weighted norm was small for isotropic permeability as long as an appropriate preconditioner, such as SSOR or AMG, was used. For anisotropic permeability the difference was larger. We stress that in this work we only considered general purpose preconditioners. Using a taylored preconditioner that can handle the weights better might further improve the efficiency. 

Different treatment of fluxes on Dirichlet boundaries for non-Cartesian grids showed only little effect on the postprocessed flux.

\section*{Acknowledgements}
LHO thanks the Center for Subsurface Modeling at ICES, UT Austin, for hosting his research stay the first half of 2015. 
In particular, thanks to Gergina Pencheva, Sanghyun Lee and Prashant Mital for constructive discussions of the current work.
LHO is funded by VISTA --- a basic research program funded by Statoil, conducted in close collaboration with The Norwegian Academy of Science and Letters. 
MGL was supported in part by the Swedish Foundation for Strategic Research Grant No.\ AM13-0029 (MGL) and the Swedish Research Council Grant No 2013-4708.


\bibliographystyle{acm}
\bibliography{references}

\end{document}